\documentclass[12pt]{article}
\usepackage[a4paper,width=150mm,top=25mm,bottom=25mm]{geometry}
\usepackage{graphicx} 
\RequirePackage{amsthm,amsmath,amsfonts,amssymb}
\RequirePackage[numbers]{natbib}
\RequirePackage[colorlinks,citecolor=blue,urlcolor=blue]{hyperref}
\usepackage[utf8]{inputenc}
\usepackage{mathtools} 
\usepackage{booktabs} 
\usepackage{tikz, pgfplots} 
\usepackage{tikz-cd}
\usepackage{amsmath}
\usepackage{amssymb}
\usepackage{mathtools}
\usepackage{graphicx}
\usepackage{xcolor}
\usepackage{caption}
\usepackage{subcaption}
\usepackage{verbatim}
\usepackage{titlesec}
\usepackage{authblk}
\usepackage{makecell}
\usepackage{bold-extra}
\usetikzlibrary{positioning, arrows.meta, quotes}
\usetikzlibrary{shapes,snakes}
\usetikzlibrary{bayesnet}
\usetikzlibrary{shapes,decorations,arrows,calc,arrows.meta,fit,positioning}
\tikzset{>=latex}
\tikzstyle{plate caption} = [caption, node distance=0, inner sep=0pt,
below left=5pt and 0pt of #1.south]

\newcommand{\X}{\mathbb{X}}

\newcommand{\Z}{\mathbb{Z}}
\newcommand{\C}{\mathbf{C}}
\newcommand{\Cplx}{\mathbb{C}}
\newcommand{\SVAR}{\mathrm{SVAR}_\mathrm{stb}^\infty}
\newcommand{\SVARs}{\mathrm{SVAR}_{<1}}
\newcommand{\R}{\mathbb{R}}
\newcommand{\reg}{\mathrm{reg}}
\newcommand{\Sp}[1]{\mathcal{S}_{#1}}

\newcommand{\T}[1]{\mathcal{T}_{#1}}

\newcommand{\causal}[2]{\mathbf{C}_{#1 \to #2}}

\newcommand{\spcausal}[2]{\mathcal{S}_{#1 \to #2}}
\newcommand{\spconf}[2]{\mathcal{S}_{#1 \leftrightarrow #2}}
\newcommand{\spresidual}[2]{\mathcal{S}_{#2 \setminus #1}}

\newcommand{\internal}{\mathrm{I}}
\newcommand{\linternal}{\mathrm{LI}}
\newcommand{\Pa}{\mathrm{Pa}}
\newcommand{\PD}{\mathrm{PD}}
\newcommand{\Rfnc}{\mathfrak{R}_{S^1}}
\newcommand{\trekL}{\mathrm{Left}}
\newcommand{\trekR}{\mathrm{Right}}

\newcommand{\diag}{\mathrm{diag}}
\newcommand{\ttop}{\mathrm{top}}

\newcommand{\tcvl}{\hat{\ast}}

\newcommand{\cetemp}[3]{\Lambda_{#1 \to #2 | \mathrm{do}(#3) = \mathbf{0}}}
\newcommand{\ucetemp}[2]{\Lambda_{#1 \to #2}}

\newcommand{\cefrq}[3]{\mathfrak{H}_{#1 \to #2 | #3}}
\newcommand{\ucefrq}[2]{\mathfrak{H}_{#1 \to #2}}

\newcommand{\dirtmp}[2]{\Lambda_{#1, #2}}
\newcommand{\dirtmpabs}[2]{|\Lambda|_{#1, #2}}

\newcommand{\ltmp}[2]{\Gamma_{#1, #2}}
\newcommand{\lfrq}[2]{\mathfrak{J}_{#1, #2}}
\newcommand{\Dirtmp}{\Lambda}
\newcommand{\Ltmp}{\Gamma}
\newcommand{\dirfrq}[2]{\mathfrak{H}_{#1, #2}}
\newcommand{\dirfrqabs}[2]{|\mathfrak{H}|_{#1, #2}}
\newcommand{\Dirfrq}{\mathfrak{H}}
\newcommand{\dirwnd}[2]{\overline{\mathbf{A}}_{#1 \to #2}}

\newcommand{\circfrq}[2]{\mathfrak{H}_{#1 \rightleftarrows #2}}

\newcommand{\Do}{\mathrm{do}}
\newcommand{\Ep}{\mathbb{E}}

\theoremstyle{plain}
\newtheorem{theorem}{Theorem}
\newtheorem{proposition}{Proposition}
\newtheorem{lemma}{Lemma}
\newtheorem{corollary}{Corollary}

 \theoremstyle{definition}
 \newtheorem{definition}{Definition}
 \newtheorem{example}{Example}
 \newtheorem{remark}{Remark}

\titleformat{\title}{display}{}{}{}[]
\titleformat{\section}[runin]{\normalsize\bfseries}{\thesection.}{0.5em}{}[.]
\titleformat{\subsection}[runin]{\normalsize\itshape}{\thesubsection.}{0.3em}{}[.]

\providecommand{\keywords}[1]
{
  \small	
  \textbf{\textit{Keywords---}} #1
}

\title{\large\textbf{CAUSAL INFERENCE ON PROCESS GRAPHS \\PART I} \\ \normalsize{\textbf{THE STRUCTURAL EQUATION PROCESS REPRESENTATION}}}

\author[1]{\normalsize Nicolas-Domenic Reiter \thanks{nicolas-domenic.reiter@dlr.de}}
\author[1]{ \normalsize Andreas Gerhardus \thanks{andreas.gerhardus@dlr.de}}
\author[2,1]{ \normalsize Jonas Wahl \thanks{wahl@tu-berlin.de}}
\author[1,2,3]{\normalsize Jakob Runge \thanks{jakob.runge@tu-dresden.de}}
\affil[1]{\normalsize German Aerospace Center (DLR), Institute of Data Science, Jena, Germany}
\affil[2]{\normalsize Technische Universität Berlin, Berlin, Germany}
\affil[3]{\normalsize Center for Scalable Data Analytics and Artificial Intelligence (ScaDS.AI) Dresden/ Leipzig, TU Dresden, Germany}
\date{}

\begin{document}

\maketitle

\begin{abstract}
    When dealing with time series data, causal inference methods often employ structural vector autoregressive (SVAR) processes to model time-evolving random systems. In this work, we rephrase recursive SVAR processes with possible latent component processes as a linear Structural Causal Model (SCM) of stochastic processes on a simple causal graph, the \emph{process graph}, that models every process as a single node. Using this reformulation, we generalise Wright's well-known path-rule for linear Gaussian SCMs to the newly introduced process SCMs and we express the auto-covariance sequence of an SVAR process by means of a generalised trek-rule. Employing the Fourier-Transformation, we derive compact expressions for causal effects in the frequency domain that allow us to efficiently visualise the causal interactions in a multivariate SVAR process. Finally, we observe that the process graph can be used to formulate graphical criteria for identifying causal effects and to derive algebraic relations with which these frequency domain causal effects can be recovered from the observed spectral density.
\end{abstract}


\section{Introduction}
Causal inference concerns the data-driven investigation of causal relationships, formalised within a framework of statistically and graphically encoded concepts and assumptions. One commonly used causal modelling framework are \textit{Structural Causal Models (SCMs)} \cite{pearl2009causality, peters2017elements, spirtes2000causation}. An SCM has an associated directed graph representing the qualitative cause-effect relations among the variables in the SCM. In this so-called \emph{causal graph}, each vertex represents a unique variable, and a directed link between two vertices symbolises a direct causal influence. Modelling causal relationships by SCMs enables one to rigorously reason about interventions on variables and, thus, allows to quantitatively define the effect of an intervention between sets of variables. These effects are often termed causal effects, and the question whether they can be estimated from the observational distribution given the causal graph, constitutes one of the major research problems in causal inference. Another major task in causal inference, known as \textit{causal discovery} or \textit{causal structure learning}, concerns learning the causal graph from the observational distribution given appropriate enabling assumptions. Linear Gaussian SCMs \cite{wright1934method, bollen1989structural}, also known as \textit{structural equation models} (SEM), provide a simple yet interesting class of example SCMs on which these two types of problems have been studied extensively, see e.g. \cite{drton2018algebraic} for an overview.  

The data with which causal questions are to be answered often carries a time structure. Consequently, the SCM and causal graphical model framework has been extended to time series, see e.g. \cite{eichler2010graphical, dahlhaus2003causality, peters2017elements, runge2023causal}, where the variables are modeled as $\Z$-indexed stochastic processes instead of scalar-valued random variables. 
The causal graph then becomes an infinite graph, referred to as time series graph \cite{runge2019detecting}, time series chain graph \cite{dahlhaus2003causality}, full time graph \cite{peters2017elements} or time series DAG \cite{gerhardus2021characterization}, that contains one node for every variable at every discrete time-point, and a link in the time series graph indicates a direct causal influence together with the time-delay with which the effect occurs. Among the examples of time series SCMs are the widely applied \cite{Seth3293, FRISTON2013172, MemoryMattersACaseforGrangerCausalityinClimateVariabilityStudies, QuantifyingtheStrengthandDelayofClimaticInteractionsTheAmbiguitiesofCrossCorrelationandaNovelMeasureBasedonGraphicalModels, bernanke1995blackbox} \textit{linear (structural) vector autoregressive (SVAR)} processes \cite{lutkepohl2005new, brockwell2009time}, which recursively model the state of the process as a linear combination of the present and past states as well as an independent Gaussian noise term. SVAR processes extend SEMs in that they allow to model time-lagged effects between processes and auto-dependencies within processes. 

The infinite time series graph can be contracted along its temporal axis by forgetting the time-lags of the direct causal influences. The reduced graph is a finite graph over the set of processes and is sometimes called the summary graph \cite{assad2022timeseries, peters2017elements}. To visualise possible auto-dependencies explicitly, the summary graph typically contains self-edges (edges pointing from a node to itself). In this work, we focus on the closely related \emph{process graph} which is simply the summary graph without explicitly drawn self-edges. We emphasize however that this is just a matter of convenience for model representation and graphical computation, and the SVAR processes under investigation are assumed to exhibit autodependence as usual. In general, the process graph is considerably less complex than the time series graph, as there are infinitely many possible time series graphs which collapse into the same process graph.

\paragraph{Advantages of process graph models} Inferring the causal structure of time series models from observational data has been subject to extensive research. This research includes the classical and widely known concept of Granger Causality \cite{granger1969investigating} and its frequency version \cite{geweke1982measurement, geweke1984measures}. More recent approaches, like \cite{runge2019inferring, runge2019detecting, runge2020discovering, gerhardus2020high}, are based on the aforementioned time series SCMs. The latter methods are designed to discover the full time series graph. Methods related to Granger Causality (GC) as well as some more recent approaches such as \cite{manten2024signature}, seek to discover the less complex summary graph, a goal that is slightly more humble and therefore easier to achieve. 

A strength of the SCM-framework is that, once the time series graph has been identified (either by a causal discovery algorithm or domain experts), it is possible to estimate causal effects (if they are identifiable). Causal effects in time series are discussed, for example in \cite{eichler2010granger,  QuantifyingtheStrengthandDelayofClimaticInteractionsTheAmbiguitiesofCrossCorrelationandaNovelMeasureBasedonGraphicalModels} or more recently in \cite{thams2022instrumental, mogensen2022instrumental, runge2023causal, gerhardus2021characterization, gerhardus2023projecting}. However, the precondition  that the possibly very complex time series graph has been fully identified is generally very challenging to establish. In practice, domain experts may not know the exact time lags of process interactions and autodependencies, rendering many of the developed estimation techniques moot, or injecting high degrees of uncertainties into the estimation task. At the same time, experts may know the less complex process graph \cite{QuantifyingCausalPathwaysofTeleconnections, wcd-1-715-2020, UsingCausalEffectNetworkstoAnalyzeDifferentArcticDriversofMidlatitudeWinterCirculation, esd-11-17-2020}. Therefore, the \textbf{main goal} of this work and its direct descendants \cite{reiter2024causal,reiter2024asymptotic} is to develop a framework that allows for
\begin{itemize}
    \item a rigorous formalization and analysis of causal effects between SVAR processes and clear, informative visual representations at the process graph level in the time as well as in the frequency domain (the focus of this work);
    \item a precise formulation of spectral density based process graph discovery, which uses an algebraic characterisation $d$- and $t$-separation \cite{reiter2024causal} in the process graph, generalising previous results for linear Gaussian SEM's \cite{pearl2009causality, spirtes2000causation, sullivant2010trek}. Furthermore, in \cite{reiter2024causal} we present an identifiability theory for causal effects at the process graph level. Based on this theory, we show in \cite{reiter2024causal} that the recent latent-factor half-trek criterion \cite{10.1214/22-AOS2221} (for linear Gaussian SEM's) can be applied to the process graph of an SVAR process to decide the identifiability of causal effects at the level of the process graph.
    \item an examination of the asymptotic distribution of a frequency-domain causal effect estimator to assess the significance of frequency domain causal effects from finite data (the focus of the follow-up work \cite{reiter2024asymptotic}). This is the final ingredient to render the developed framework practically useful. As an illustration of this practical usefulness, \cite{reiter2024asymptotic} also contains an analysis of a real-world example from the Earth sciences, confirming a significant effect of solar activity variations on the Northern Atlantic Oscillation (NAO) at the 10-11 year time scale, the so-called 'solar cycle'. 
\end{itemize}

Regarding the present paper, the main contributions are as follows:
\begin{enumerate}
    \item In Theorem \ref{prop: summary SEM}, we rephrase a given SVAR process with latent processes as an SCM-like model of stochastic processes at the level of its finite process graph. We term this reformulation the \textit{structural equation process (SEP)} representation of the SVAR process. The SEP representation clearly separates any involved process into \emph{internal dynamics}, \emph{dynamics introduced by its parent processes} and \emph{dynamics due to latent confounding} (Section \ref{subsection: SEP representation}).
    \item Based on this reformulation, we generalise the notion of direct and controlled causal effects in linear Gaussian SCMs to direct and controlled causal effects between processes, see Definition \ref{def: direct effect-filter} resp. Definition \ref{def: ccf}. We express these causal effects between processes in terms of generalised path-coefficients of directed paths on the process graph, i.e., via a generalisation of Wright's path rule  (Section \ref{subsection: causal effects}).
    \item Furthermore, in Proposition \ref{prop: trek-rule processes} we describe the \textit{auto-covariance sequence (ACS)} of a SVAR process in terms of certain paths on the process graph, so-called treks. This description generalises the well known trek-rule for linear Gaussian SCMs (Section \ref{subsectio: trek-rule}). An important motivation for providing a process-level path- and trek-rule is that, in the context of SEMs, the path- and trek-rule have been instrumental in deriving conditions on the causal graph under which causal effects can be recovered from observational data in the presence of latent confounding, see e.g. \cite{10.1214/22-AOS2221, foygel2012half, weihs2018determinantal, kuroki2014measurement, drton2011global, leung2016identifiability}. 
    \item Finally, we transport these generalisations to the frequency domain. The resulting frequency domain expressions are more compact and closely related to those familiar from linear Gaussian SCMs. In addition, these expressions allow us to efficiently visualise the causal structure in SVAR processes at the level of the process graph (Section \ref{section: frequency domain}). 
\end{enumerate}
See also Table \ref{tab:summary} for a formula-based summary of these generalisations. Taken together, they provide a calculus that is fundamental to process-level causal effect computation and that allows one to disregard the details of the underlying time series graph. \\

The remainder of the paper is organized as follows. In Section \ref{section: summary}, we illustrate our main findings using a simple SVAR process as a toy model. In Section \ref{section: Preliminaries}, we recall the necessary preliminaries on mixed graphs, linear Gaussian SCMs and SVAR processes. The main part of the paper is covered in Section \ref{section: structural equation process} and \ref{section: frequency domain}, as outlined above. Section \ref{section: outlook} concludes with a summary and a selection of possible future research directions for which the findings formulated in this paper could be a starting point.
\section{Illustration of main results} \label{section: summary}
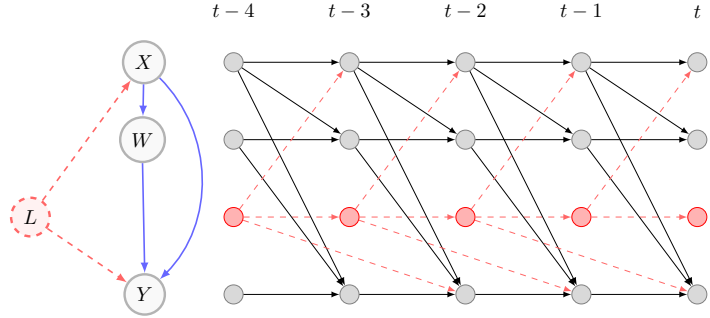
\begin{figure}
    \centering
    \resizebox{0.8\textwidth}{!}{
    \begin{tikzpicture}[
       ts_node/.style={circle, draw=gray, fill=gray!30},
       ts_node_latent/.style={circle, draw=red, fill=red!30},
       ts_node_cause/.style={circle, draw=red!60, fill=red!5, thick},
       ts_node_cause_control/.style={circle, draw=red!60, fill=red!30, ultra thick},
       ts_node_condition/.style={circle, draw=violet!60, fill=violet!30, thick},
       ts_node_target/.style={circle, draw=green!60, fill=green!60},
       ts_header/.style={rectangle, minimum size=1cm},
       snode/.style={circle, draw=gray!60, fill=gray!5, very thick},
       snode_latent/.style={circle, draw=red!60, fill=red!5, dashed, very thick},
       snode_cause/.style={circle, draw=red!60, fill=red!5, very thick},
       snode_condition/.style={circle, draw=violet!60, fill=violet!5, very thick},
       snode_target/.style={circle, draw=green!60, fill=green!5, very thick},
    ]
    \node[ts_header] (t_0) at (0,0) {$t$};
    \node[ts_node] (Xt_0) [below=0.2cm of t_0] {};
    \node[ts_node] (Wt_0) [below=of Xt_0] {};
    \node[ts_node_latent] (Zt_0) [below=of Wt_0] {};
    \node[ts_node] (Yt_0) [below=of Zt_0] {};

    \node[ts_header] (t_1) [left=of t_0] {$t-1$};
    \node[ts_node] (Xt_1) [below=0.2cm of t_1] {};
    \node[ts_node] (Wt_1) [below=of Xt_1] {};
    \node[ts_node_latent] (Zt_1) [below=of Wt_1] {};
    \node[ts_node] (Yt_1) [below=of Zt_1] {};

    \node[ts_header] (t_2) [left=of t_1] {$t-2$};
    \node[ts_node] (Xt_2) [below=0.2cm of t_2] {};
    \node[ts_node] (Wt_2) [below=of Xt_2] {};
    \node[ts_node_latent] (Zt_2) [below=of Wt_2] {};
    \node[ts_node] (Yt_2) [below=of Zt_2] {};

    \node[ts_header] (t_3) [left=of t_2] {$t-3$};
    \node[ts_node] (Xt_3) [below=0.2cm of t_3] {};
    \node[ts_node] (Wt_3) [below=of Xt_3] {};
    \node[ts_node_latent] (Zt_3) [below=of Wt_3] {};
    \node[ts_node] (Yt_3) [below=of Zt_3] {};

    \node[ts_header] (t_4) [left=of t_3] {$t-4$};
    \node[ts_node] (Xt_4) [below=0.2cm of t_4] {};
    \node[ts_node] (Wt_4) [below=of Xt_4] {};
    \node[ts_node_latent] (Zt_4) [below=of Wt_4] {};
    \node[ts_node] (Yt_4) [below=of Zt_4] {};

    \node[snode] (X) [left=of Xt_4] {$X$};
    \node[snode] (W) [left=of Wt_4] {$W$};
    \node[snode_latent] (Z) [left=3.cm of Zt_4] {$L$}; 
    \node[snode] (Y) [left=of Yt_4] {$Y$};


    \draw[->] (Xt_1) -- (Xt_0);
    \draw[->] (Yt_1) -- (Yt_0);
    \draw[->] (Wt_1) -- (Wt_0);
    \draw[->, draw=red!60, dashed] (Zt_1) -- (Zt_0);

    \draw[->] (Xt_2) -- (Xt_1);
    \draw[->] (Yt_2) -- (Yt_1);
    \draw[->] (Wt_2) -- (Wt_1);
    \draw[->, draw=red!60, dashed] (Zt_2) -- (Zt_1);

    \draw[->] (Xt_3) -- (Xt_2);
    \draw[->] (Yt_3) -- (Yt_2);
    \draw[->] (Wt_3) -- (Wt_2);
    \draw[->, draw=red!60, dashed] (Zt_3) -- (Zt_2);

    \draw[->] (Xt_4) -- (Xt_3);
    \draw[->] (Yt_4) -- (Yt_3);
    \draw[->] (Wt_4) -- (Wt_3);
    \draw[->, draw=red!60, dashed] (Zt_4) -- (Zt_3);

    \draw[->] (Xt_1) -- (Wt_0);
    \draw[->] (Xt_1) -- (Yt_0);
    \draw[->] (Wt_1) -- (Yt_0);
    \draw[->, draw=red!60, dashed] (Zt_1) -- (Xt_0);

    \draw[->] (Xt_2) -- (Wt_1);
    \draw[->] (Xt_2) -- (Yt_1);
    \draw[->] (Wt_2) -- (Yt_1);
    \draw[->, draw=red!60, dashed] (Zt_2) -- (Xt_1);

    \draw[->] (Xt_3) -- (Wt_2);
    \draw[->] (Xt_3) -- (Yt_2);
    \draw[->] (Wt_3) -- (Yt_2);
    \draw[->, draw=red!60, dashed] (Zt_3) -- (Xt_2);

    \draw[->] (Xt_4) -- (Wt_3);
    \draw[->] (Xt_4) -- (Yt_3);
    \draw[->] (Wt_4) -- (Yt_3);
    \draw[->, draw=red!60, dashed] (Zt_4) -- (Xt_3);

    \draw[->, draw=red!60, dashed] (Zt_2) -- (Yt_0);

    \draw[->, draw=red!60, dashed] (Zt_3) -- (Yt_1);
    
    \draw[->, draw=red!60, dashed] (Zt_4) -- (Yt_2);

    \draw[->, thick, draw=blue!60] (X) to [out=315, in=45](Y);
    \draw[->, thick, draw=blue!60] (X) -- (W);
    \draw[->, thick, draw=blue!60] (W.south) -- (Y.north);
    \draw[->, thick, draw=red!60, dashed] (Z) -- (X);
    \draw[->, thick, draw=red!60, dashed] (Z) -- (Y);
    \end{tikzpicture}
    }
    \caption{The graph on the right shows a finite subgraph of the infinite time series graph of a SVAR process that consists of three observed processes indexed by $X,W,Y$ and a latent process indexed by $L$. The first row models the states of process $X$ at time-steps $t-4, \dots , t$. The horizontal arrows model the auto-dependencies in the respective process. The arrows between different lines model cross-dependencies. The infinite time series graph extends infinitely to the left and right by repeating the structure depicted in this figure. The graph on the left is the process graph. It collapses the time series graph along the temporal axes. Two distinct nodes in the process graph are connected by a directed link if and only if there is at least one lagged or contemporaneous edge between the respective processes and with the respective direction in the  time series graph.}
    \label{fig:time_series_graph}
\end{figure}

We exemplify graphical time series models on a system composed of three observed processes $\X_X,\X_W,\X_Y$ and a latent process $\X_L$. Figure \ref{fig:time_series_graph} depicts the time series graph together with its associated process graph. The vertices of the process graph consist of the three observed processes-indices and the latent process-index. In the process graph, a directed link connects two distinct processes if and only if there is a time-lagged or contemporaneous link in the time series graph between the two processes. For each link in the process graph we construct two versions of a generalised link-coefficient, namely, a time and frequency domain version, see Figure \ref{fig:summary_graph_parameterizations}. For each version, we now give qualitative description.  

\textbf{Time domain.} (Figure \ref{fig:summary_graph_parameterizations} a.)). Consider the time series graph in Figure \ref{fig:time_series_graph}. The direct causal link between $X$ and $Y$, together with the assumption that the underlying data-generating process is linear, implies that for any given time-point $t$ we can identify part of $\X_Y(t)$ as the "direct" contribution of the random variables $\X_X(t-k)$ with $k \geq0$. This contribution is given as a linear function determined by only those links in the time series graph that go directly from $X$ to $Y$, and those that encode the auto-dependencies of $Y$. We term this function the direct effect-filter, denoted as $\dirtmp{X}{Y}$, and assign it to the link $X \to Y$ in Figure \ref{fig:summary_graph_parameterizations}a.). Applying this function to the process $\X_X$ gives the process $\dirtmp{X}{Y}\ast \X_X$, which is the component of $\X_Y$ that is directly determined by $\X_X$. This notion of causal effects is based on our previous work \cite{reiter2022causal}. Recent works on instrumental processes \cite{thams2022instrumental, mogensen2022instrumental} also employ a related notion of causal effects.

\textbf{Frequency domain.} (Figure \ref{fig:summary_graph_parameterizations} b.)). Instead of viewing a process as a time-ordered sequence of dependent random states, it can be represented as a superposition of oscillations, each at a particular frequency. This view is also known as the \textit{spectral representation} of a stochastic process, see \cite{brockwell2009time}. The amplitude of a given frequency is a random variable, so that, roughly speaking, the amplitudes of two different frequencies are independent of each other, see \cite[Section 11.8]{brockwell2009time} for a precise formulation. The so-called \textit{spectral density} measures how the temporal variability of a process is distributed across frequencies. The Fourier-Transformation of $\dirtmp{X}{Y}$ denoted by $\dirfrq{X}{Y}$ in Figure \ref{fig:summary_graph_parameterizations}.b.), can be considered as a frequency domain link coefficient. It quantifies the relationship between the spectral density of the process $ \X_X$ and the spectral density of $\dirtmp{X}{Y}\ast \X_X$, the latter of which we have identified as the component of $\X_Y$ that is directly determined by $\X_X$  (see previous paragraph). In Section \ref{section: frequency domain}, we express $\dirfrq{X}{Y}$ as the restriction of a complex rational polynomial to the complex unit circle, such that the parameters of this rational polynomial are given by the coefficients of the underlying SVAR process.  

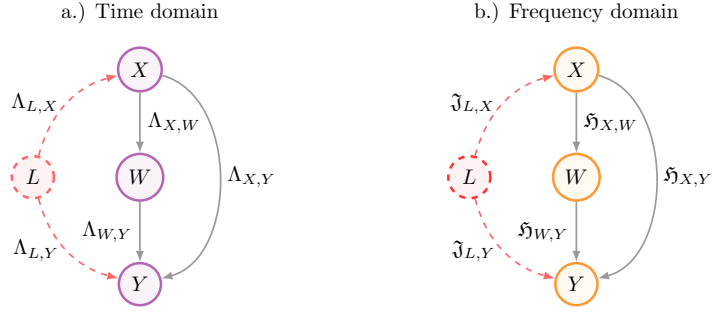
\begin{figure}
    \centering
    \resizebox{0.8\textwidth}{!}{
    \begin{tikzpicture}[
        time_dom/.style={circle, draw=violet!60, fill=violet!5, very thick, minimum size=7mm},
        time_dom_latent/.style={circle, draw=red!60, fill=red!5, dashed, very thick, minimum size=7mm},
        freq_dom/.style={circle, draw=orange!80, fill=orange!5, very thick, minimum size=7mm},
        freq_dom_latent/.style={circle, draw=red!80, fill=red!5, dashed, very thick, minimum size=7mm},
        window_dom/.style={circle, draw=orange!60, fill=orange!5, very thick, minimum size=7mm},
        header/.style={rectangle, minimum size=7mm}
    ]
        \node[header] (time_dom) at (0.,0.) {\text{a.) Time domain}};
        \node[time_dom] (x_time) [below=0.2cm of time_dom] { $X$};
        \node[time_dom] (w_time) [below=of x_time] { $W$};
        \node[time_dom_latent] (z_time) [left=of w_time] {$L$};
        \node[time_dom] (y_time) [below=of w_time] {$Y$};

        \draw[->, draw=gray!80, thick] (x_time.south) -- node[right] {$\dirtmp{X}W{}$} (w_time.north);
        \draw[->, draw=gray!80,thick] (w_time.south) -- node[left] {$\dirtmp{W}{Y}$} (y_time.north);
        \draw[->, draw=red!60,thick, dashed] (z_time) to [out=75, in=195] node[left] {$\dirtmp{L}{X}$} (x_time);
        \draw[->, draw=red!60, dashed, thick] (z_time) to [out=285, in=165] node[left] {$\dirtmp{L}{Y}$} (y_time);
        \draw[->, draw=gray!80, thick] (x_time) to [out=345, in=15] node[right] {$\dirtmp{X}{Y}$} (y_time);

        
        \node[header] (freq_dom) [right=4.cm of time_dom] {\text{b.) Frequency domain}};
        \node[freq_dom] (x_freq) [below=.2cm of freq_dom] { $X$};
        \node[freq_dom] (w_freq) [below=of x_freq] { $W$};
        \node[freq_dom_latent] (z_freq) [left=of w_freq] {$L$};
        \node[freq_dom] (y_freq) [below=of w_freq] {$Y$};

        \draw[->, draw=gray!80, thick] (x_freq.south) -- node[right] {$\dirfrq{X}W{}$} (w_freq.north);
        \draw[->, draw=gray!80,thick] (w_freq.south) -- node[left] {$\dirfrq{W}{Y}$} (y_freq.north);
        \draw[->, draw=red!60,thick, dashed] (z_freq) to [out=75, in=195] node[left] {$\lfrq{L}{X}$} (x_freq);
        \draw[->, draw=red!60, dashed, thick] (z_freq) to [out=285, in=165] node[left] {$\lfrq{L}{Y}$} (y_freq);
        \draw[->, draw=gray!80, thick] (x_freq) to [out=345, in=15] node[right] {$\dirfrq{X}{Y}$} (y_freq);

    \end{tikzpicture}
    }
    \caption{These graphs are the process graphs from the example process in Figure \ref{fig:time_series_graph}. Part a.) shows the process graph in which the edges are annotated by the time domain link-coefficients. Part b.) Shows the process graph in which the edges are annotated by the frequency domain link coefficients.}
    \label{fig:summary_graph_parameterizations}
\end{figure}

 The filters we assigned to the edges in Figure \ref{fig:summary_graph_parameterizations}.a.) constitute the linear model for the process graph. The multiplication operation in SEMs generalises to an operation on $\mathbb{Z}$-indexed stochastic processes, that is, the above mentioned convolution, denoted by the symbol ``$\ast$''. If $\X = (\X(t))_{t\in \Z}$ is a stationary stochastic process and $\Lambda = (\Lambda(s))_{s \in \Z}$ an absolutely summable filter, then the convolution $\Lambda \ast \X$ is likewise a stationary stochastic process. Assuming that the time series graph in Figure\ref{fig:time_series_graph} is parameterised by a SVAR-process, we can decompose each of the processes $\X_X, , \X_L, \X_W, \X_Y$ into internal and externally imposed dynamics, namely
\begin{align*}
    \X_X &= \ltmp{L}{X} \ast \X_L + \X_X^\internal \\ 
    \X_W &= \dirtmp{X}{W} \ast \X_X + \dirtmp{Y}{W}\ast \X_Y + \X_W^\internal\\
    \X_L &= \X_L^\internal\\
    \X_Y &= \dirtmp{W}{Y} \ast \X_W + \dirtmp{X}{Y}\ast \X_X + \ltmp{L}{Y} \ast \X_L + \X_Y^\internal.
\end{align*}
The internal dynamics $\X_X^\internal, \X_W^\internal, \X_L^\internal, \X_Y^\internal$ represent mutually independent stochastic processes that need not be white noise processes and include the auto-dependency structure in SVAR processes. Let us combine the observed processes into the multivariate process $\X_\mathbf{O} = (\X_X, \X_W, \X_Y)$, and similarly we combine the internal dynamics of the observable processes as $\X_\mathbf{O}^\internal = (\X_X^\internal, \X_W^\internal, \X_Y^\internal)$. We arrange the filters of the edges between observed processes in the matrix-valued filter $\Dirtmp$ and the filters of the edges pointing from the latent to an observed process in the matrix-valued filter $\Gamma$ so that the equations above can be put together and applied recursively, which yields the expression
\begin{align*}
    \X_\mathbf{O} &= \Dirtmp^\top \ast \X_\mathbf{O} + \X_\mathbf{O}^\internal + \Gamma^\top \ast \X_L \\
    &= (\sum_{k \geq 0} \Dirtmp^k)^\top \ast (\X_\mathbf{O}^\internal + \Gamma \ast \X_L).
\end{align*}
The \textit{auto-covariance sequence (ACS)} of the stochastic process $\X_\mathbf{O}$, written $\C_\mathbf{O}$, is the generalisation of the covariance matrix $\Sigma$ to SEMs. Fourier transforming the ACS gives the spectral density of $\mathbf{O}$, i.e, a complex valued function defined for every frequency $\omega \in [0, 2\pi)$. We will see that the spectral density $\mathcal{S}_\mathbf{O}$ of the observable processes $\mathbf{O}$ is given in terms of the causal parameter $\Dirfrq$, which is the Fourier-Transformation of $\Dirtmp$, and the spectral density $\Sp{\mathbf{O}}^\linternal \coloneqq \Sp{\X_\mathbf{O}^\internal + \Gamma^\top \ast \X_L}$ by the formula in the third row and third column of Table \ref{tab:summary}.
In Section \ref{section: frequency domain}, we use this expression for the spectral density to derive a frequency domain generalisation of the trek-rule at the level of the process graph. This explicates that at each frequency the structure of the spectral density is dictated only by the process graph, while the complexity due to the lag-structure of a SVAR process resolves into the specific functional forms of $\Dirfrq$ and $\Sp{\mathbf{O}}^\linternal$. 
\begin{table}
    \centering
    \resizebox{\textwidth}{!}{
    \begin{tabular}{c|c||c|c}
        & \makecell{\textbf{lin. Gaussian SCM }\\ (Section \ref{section: Preliminaries})} & \makecell{\textbf{SEP time domain} \\ (Section \ref{section: structural equation process})} & \makecell{\textbf{SEP frequency domain } \\ (Section \ref{section: frequency domain})} \\ \hline
        \makecell{ \\ Causal model \\ (parameters for $D$) } & \makecell{Direct effect-matrix \\  $A\in \R^{|\mathbf{O}| \times |\mathbf{O}|}$ } & \makecell{Direct effect-filter \\ $ \Lambda : \Z \to \R^{|\mathbf{O}| \times |\mathbf{O}|}$} & \makecell{Direct transfer-function \\$\Dirfrq: [0, 2 \pi) \to \Cplx^{|\mathbf{O}| \times |\mathbf{O}|}$}\\ \hline
        \makecell{ \\ Hidden confounding \\ (parameters for $B$)} & \makecell{Noise covariance \\ $\Omega \in \R^{|\mathbf{O}| \times |\mathbf{O}|}$} & \makecell{Noise ACS \\ $\C_\mathbf{O}^\linternal: \Z \to \R^{|\mathbf{O}| \times |\mathbf{O}|} $} & \makecell{Noise spectral density\\$\Sp{\mathbf{O}}^\linternal: [0, 2\pi) \to \Cplx^{|\mathbf{O}| \times |\mathbf{O}|}$ } \\ \hline \hline
        \makecell{\\ Observational \\ covariance} & \makecell{$\Sigma =$ \\$ (I-A)^{-\top} \Omega (I-A)^{-1}$} & \makecell{$\C_\mathbf{O}=$ \\ $(\Dirtmp^\infty)^\top \ast\C^\linternal_\mathbf{O}\tcvl\Dirtmp^\infty$} & \makecell{$\Sp{\mathbf{O}}=$\\ $(\mathcal{I} - \Dirfrq)^{-\top} \Sp{\mathbf{O}}^\linternal (\mathcal{I} - \Dirfrq)^{-\ast}$} \\ \hline
        \makecell{\\ Path-rule \\ (path-coefficients)} & \makecell{$(I-A)^{-1}_{V,W}=$ \\ $\sum_{\pi \in \mathrm{P}(V,W)}A^{(\pi)}$} & \makecell{ $\Dirtmp^\infty_{V,W}=$ \\ $\sum_{\pi \in \mathrm{P}(V,W)} \Dirtmp^{(\pi)}$} & \makecell{$(\mathcal{I}- \Dirfrq)^{-1}_{V,W}=$ \\ $ \sum_{\pi \in \mathrm{P}(V,W)} \Dirfrq^{(\pi)}$} \\ \hline
        \makecell{\\ Trek-rule \\ (trek-monomials)} & \makecell{$\Sigma_{V,W}=$ \\ $\sum_{\pi \in \T{}(V,W)} \Sigma^{(\pi)}$} & \makecell{$\C_{V,W}=$ \\ $\sum_{\pi \in \T{}(V,W)} \C^{(\pi)}$} & \makecell{$\Sp{V,W}=$ \\ $ \sum_{\pi \in \T{}(V,W)} \Sp{}^{(\pi)}$}
    \end{tabular}}
    \caption{In this table we compare the known formulae for linear Gaussian SCMs parameterising a mixed causal graph $G=(\mathbf{O}, D, B)$ over a finite vertex set $\mathbf{O}$ with directed edges $D$ and bidirectional edges $B$ with their generalisations to SVAR processes represented as structural equation processes in both the time- and frequency domain.}
    \label{tab:summary}
\end{table}

A central problem in causal inference is to decide whether causal effects can be identified from observational data and, given that this is possible, how to compute them. The following example suggests that the process graph of a SVAR process already contains enough information to formulate graphical criteria for identifying causal effects and corresponding algebraic relations by which these effects can be computed from observations.   
\begin{example}["Front-door" \cite{pearl2009causality} in the frequency domain]\label{ex: illustration main result}
    The relations among the entries in the spectral density, see Table \ref{tab:summary}, let us recover part of the causal parameter $\Dirfrq$ from the spectral density $\Sp{\mathbf{O}}$ of the observed processes. Specifically, the link-coefficients $\dirfrq{X}{W}$ and $\dirfrq{W}{Y}$ can be computed frequency-wise as follows
    \begin{align} \label{eq: front-door frequency domain}
         \dirfrq{X}{W} &= \frac{\Sp{W,X}}{\Sp{X}} & \dirfrq{W}{Y} &= \frac{\Sp{W, Y} - \dirfrq{X}{W} \Sp{X,Y}}{\Sp{W}- 2\mathrm{Re}(\dirfrq{X}{W} \Sp{X,W})+|\dirfrq{X}{W}|^2 \Sp{X}}.
    \end{align}
    In this case we do not have to worry about zeros in the denominator, as follows from the stability assumptions we are imposing on SVAR processes. That means, for every possible SVAR process that parameterises the summary graph in Figure \ref{fig:summary_graph_parameterizations} we can recover the functions $\dirfrq{X}{W}$ and $\dirfrq{W}{Y}$. Note that these relations only use the structure of the process graph. The additional complexity stemming from the temporal structure does not affect these relations.
\end{example}

\section{Preliminaries}\label{section: Preliminaries}
We begin this section by clarifying the necessary notions about mixed graphs. After that we recall some well known facts about linear Gaussian SCMs with latent factors, as studied in e.g. \cite{10.1214/22-AOS2221, bollen1989structural}, which we seek to generalise to SVAR processes at the level of their process graphs. Finally, we spell out the concepts and notions on SVAR processes that we need throughout the main part of the this paper. 
\subsection{Graphical notions}
A mixed graph is a triple $G=(V, D, B)$, which consists of a finite set of vertices $V$ and two sets of edges $D, B \subset V \times V$. An element $(v,w) \in D$ is called a directed edge on $G$ and is denoted as $v \to w$. We write $v \not \to w$ to signify that $(v,w) \notin D$. The edges in $B$ are bidirectional, which means that a pair $(v,w) \in B$ if and only if $(w,v) \in B$. A bidirectional edge between two nodes $v$ and $w$ is denoted by $v \leftrightarrow w$, and with $v \not \leftrightarrow w$ we express that there is no bidirectional edge between $v$ and $w$.  

In this work, a \textit{path} in a mixed graph $G$ is a sequence of consecutive edges $\pi = (e_1, \dots , e_k)$, where $e_i$ is either $v_i \to v_{i+1}$, $v_{i} \leftarrow v_{i+1}$ or $v_i \leftrightarrow v_{i+1}$. In particular, edges and vertices can appear more than once on a path. Irrespective of which form the edges take, the path $\pi$ is said to go from $v_1$ to $v_{k+1}$. A path is called a directed path if $e_i = v_i \to v_{i+1}$ for all $1\leq i \leq k$. We denote the set of all directed paths going from $v$ to $w$ by $\mathcal{P}(v,w)$. By convention, we include the empty path in $\mathcal{P}(v,v)$. So with our path term we follow the convention as used in e.g. \cite{foygel2012half, 10.1214/22-AOS2221}. 

A path is said to be a \textit{trek}, see e.g. \cite{sullivant2010trek, foygel2012half}, if it has one of the following two types of structure
\begin{align}\label{trek-type I}
    v_{l}^{L} \leftarrow \cdots \leftarrow v_{1}^L \leftarrow v \to v_1^R \to \cdots \to v_r^R\\
    \label{trek-type II}
    v_l^L \leftarrow \cdots \leftarrow v_0^L \leftrightarrow v_0^R \to \cdots \to v_r^R, 
\end{align}
where $\{v_i^L\}$ or $\{v_j^R\}$ in (\ref{trek-type I}) can be empty. In this case the trek is directed path. 
If $\pi$ is a trek of type (\ref{trek-type I}), then its left-hand side is $\trekL(\pi) \coloneqq \{ v, v_1^L \dots, v_l^L \}$ and its right-hand side is $\trekR(\pi) \coloneqq \{ v, v_1^R, \dots, v_r^R \}$. If $\pi$ is a trek of type (\ref{trek-type II}), then its left-hand side is $\trekL(\pi) \coloneqq \{v_0^L, \dots, v_l^L\}$ and its right-hand side is $\trekR(\pi) \coloneqq \{v_0^R, \dots, v_r^R \}$, see \cite{10.1214/22-AOS2221, foygel2012half}. We represent the set of all treks from $v$ to $w$ with $\mathcal{T}(v,w)$.   

\subsection{Preliminaries on linear SEMs}
 A SEM with latent factors is given by a collection of $m$ observed variables $\mathbf{X} = (X_v)_{v\in V}$, a collection of $d$ mutually independent Gaussian latent variables $\mathbf{L}=(L_{h})_{h \in H}$, an $m \times m$ matrix $A$, an $d \times m$ matrix $C$ and an $m$-variate mutually independent Gaussian noise vector $\eta \sim \mathcal{N}(\mathbf{0}, \Omega_\diag)$ such that
\begin{equation} \label{eq: SEM}
    \begin{split}
        \mathbf{X} &= A^\top \mathbf{X} + \eta + C^\top \mathbf{L} \\
        &= (I - A)^{-\top} (\eta + C^\top \mathbf{L}).
    \end{split}
\end{equation}
In particular, $I-A$ is required to be invertible. Furthermore, the latent vector and the noise vector are assumed to be independent of each other. Then the observed variables $\mathbf{X}$ follow a zero-mean Gaussian distribution, whose covariance is given by
\begin{align}\label{eq: covariance Gaussian SEM}
    \Sigma = (I- A)^{- \top} (\Omega_\diag + C^\top \Sigma_\mathbf{L} C) (I - A)^{-1}.
\end{align}

Let $G=(V, D,B)$ be a mixed graph consisting of a set of directed edges $D$ and a set of bidirectional edges $B$. Following \cite{10.1214/22-AOS2221, foygel2012half}, we denote by $\R^D_\reg$ the space of all $m\times m$-dimensional matrices $A$ such that $A_{v,w} = 0$ if $v \not \to w$. Furthermore, we denote by $\PD_m$ the set of all positive definite symmetric (hermitian if considered over the complex numbers) $m\times m$ matrices and by $\PD(B)$ the set of matrices $\Omega=(\Omega_{v,w})\in \PD_m$ such that $\Omega_{v,w}= 0$ if $v \neq w$ and $v \not \leftrightarrow w$. The SEM (\ref{eq: SEM}) is said to parameterise the mixed graph $G$ if $A \in \R^D_\reg$ and $\Omega \coloneqq \Omega_\diag + C^\top \Sigma_L C \in \PD(B)$. 

A fundamental observation, known as the path-rule \cite{wright1934method}, is that causal effects in linear SEMs are characterised by directed paths on $G$ and the entries in the parameter matrix $A$. If $\pi = v_1 \to \cdots \to v_{k+1}$ is a directed path on $G$, then its \textit{path-coefficient} $A^{(\pi)}$ is the product of the entries in $A$ that are associated with the links composing $\pi$, i.e.,   
\begin{align*}
    A^{(\pi)} &\coloneqq \prod_{i=1}^{k} A_{v_{i}, v_{i+1}}.
\end{align*}
The path-rule exhibits the entries in $(I-A)^{-1}$ as sums of path-coefficients, i.e., 
\begin{align}\label{eq: path-rule}
    (I - A)^{-1}_{v,w} &= \sum_{\pi \in \mathcal{P}(v,w)} A^{(\pi)}. 
\end{align}
Similarly, the so-called \textit{trek-rule} \cite{wright1934method} expresses the covariance structure in terms the parameters of $A$, the covariance $\Omega = \Omega_\diag + C^\top \Sigma_\mathbf{L} C$ and the treks on $G$. Suppose $\pi$ is a trek on $G$, then its \textit{trek-monomial} \cite{foygel2012half, sullivant2010trek} is defined as
\begin{align*}
    \Sigma^{(\pi)} &\coloneqq \begin{cases}
        A^{(\trekL(\pi))}\Omega_{v,v} A^{(\trekR(\pi))}  & \text{if $\pi$ is of type (\ref{trek-type I})} \\
        A^{(\trekL(\pi))} \Omega_{v_0^L, v_0^R} A^{(\trekR(\pi))}  & \text{if $\pi$ is of type (\ref{trek-type II})}
    \end{cases}
\end{align*}
The trek-rule characterises the entries in $\Sigma$ as sums of trek-monomials on $G$, i.e.,  
\begin{align}\label{eq: trek-rule}
    \Sigma_{v,w} &= \sum_{\pi \in \mathcal{T}(v,w)} \Sigma^{(\pi)}. 
\end{align}

Previous research, see e.g. \cite{10.1214/22-AOS2221, foygel2012half, weihs2018determinantal, kuroki2014measurement, drton2011global, leung2016identifiability, yao2022algebraic}, has explored conditions on the mixed graph $G$ that allow to recover the parameter matrix $A$ as rational polynomials evaluated on the covariance $\Sigma$. Some of these conditions exploit the trek-rule as a bridge between the graphical structure $G$ and the structure of the covariance $\Sigma$. Another line of research is concerned with the problem of inferring the structure of $G$ from $\Sigma$. For directed acyclic graphs $G=(V, D, \emptyset)$, the connection between the structure of $G$ and the covariance $\Sigma$ lies in a correspondence  \cite{lauritzen1996, pearl1990independence} between \textit{d-separation} \cite{pearl1990independence} and conditional independence statements, which can be expressed as rank conditions on certain submatrices of $\Sigma$ \cite{sullivant2010trek, lauritzen1996}. In general mixed graphs, the graphical notion of \textit{t-separation} \cite{sullivant2010trek} generalises d-separation. Using the trek-rule, it has been found in \cite{sullivant2010trek} that also t-separation statements correspond to rank conditions on associated submatrices of $\Sigma$. 

\begin{example}[Front-door \cite{pearl2009causality} effect-identification as a special case of Example \ref{ex: illustration main result}]\label{ex: sem}
To demonstrate that the concepts developed in this paper generalise commonly known notions about SEMs, we consider an SEM with latent factors whose causal graph (Figure \ref{fig:intro_causal_graph}) is the same as the process graph considered in the previous section (Figure \ref{fig:time_series_graph}).  
 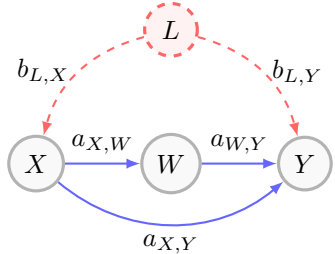
\begin{figure}
     \centering
     \begin{tikzpicture}[
        main_node/.style={{circle, draw=gray!60, fill=gray!5, very thick, minimum size=7mm}},
        latent_node/.style={circle, draw=red!60, fill=red!5, dashed, very thick, minimum size=7mm}
     ]
        \node[latent_node] (z) at (0,0) { $L$};
        \node[main_node] (w) [below=of z] { $W$};
        \node[main_node] (x) [left=of w] {$X$};
        \node[main_node] (y) [right=of w] {$Y$};

        \draw[->, thick, draw=red!60, dashed] (z) to [out=195, in=75] node[left] {$b_{L,X}$} (x);
        \draw[->, thick, draw=red!60, dashed] (z) to [out=345, in=105] node[right] {$b_{L,Y}$} (y);
        \draw[->, thick, draw=blue!60] (x.east) -- node[above] {$a_{X,W}$} (w.west);
        \draw[->, thick, draw=blue!60] (w.east) -- node[above] {$a_{W,Y}$} (y.west);
        \draw[->, thick, draw=blue!60] (x) to [out=320, in=220] node[below] {$a_{X,Y}$} (y);
     \end{tikzpicture}
     \caption{This graph is the causal graph of the SEM considered in Example \ref{ex: sem}.}
     \label{fig:intro_causal_graph}
 \end{figure}

In this SEM, the causal effect of $X$ on $Y$ is given by the sum $a_{X,W}a_{W,Y} + a_{X,Y}$. The first summand is called the indirect effect of $X$ on $Y$, as it corresponds to the path $X\to W\to Y$. The second is referred to as the direct effect since it corresponds to the direct link $X\to Y$. 

The link-coefficients $a_{X,W}$ and $a_{W,Y}$ can be expressed in terms of rational functions of the entries in $\Sigma$ as follows
\begin{align} \label{eq: effect identification sem}
    a_{X,W} &= \frac{\sigma_{W,X}}{\sigma_{X}} & a_{W,Y} &= \frac{\sigma_{W,Y} - a_{X,W}\sigma_{X,Y}}{\sigma_{W} + 2a_{X,W}\sigma_{W,X} + a_{X,W}^2\sigma_X}.
\end{align}
Observe that these equations structurally coincide with the frequency domain equations (\ref{eq: front-door frequency domain}). Actually, the equations (\ref{eq: front-door frequency domain}) specialise to equations (\ref{eq: effect identification sem}) if the underlying time series model has only contemporaneous effects, namely, if the associated time series graph in Figure \ref{fig:time_series_graph} has only vertical arrows. 
\end{example}

\subsection{Structural vector autoregressive processes}
In this section, we recall the necessary preliminaries on linear time series models and some notions of causal inference in discrete-time systems. 
\begin{definition}[Time series graph, \cite{runge2019detecting}]\label{def: time series graph}
     Let $\mathbf{V}= \{V_1, \dots, V_{m}\}$ be a finite set of process indices. An infinite graph $\mathcal{G}= (\mathbf{V} \times \mathbb{Z}, \mathcal{D})$ is called a \textit{time series graph (TSG)} if the directed edges $\mathcal{D}$ are such that (1) the temporal ordering on the integers is respected, i.e., $V(t-k) \to W(t) \in \mathcal{D} $ only if $k \geq 0$, (2) a contemporaneous link $V(t) \to W(t) \in \mathcal{D}$ only if $V\neq W$, (3) there is a $p \geq 0$ such that $V(t-k) \to W(t) \in \mathcal{D}$ implies $k \leq p$ and (4) if $V(t-k)\to W(t) \in \mathcal{D}$, then $V(t-k +\tau)\to W(t+\tau)\in \mathcal{D}$ for every $\tau \in \mathbb{Z}$. 
\end{definition} 

\begin{definition}[Process graph]
	Let $\mathcal{G} = (\mathbf{V}, \mathcal{D})$ be a time series graph. Its \textit{process graph}, written $S(\mathcal{G}) = (\mathbf{V}, S(\mathcal{D}))$, contains a directed link $V \to W$ if $V \neq W$ and there is a non-negative integer $k\geq 0$ such that $V(t-k) \to W(t) \in \mathcal{D}$.
\end{definition}
What we call a process graph is called a summary graph in \cite[Chapter 10]{peters2017elements}. Since our goal is to parameterise this graph with an SCM of processes, we choose to call it a process graph. 

A \textit{latent structure} on a time series graph $\mathcal{G} = (\mathbf{V} \times \Z, \mathcal{D})$ is given by a partition $\mathbf{V} = \mathbf{O} \cup \mathbf{L}$ such that for any $V \in \mathbf{O}$ and $H \in \mathbf{L}$ one has that $V \to H \notin S(\mathcal{D})$. The elements in $\mathbf{O}$ represent the observed factors and the elements in $\mathbf{L}$ signify the latent factors. A time series graph equipped with a latent-component structure shall be referred to as a \textit{latent-component time series graph}. The set of directed links in a latent-component time series graph can be partitioned into the edges between the observed processes $\mathcal{D}^\mathbf{O}$, the set of edges between the latent processes $\mathcal{D}^\mathbf{L}$, and the set of edges pointing from a latent to an observed process $\mathcal{D}^{\mathbf{L}, \mathbf{O}}$. We term the process graph of a latent-component time series graph \textit{latent-component process graph}. The \textit{latent projection} \cite{maathuis2019handbook} of the latent-component process graph $S(\mathcal{G}) = (\mathbf{O} \cup \mathbf{L}, S(\mathcal{D}^\mathbf{L}) \cup S(\mathcal{D}^\mathbf{L,O}) \cup S(\mathcal{D}^\mathbf{O}))$ is a mixed graph $G' = (\mathbf{O}, S(\mathcal{D}^\mathbf{O}), B)$ over the observed processes $\mathbf{O}$. The set of bidirectional edges $B$ is such that $V \leftrightarrow W$ if and only if $V \neq W$ and there is a trek $\pi$ from $V$ to $W$ on the process graph $G$ that passes through at least one latent process and exactly two observed processes, i.e., $V$ and $W$.  

\begin{definition}[SVAR process]
	Let $\Phi = ( \Phi(k))_{0 \leq k \leq p} $ be a finite sequence of $m \times m$-dimensional matrices and $\mathbf{w} \in \mathbb{R}^{m\times m}$ a diagonal matrix. The parameter pair $(\Phi, \mathbf{w})$ specifies a $m$-dimensional real valued \textit{structural vector autoregressive process (SVAR)} \newline $\X_\mathbf{V} = ( \X_V(t))_{t \in \Z, V \in \mathbf{V}}$ of order $p$ if for every $t$ 
	\begin{displaymath}
		\X_\mathbf{V}(t) = \sum_{k = 0}^p \Phi^\top(k)\X_\mathbf{V}(t-k) + \eta_\mathbf{V}(t),
	\end{displaymath}
    where $\eta_\mathbf{V}(t)\sim \mathcal{N}(\mathbf{0}, \mathbf{w})$ are such that the covariance $\Ep[\eta_{V}(t) \eta_W(s)] = \mathbf{w}_{V,W}$ if and only if $s=t$ and zero otherwise. 
    Let $\mathcal{G} = (\mathbf{V}\times \mathbb{Z}, \mathcal{D})$ be a causal time series graph over a set of processes $\mathbf{V}$. We say that a SVAR process $(\Phi, \mathbf{w})$ is consistent with $\mathcal{G}$ if $V(t-k) \not\to W(t)$ implies that the entry in the $V$-th row and $W$-th column of $\Phi(k)$, denoted by $\phi_{V,W}(k)$, is zero. We call a SVAR process consistent with the process graph $S(\mathcal{G})$ if $V \not\to W$ implies that $\phi_{V,W}(k) = 0$ for all $0 \leq k \leq p$. 
\end{definition}    
For a sub-process $\mathbf{U}\subset \mathbf{V}$ its \textit{mean sequence} and \textit{auto-covariance sequence} (ACS) are defined as follows
\begin{align*}
    \mu_\mathbf{U}(t) &\coloneqq \Ep[\X_\mathbf{U}(t)], \text{ } t \in \Z, \\
    \mathbf{C}_{\mathbf{U}}(t,s) & \coloneqq \Ep[(\X_\mathbf{U}(t)- \mu_\mathbf{U}(t)) (\X_\mathbf{U}(t-s)- \mu_\mathbf{U}(t-s))^\top],\text{ } t,s \in \mathbb{Z}.
\end{align*}
For the remainder of this work, all SVAR processes are assumed to be \textit{stable} \cite{lutkepohl2005new, brockwell2009time}. It then follows from \cite[Proposition 2.1]{lutkepohl2005new} that their mean and auto-covariance-sequences are independent of $t$. Furthermore, all processes in this work have a zero-mean sequence, as all processes are assumed to be stable and the innovation terms $\eta_\mathbf{V}(t)$ are zero-mean vectors.

\section{The structural equation process representation} \label{section: structural equation process}
In this section, we fix a latent-component time series graph $\mathcal{G}= (\mathbf{V}\times \Z, \mathcal{D})$, where $\mathbf{V} = \mathbf{O} \cup \mathbf{L}$ consists of $|\mathbf{O}| = m$ observed and $|\mathbf{L}| = d$ latent processes, and a SVAR-process $\X_{\mathbf{O} \cup \mathbf{L}}$ of order $p$ that is consistent with $\mathcal{G}$ and specified by the parameter-pair $(\Phi, \mathbf{w})$. We denote the latent-component process graph of $\mathcal{G}$ by $G=(\mathbf{O} \cup \mathbf{L}, D)$ and its latent projection by $G'=(\mathbf{O}, D^\mathbf{O}, B)$. We begin this section by recasting the observed SVAR process $\X_\mathbf{O}$ as a structural linear model of processes on the directed part of the projected process graph $G'$ plus the contribution due to the latent process $\X_\mathbf{L}$. We then generalise the notion of causal effects as well the trek-rule to this SCM of processes on the process graph.
\subsection{Reformulating a SVAR process as a structural equation process} \label{subsection: SEP representation} 
We will now rephrase the SVAR process specified by $(\Phi, \mathbf{w})$ as a linear model of stochastic processes on the process graph $G$. For this reformulation, we introduce objects that parameterise the directed links of the process graph, and as we will see in the next subsection, these objects are closely related to direct causal effects in the classical sense \cite{pearl2009causality}.

\begin{definition}[Direct effect-filter] \label{def: direct effect-filter}
    Let $(V, W) \in \mathbf{V} \times \mathbf{V}$ be a pair of process-indices, then the \textit{direct effect-filter} of this pair is the sequence $\dirtmp{V}{W} = (\dirtmp{V}{W}(s))_{s \in \Z}$ whose elements are recursively defined as follows 
 \begin{align}\label{eq: dir_tmp_effect}
        \dirtmp{V}{W}(s) &\coloneqq \begin{cases}
            \sum_{j=1}^s \dirtmp{V}{W}(s -j) \phi_{W, W}(j) + \phi_{V, W}(s) & \text{if $0 \leq s \leq p$} \\
            \sum_{j=1}^p \dirtmp{V}{W}(s-j) \phi_{W,W}(j) & \text{if $s > p$} \\
            0, & \text{ otherwise}
        \end{cases}.
    \end{align}
\end{definition} 
According to this definition, the direct effect-filter $\dirtmp{V}{W}$ is zero at all indices whenever $V \not \to W$. In the following, we denote each direct effect-filter associated with a link $H \to V$ pointing from a latent to an observed process by  $\ltmp{H}{V}$ (instead of $\dirtmp{H}{V}$). We arrange the direct effect-filters into the matrix-valued filters $\Dirtmp$ and $\Ltmp$ such that for every $s \in \Z$
\begin{align*}
    \Dirtmp(s) &\coloneqq (\dirtmp{V}{W}(s))_{V,W \in \mathbf{O}} \in \mathbb{R}^{m \times m} & \Gamma(s) &\coloneqq (\ltmp{H}{V}(s))_{H \in \mathbf{L}, V \in \mathbf{O}} \in \mathbb{R}^{m \times d}. 
\end{align*}
These two filters will constitute the set of linear equations by which the processes $\X_\mathbf{O}$ are defined in terms of themselves, the latent and noise processes at the level of the process graph. 

To formulate the set of equations that define the SCM of processes at the level of the process graph, we recall the convolution operation as a way of composing filters and letting filters operate linearly on stationary stochastic processes. Let $\Lambda = (\Lambda(s))_{s \in \Z} \subset \mathbb{R}^{p \times n}$ and $\Gamma = (\Gamma(r))_{r \in \mathbb{Z}}\subset \mathbb{R}^{n \times m}$ be two filters such that they are entry-wise absolutely summable, i.e., $\sum_{s \in \Z}|\Lambda(s)_{i,j}| < \infty$ resp. $\sum_{r \in \Z}|\Gamma(r)_{i,j}| < \infty$. Then their \textit{convolution}, denoted $\Lambda\ast \Gamma$, and their \textit{tilted convolution}, denoted $\Lambda \tcvl \Gamma$ are defined as follows  
\begin{align*}
	\Lambda \ast \Gamma(u) &\coloneqq \sum_{t \in \Z} \Lambda(t)\Gamma(u - t), \text{ } u \in \Z & \Lambda \tcvl \Gamma (v) &\coloneqq \sum_{t \in \mathbb{Z}} \Lambda(t + v) \Gamma(t), \text{ } v \in \Z.
\end{align*}
The convolution of a filter with a stationary stochastic process is defined analogously.

If $\Lambda$ is a $n\times n$-dimensional filter, then $\Lambda^k$, with $k$ a non-negative integer, denotes the $k$-fold convolution of $\Lambda$. In particular, $\Lambda^0$ is the unit filter $\iota$, which has the identity matrix $I$ at its $0$-th entry and the zero matrix at all other entries. Note that convolution equips the set of entry-wise absolutely summable filters with an associative operation, where $\iota$ acts as the neutral element. 

\begin{definition}[(Projected) internal dynamics]
    Let $O \in \mathbf{O}$, then we define the \textit{internal dynamics} $\X_O^\internal$ to be the linear component of $\X_O$ that is generated by the white-noise process $\eta_{O}$ and the auto-dependencies in $O$. The \textit{projected internal dynamics}, written $\X_O^\linternal$, of process $\X_O$ is the sum of its internal dynamics and the direct contributions of the latent processes, that is, 
\begin{align*}
    \X_O^\internal(t) &\coloneqq \sum_{k =1}^{p} \phi_{O, O}(k)\X_O^\internal(t-k) +  \eta_{O}(t) \\
    \X_O^\linternal(t) &\coloneqq \X_O^\internal(t) + \sum_{L \in \mathbf{L}} \sum_{k=1}^p \phi_{L,O}(k) \X_L(t-k).
\end{align*}
\end{definition}
 In order to ensure that the (projected) internal dynamics are well-defined processes, we will introduce in Theorem \ref{prop: summary SEM} a further requirement on the SVAR parameter $\Phi$. In the following we write $\X_\mathbf{O}^\internal$ and  $\X_\mathbf{O}^\linternal$ to refer to the multivariate processes whose components are the (projected) internal dynamics of $\X_\mathbf{O}$. If $V,W$ are two distinct observed processes, then their internal dynamics are independent of each other, i.e., their cross-covariance sequence, denoted by $\C_{V, W}^\internal$, is the zero-sequence. However, the cross-covariance between the projected internal dynamics of $V$ and $W$, labelled $\C^\linternal_{V,W}$, can be non-zero if $G'$ contains the bidirectional link $V \leftrightarrow W $. These observations follow from the independence assumption on the innovation-terms in a SVAR-model.

We are eventually equipped to express a SVAR-process as a linear model of processes at the level of its process graph. 
\begin{theorem}[Structural equation process representation]\label{prop: summary SEM}
    Let $(\Phi, \mathbf{w})$ be a parameter pair that specifies a stable SVAR process $\X_\mathbf{V}$ consistent with $G$ such that for every $V \in \mathbf{V}$ the following stability condition is satisfied
    \begin{align} \label{condition: stability II}
        \sum_{k=1}^p |\phi_{V,V}(k)| < 1, 
    \end{align} and the power-series of filters $\Dirtmp^\infty \coloneqq\sum_{k = 0}^\infty \Dirtmp^k$ exists and is an entry-wise absolutely summable filter. Then the observed processes satisfy the linear relation
    \begin{equation} \label{eq: processes equation}
        \begin{split}
             \X_\mathbf{O} &= \Dirtmp^\top \ast \X_\mathbf{O} + \Ltmp^\top \ast \mathbf{L} + \X_\mathbf{O}^\internal \\
            &=(\Dirtmp^\infty)^\top \ast \X_\mathbf{O}^\linternal, 
        \end{split}
    \end{equation}
    so that the auto-covariance sequence of the observed process is 
    \begin{equation}
        \begin{split} \label{eq: corss covariance process graph}
            \C_{\mathbf{O}} &= (\Dirtmp^\infty)^\top \ast (\C_\mathbf{O}^\linternal) \tcvl \Dirtmp^\infty, \text{ where}\\
            \C_\mathbf{O}^\linternal &= \C_\mathbf{O}^\internal + \Gamma^\top \ast \C_\mathbf{L} \tcvl \Gamma.
        \end{split}
    \end{equation} 
\end{theorem}
We defer the proof of this theorem to the Appendix \ref{section: proof main theorem}. Suppose $\X_\mathbf{V}= (\X_{V})_{V \in\mathbf{O} \cup \mathbf{L}}$ is a SVAR process specified by a parameter pair $(\Phi, \mathbf{w}) \in \SVAR(G)$, then we refer to the tuple $(\Dirtmp, \C_\mathbf{O}^\linternal)$ as the \textit{time domain SEP representation} of the observable process $\X_\mathbf{O}$. 

\subsection{Causal effects at the level of the process graph}\label{subsection: causal effects}
In this section, we define (controlled) causal effects between entire processes at the level of the process graph $G$ in terms of directed paths on $G$ and the direct effect-filters from Definition \ref{def: direct effect-filter}. On top of that, we express these process level effects as sequences of causal effects at the level of the time series graph, i.e., sums of path-coefficients on the time series graph $\mathcal{G}$. 

Our notion of direct effect-filters in combination with convolution allows us to generalise the notion of path-coefficients from SEMs to SEPs. 
Let $\pi = W_1 \to \cdots \to W_n$ be a directed path on the process graph $G$. The \textit{path-filter} $\Lambda^{(\pi)}$ of $\pi$ is the convolution of the direct effect-filters associated with the links forming the path $\pi$, i.e.,
\begin{align*}
    \Lambda^{(\pi)} \coloneqq \dirtmp{W_1}{W_2}\ast \cdots \ast \dirtmp{W_{n-1}}{W_n}.
\end{align*}

Let $X, Y \in \mathbf{O}$ and $\mathbf{Z} \subset \mathbf{O}$, then a directed path $X \to W_1 \to \cdots \to W_n \to Y$ is said to avoid $\mathbf{Z}$ if none of the $W_i$ lies in $\mathbf{Z}$. We call the possible infinite set of directed paths from $X$ to $Y$ that avoid $\mathbf{Z}$, the $\mathbf{Z}$-\textit{avoiding-paths} from $X$ to $Y$ and denote it by $\mathrm{P}_\mathbf{Z}(X, Y)$. The $\emptyset$-avoiding paths from $X$ to $Y$ are simply all directed paths from $X$ to $Y$, and we write $\mathrm{P}(X,Y)$.

\begin{definition}[Controlled causal effect filter]\label{def: ccf}
    Let $X,Y \in \mathbf{O}$ be two processes and let $\mathbf{Z} \subset \mathbf{O} \setminus \{X,Y\}$ be a collection of processes. The \textit{controlled causal effect filter (CCF)} $\cetemp{X}{Y}{\mathbf{Z}}$ for the causal effect of $X$ on $Y$, controlled for $\mathbf{Z}$, is defined to be the sum of path-filters that are associated with the $\mathbf{Z} \cup \{X\}$-avoiding-paths, i.e.,
 \begin{align*}
    \cetemp{X}{Y}{\mathbf{Z}} &\coloneqq \sum_{\pi \in \mathrm{P}_{\mathbf{Z} \cup \{X\}}(X, Y)}  \Lambda^{(\pi)}.
 \end{align*}
 If $\mathbf{Z}$ is the empty set, then we abbreviate $\cetemp{X}{Y}{\emptyset}$ by $\ucetemp{X}{Y}$ and refer to it as the \textit{causal effect filter (CF)}  of the causal effect of the process $X$ on $ Y$. 
\end{definition}
Note that for any pair $V,W$ of processes the equality $\cetemp{V}{W}{\mathbf{O} \setminus \{W\}} = \dirtmp{V}{W}$, where $\dirtmp{V}{W}$ is the direct effect-filter (\ref{eq: dir_tmp_effect}). Using the above notion of path-filters, Wright's path-rule (\ref{eq: path-rule}) generalises as follows.
\begin{proposition}[Generalised path-rule] \label{prop: generalised path-rule}
    Let $\Dirtmp^\infty$ be as defined in Theorem \ref{prop: summary SEM} and $V,W \in \mathbf{O} $, then 
    \begin{align}\label{eq: generalised path-rule}
    \Dirtmp^\infty_{V,W} &= \sum_{\pi \in \mathrm{P}(V,W)} \Dirtmp^{(\pi)}.
\end{align}
\end{proposition}
Proposition \ref{prop: generalised path-rule} follows from the definition of $\Dirtmp^\infty$ as a power-series of filters and the definition of convolution.

We now express a given CCF $\cetemp{X}{Y}{\mathbf{Z}}$ in terms of path-coefficients at the level of the time series graph $\mathcal{G}$. Let $\rho = W_1(t_1) \to \cdots \to W_m(t_m)$ be a directed path on $\mathcal{G}$, then its \textit{path-coefficient}, written $\Phi^{(\rho)}$, is given by the product of its link-coefficients
\begin{align*}
    \Phi^{(\rho)} &\coloneqq \prod_{i=1}^{m-1} \phi_{W_i, W_{i+1}}(t_{i+1}-t_i).
\end{align*}
\begin{proposition} \label{prop: causal effects process graph}
    Let $X,Y, \mathbf{Z}$ be as in Definition \ref{def: ccf}, and $\mathcal{G}'$ the time series subgraph of $\mathcal{G}$ over the nodes $\mathbf{V} \times  \Z$ such that its set of directed edges is 
    \begin{displaymath}
        \mathcal{D}' \coloneqq\{ V(t-k) \to W(t) \in \mathcal{D}: W \notin \mathbf{Z} \cup \{X\} \} \subset \mathcal{D}.
    \end{displaymath}
    This graph $\mathcal{G}' = (\mathbf{V}, \mathcal{D}')$ is the graph one obtains by deleting all edges in $\mathcal{G}$ which point to any of the nodes in $(X(s), \mathbf{Z}(s))_{s \in \Z}$.
    Further, we denote by $\mathcal{P}'(X(t-s), Y(t))$ the set of all directed paths on $\mathcal{G}'$ that start at $X(t-s)$ and end at $Y(t)$. Then, the $s$-th element of the CCF $\cetemp{X}{Y}{\mathbf{Z}}$ is given by the sum of the path-coefficients of the directed paths from $X(t-s)$ to $Y(s)$ on the subgraph $\mathcal{G}'$, i.e.,   
    \begin{align*}
        \cetemp{X}{Y}{\mathbf{Z}}(s) &=\sum_{\rho  \in \mathcal{P}'(X(t-s), Y(t))} \Phi^{(\rho)}. 
    \end{align*}
\end{proposition}
We provide the proof in the Appendix \ref{section: proof generalised path rule}. 

\begin{remark}
    Proposition \ref{prop: causal effects process graph} characterises the CCF associated with the effect of process $X$ on $Y$, controlled on  $\mathbf{Z}$, as a sequence of multivariate controlled causal effects at the level of the time series graph $\mathcal{G}$. Namely, the $s$-th element of the filter $\cetemp{X}{Y}{\mathbf{Z}}$ measures the expected change in the mean of $\X_Y(t)$ under the multivariate intervention which sets the state of $X$ at time $t-s$ to the value one, the states of the processes $\mathbf{Z}$ at $t-s$ to zero, and the states of $X$ and all processes $\mathbf{Z}$ to zero at every time $r \in \{t-s +1, \dots t \}$. 
    So we can express the entries of the CCF using the well known $\Do$-notation \cite{pearl2009causality}
    \begin{align*}
        \cetemp{X}{Y}{\mathbf{Z}}(s) &= \Ep[\X_Y(t) | \Do((\X_X(r), \X_\mathbf{Z}(r))) \coloneqq (x(r), \mathbf{z}(r)), r \in \{t-s, \dots, t \}],
    \end{align*}
    where $z(r) = 0$ for all $t-s \leq r \leq t$ and $x(t-s) = 1$ and $x(r) = 0$ for $t-s < r \leq 0$. This characterisation is linked to time-windowed causal effects, which we have introduced in preliminary work \cite{reiter2022causal}. 
    Thus, the CCF associated to the causal effect of $X$ on $Y$, controlling $\mathbf{Z}$, can be understood as a collection of controlled causal effects in the classical sense \cite{pearl2009causality, Pearl2001DirectAI, witte2020efficient}. 
    
    Alternatively, we can view the CCF as the result of a process-level intervention. To do so, for every bounded sequence $x = (x(t))_{t \in \Z}$, we define a new multivariate process $\mathbf{O}_{x| \mathbf{Z}} = (V_{\Do(X)=x, \Do(\mathbf{Z}) = \mathbf{0}})_{V \in \mathbf{O} \setminus (\{X\} \cup \mathbf{Z})}$ such that among its component processes the following equations hold  
    \begin{align*}
        V_{x | \mathbf{Z}} &= \dirtmp{X}{V} \ast x + \sum_{U \in \Pa(V) \setminus (\{X \} \cup \mathbf{Z})} \dirtmp{U}{V} \ast U_{x|\mathbf{Z}} + \Gamma \ast \mathbf{L}.
    \end{align*}
     The process $\mathbf{O}_{x | \mathbf{Z}}$ is the process obtained after intervention in the component process $X$, that is, after setting $X(t)$ to the value $x(t)$ and $\mathbf{Z}(t)$ to zero for all $t$. By comparing the mean sequences of two interventional processes $\mathbf{V}_{x| \mathbf{Z}}$ and $\mathbf{V}_{x' | \mathbf{Z}}$  we gain information on the causal effect of $X$ on $Y$, controlled for $\mathbf{Z}$, in terms of the corresponding CCF, namely,  
    \begin{align*}
        \Ep[Y_{x|\mathbf{Z}}] - \Ep[Y_{x' | \mathbf{Z}}] &= \cetemp{X}{Y}{\mathbf{Z}} \ast (x - x'). 
    \end{align*}
\end{remark}

Proposition \ref{prop: causal effects process graph} provides a graphical interpretation of the direct effect-filters. Namely, for a pair of processes $V,W$ the $s$-th element of the direct effect-filter $\dirtmp{V}{W}$ quantifies how the value of the random variable $\X_V(t-s)$ is contributing to the value of $\X_W(t)$ through all the directed paths on $\mathcal{G}$ on which the first link is a time-lagged or contemporaneous effect of $V$ on $W$ and all subsequent links are auto-dependency links of $W$. A directed edge $V \to W$ on the process graph $G$ implies that part of the process $W$ is determined by the process $V$ through its direct effect-filter $\dirtmp{V}{W}$. This part is the stochastic process $\dirtmp{V}{W} \ast \X_V$. 

Similarly, there is an explicit interpretation of CCFs. Let us therefore pick $X,Y$ and $\mathbf{Z}$ as in Proposition \ref{prop: causal effects process graph}. Convolving the CCF $\cetemp{V}{W}{\mathbf{Z}}$ with the process $X$ yields the stochastic process $\cetemp{V}{W}{\mathbf{Z}}\ast \X_X$, which is the component of the process $\X_Y$ that is determined by $\X_X$ along all the causal path-ways that start at $X$ and end at $Y$ without ever passing through any of the processes in $\mathbf{Z}$ or revisiting $X$. 
\begin{example}[Graphical illustration of CCFs]
\begin{figure}
    \centering
    \resizebox{0.9\textwidth}{!}{
    \begin{tikzpicture}[
       ts_node/.style={circle, draw=gray, fill=gray!30},
       ts_node_cause/.style={circle, draw=purple!60, fill=purple!5, very thick},
       ts_node_cause_control/.style={circle, draw=purple!60, fill=purple!30, ultra thick},
       ts_node_condition/.style={circle, draw=violet!60, fill=violet!5,  very thick},
       ts_node_target/.style={circle, draw=olive!60, fill=olive!30, ultra thick},
       ts_header/.style={rectangle, minimum size=1cm},
       snode/.style={circle, draw=gray!60, fill=gray!5, very thick},
       snode_latent/.style={circle, draw=gray!60, fill=gray!5, very thick, dashed},
       snode_cause/.style={circle, draw=purple!60, fill=purple!5, ultra thick},
       snode_condition/.style={circle, draw=violet!60, fill=violet!5, very thick},
       snode_target/.style={circle, draw=olive!60, fill=olive!5, ultra thick},
    ]


    \node[ts_header] (t_0) at (0,0) {$t$};
    \node[ts_node_cause] (Xt_0) [below=0.2cm of t_0] {};
    \node[ts_node] (Wt_0) [below=of Xt_0] {};
    \node[ts_node] (Zt_0) [below=of Wt_0] {};
    \node[ts_node_target] (Yt_0) [below=of Zt_0] {};

    \node[ts_header] (t_1) [left=of t_0] {$t-1$};
    \node[ts_node_cause] (Xt_1) [below=0.2cm of t_1] {};
    \node[ts_node] (Wt_1) [below=of Xt_1] {};
    \node[ts_node] (Zt_1) [below=of Wt_1] {};
    \node[ts_node] (Yt_1) [below=of Zt_1] {};

    \node[ts_header] (t_2) [left=of t_1] {$t-2$};
    \node[ts_node_cause] (Xt_2) [below=0.2cm of t_2] {};
    \node[ts_node] (Wt_2) [below=of Xt_2] {};
    \node[ts_node] (Zt_2) [below=of Wt_2] {};
    \node[ts_node] (Yt_2) [below=of Zt_2] {};

    \node[ts_header] (t_3) [left=of t_2] {$t-3$};
    \node[ts_node_cause] (Xt_3) [below=0.2cm of t_3] {};
    \node[ts_node] (Wt_3) [below=of Xt_3] {};
    \node[ts_node] (Zt_3) [below=of Wt_3] {};
    \node[ts_node] (Yt_3) [below=of Zt_3] {};

    \node[ts_header] (t_4) [left=of t_3] {$t-4$};
    \node[ts_node_cause_control] (Xt_4) [below=0.2cm of t_4] {};
    \node[ts_node] (Wt_4) [below=of Xt_4] {};
    \node[ts_node] (Zt_4) [below=of Wt_4] {};
    \node[ts_node] (Yt_4) [below=of Zt_4] {};

    \node[snode_cause] (X) [left=of Xt_4] {$X$};
    \node[snode] (W) [left=of Wt_4] {$W$};
    \node[snode_latent] (Z) [left=3.cm of Zt_4] {$L$}; 
    \node[snode_target] (Y) [left=of Yt_4] {$Y$};
    \node[ts_header] (name_a) [left= of Z] {\textbf{a.)} };


    \draw[->, dotted] (Xt_1) -- (Xt_0);
    \draw[->, teal, thick] (Yt_1) -- (Yt_0);
    \draw[->, lightgray] (Wt_1) -- (Wt_0);
    \draw[->, dashed, draw=red!60] (Zt_1) -- (Zt_0);

    \draw[->, dotted] (Xt_2) -- (Xt_1);
    \draw[->, teal, thick] (Yt_2) -- (Yt_1);
    \draw[->, teal, thick] (Wt_2) -- (Wt_1);
    \draw[->, dashed, draw=red!60] (Zt_2) -- (Zt_1);

    \draw[->, dotted] (Xt_3) -- (Xt_2);
    \draw[->, teal, thick] (Yt_3) -- (Yt_2);
    \draw[->, teal, thick] (Wt_3) -- (Wt_2);
    \draw[->, dashed, draw=red!60] (Zt_3) -- (Zt_2);

    \draw[->, dotted] (Xt_4) -- (Xt_3);
    \draw[->, lightgray] (Yt_4) -- (Yt_3);
    \draw[->, lightgray] (Wt_4) -- (Wt_3);
    \draw[->, dashed, draw=red!60] (Zt_4) -- (Zt_3);

    \draw[->, lightgray] (Xt_1) -- (Wt_0);
    \draw[->, lightgray] (Xt_1) -- (Yt_0);
    \draw[->, teal, thick] (Wt_1) -- (Yt_0);
    \draw[->, dotted, draw=red!60] (Zt_1) -- (Xt_0);
    \draw[->, lightgray] (Yt_1) -- (Wt_0);

    \draw[->, lightgray] (Xt_2) -- (Wt_1);
    \draw[->, lightgray] (Xt_2) -- (Yt_1);
    \draw[->, teal, thick] (Wt_2) -- (Yt_1);
    \draw[->, dotted, draw=red!69] (Zt_2) -- (Xt_1);
    \draw[->, teal, thick] (Yt_2) -- (Wt_1);

    \draw[->, lightgray] (Xt_3) -- (Wt_2);
    \draw[->, lightgray] (Xt_3) -- (Yt_2);
    \draw[->, teal, thick] (Wt_3) -- (Yt_2);
    \draw[->, dotted, draw=red!60] (Zt_3) -- (Xt_2);
    \draw[->, teal, thick] (Yt_3) -- (Wt_2);

    \draw[->, teal, thick] (Xt_4) -- (Wt_3);
    \draw[->, teal, thick] (Xt_4) -- (Yt_3);
    \draw[->, lightgray] (Wt_4) -- (Yt_3);
    \draw[->, dotted, draw=red!60] (Zt_4) -- (Xt_3);
    \draw[->, lightgray] (Yt_4) -- (Wt_3);

    \draw[->, dashed, draw=red!60] (Zt_2) -- (Yt_0);

    \draw[->, dashed, draw=red!60] (Zt_3) -- (Yt_1);
    
    \draw[->, dashed, draw=red!60] (Zt_4) -- (Yt_2);

    \draw[->, teal, very thick] (X) to [out=315, in=45](Y);
    \draw[->, teal, very thick] (X) -- (W);
    \draw[->, teal, very thick] (W) to [out=255, in=105] (Y);
    \draw[->, draw=red!60, dotted, thick] (Z) -- (X);
    \draw[->, draw=red!60, dashed, thick] (Z) -- (Y);
    \draw[->, teal, very thick] (Y) to [out=75, in=285] (W);

    
    \node[ts_header] (t_0_b) [below=of Yt_0] {$t$};
    \node[ts_node_cause] (Xt_0_b) [below=0.2cm of t_0_b] {};
    \node[ts_node_condition] (Wt_0_b) [below=of Xt_0_b] {};
    \node[ts_node] (Zt_0_b) [below=of Wt_0_b] {};
    \node[ts_node_target] (Yt_0_b) [below=of Zt_0_b] {};

    \node[ts_header] (t_1_b) [left=of t_0_b] {$t-1$};
    \node[ts_node_cause] (Xt_1_b) [below=0.2cm of t_1_b] {};
    \node[ts_node_condition] (Wt_1_b) [below=of Xt_1_b] {};
    \node[ts_node] (Zt_1_b) [below=of Wt_1_b] {};
    \node[ts_node] (Yt_1_b) [below=of Zt_1_b] {};

    \node[ts_header] (t_2_b) [left=of t_1_b] {$t-2$};
    \node[ts_node_cause] (Xt_2_b) [below=0.2cm of t_2_b] {};
    \node[ts_node_condition] (Wt_2_b) [below=of Xt_2_b] {};
    \node[ts_node] (Zt_2_b) [below=of Wt_2_b] {};
    \node[ts_node] (Yt_2_b) [below=of Zt_2_b] {};

    \node[ts_header] (t_3_b) [left=of t_2_b] {$t-3$};
    \node[ts_node_cause] (Xt_3_b) [below=0.2cm of t_3_b] {};
    \node[ts_node_condition] (Wt_3_b) [below=of Xt_3_b] {};
    \node[ts_node] (Zt_3_b) [below=of Wt_3_b] {};
    \node[ts_node] (Yt_3_b) [below=of Zt_3_b] {};

    \node[ts_header] (t_4_b) [left=of t_3_b] {$t-4$};
    \node[ts_node_cause_control] (Xt_4_b) [below=0.2cm of t_4_b] {};
    \node[ts_node_condition] (Wt_4_b) [below=of Xt_4_b] {};
    \node[ts_node] (Zt_4_b) [below=of Wt_4_b] {};
    \node[ts_node] (Yt_4_b) [below=of Zt_4_b] {};

    \node[snode_cause] (X_b) [left=of Xt_4_b] {$X$};
    \node[snode_condition] (W_b) [left=of Wt_4_b] {$W$};
    \node[snode_latent] (Z_b) [left=3.cm of Zt_4_b] {$L$}; 
    \node[snode_target] (Y_b) [left=of Yt_4_b] {$Y$};
    \node[ts_header] (name_b) [left=of Z_b] {\textbf{b.)} };


    \draw[->, dotted] (Xt_1_b) -- (Xt_0_b);
    \draw[->, teal, thick] (Yt_1_b) -- (Yt_0_b);
    \draw[->, dotted] (Wt_1_b) -- (Wt_0_b);
    \draw[->, dashed, draw=red!60] (Zt_1_b) -- (Zt_0_b);

    \draw[->, dotted] (Xt_2_b) -- (Xt_1_b);
    \draw[->, teal, thick] (Yt_2_b) -- (Yt_1_b);
    \draw[->, dotted] (Wt_2_b) -- (Wt_1_b);
    \draw[->, dashed, draw=red!60] (Zt_2_b) -- (Zt_1_b);

    \draw[->, dotted] (Xt_3_b) -- (Xt_2_b);
    \draw[->, teal, thick] (Yt_3_b) -- (Yt_2_b);
    \draw[->, dotted] (Wt_3_b) -- (Wt_2_b);
    \draw[->, dashed, draw=red!60] (Zt_3_b) -- (Zt_2_b);

    \draw[->, dotted] (Xt_4_b) -- (Xt_3_b);
    \draw[->, lightgray] (Yt_4_b) -- (Yt_3_b);
    \draw[->, dotted] (Wt_4_b) -- (Wt_3_b);
    \draw[->, dashed, draw=red!60] (Zt_4_b) -- (Zt_3_b);

    \draw[->, dotted] (Xt_1_b) -- (Wt_0_b);
    \draw[->, lightgray] (Xt_1_b) -- (Yt_0_b);
    \draw[->, lightgray] (Wt_1_b) -- (Yt_0_b);
    \draw[->, dotted, draw=red!60] (Zt_1_b) -- (Xt_0_b);
    \draw[->, dotted] (Yt_1_b) -- (Wt_0_b);

    \draw[->, dotted] (Xt_2_b) -- (Wt_1_b);
    \draw[->, lightgray] (Xt_2_b) -- (Yt_1_b);
    \draw[->, lightgray] (Wt_2_b) -- (Yt_1_b);
    \draw[->, dotted, draw=red!69] (Zt_2_b) -- (Xt_1_b);
    \draw[->, dotted] (Yt_2_b) -- (Wt_1_b);

    \draw[->, dotted] (Xt_3_b) -- (Wt_2_b);
    \draw[->, lightgray] (Xt_3_b) -- (Yt_2_b);
    \draw[->, dotted] (Wt_3_b) -- (Yt_2_b);
    \draw[->, dotted, draw=red!60] (Zt_3_b) -- (Xt_2_b);
    \draw[->, dotted] (Yt_3_b) -- (Wt_2_b);

    \draw[->, dotted] (Xt_4_b) -- (Wt_3_b);
    \draw[->, teal, thick] (Xt_4_b) -- (Yt_3_b);
    \draw[->, lightgray] (Wt_4_b) -- (Yt_3_b);
    \draw[->, dotted, draw=red!60] (Zt_4_b) -- (Xt_3_b);
    \draw[->, dotted] (Yt_4_b) -- (Wt_3_b);

    \draw[->, dashed, draw=red!60] (Zt_2_b) -- (Yt_0_b);

    \draw[->, dashed, draw=red!60] (Zt_3_b) -- (Yt_1_b);
    
    \draw[->, dashed, draw=red!60] (Zt_4_b) -- (Yt_2_b);

    \draw[->, teal, very thick] (X_b) to [out=315, in=45](Y_b);
    \draw[->, thick, dotted] (X_b) -- (W_b);
    \draw[->, thick, lightgray] (W_b) to [out=255, in=105] (Y_b);
    \draw[->, draw=red!60, dotted, thick] (Z_b) -- (X_b);
    \draw[->, draw=red!60, dashed, thick] (Z_b) -- (Y_b);
    \draw[->, dotted, thick] (Y_b) to [out=75, in=285] (W_b);

    \end{tikzpicture}
    }
    \caption{a.) This figure illustrates the graphical definition of the CF $\ucetemp{X}{Y}$, i.e. $\ucetemp{X}{Y}(4)$. The figure shows the process graph (left) of a time series graph of which a finite subgraph is shown on the (right). The dotted red edges (latent to observed) and dotted black edges (between observed) have to be ignored for the intervention in the variables $\{\X_X(t-i)\}_{i=0, \dots, 4}$. The red dashed, full grey and teal colored edges are the edges that are not affected by the intervention, and of those edges only the full teal colored edges contribute to $\ucetemp{X}{Y}(4)$. The coefficient $\ucetemp{X}{Y}(4)$ is the sum of all path-coefficients of paths from the red node at $t-4$ to the olive node at $t$ that are composed of the teal coloured edges. On the left of the time series graph we depict its process graph whose nodes and edges admit an analogous interpretation. b.) This figure illustrates the graphical definition of the CCF $\cetemp{X}{Y}{W}(4)$, specifically, the entry $\cetemp{X}{Y}{W}(4)$. The purple coloured nodes are those states of the process $W$ on which we condition. All other edges and nodes admit the same interpretation as in a.)}
    \label{fig: CCF example}
\end{figure}
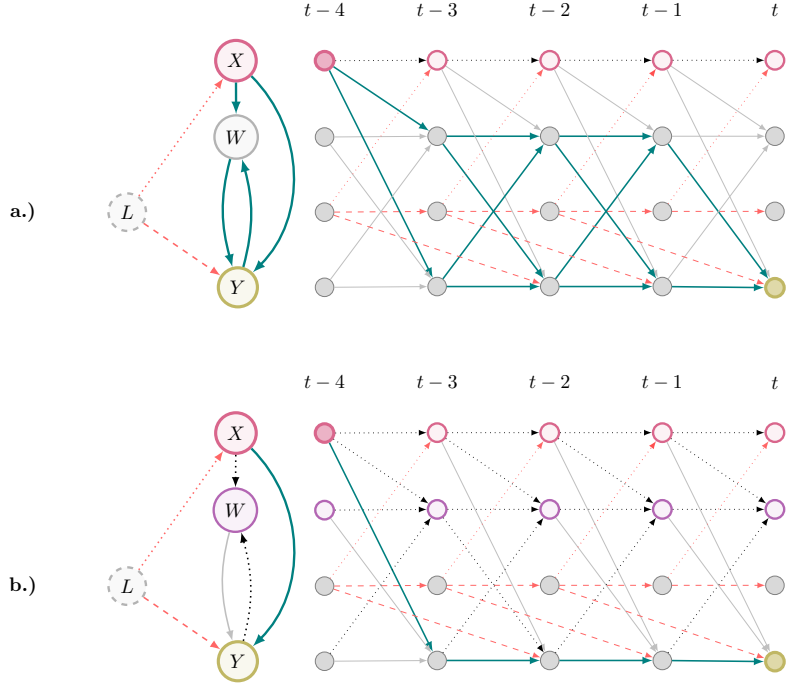
    In Figure \ref{fig: CCF example} we illustrate the definition of CCFs and CFs on a specific time series graph. The system is composed of three observed processes $X,W,Y$ and one latent process $L$. The time series graph is similar to the one in Figure \ref{fig:time_series_graph}, the only difference being a link from $Y$ to $W$. By means of that additional link there is a feed-back loop between these two processes. 
    
    Figure \ref{fig: CCF example}.a) is a graphical illustration of the causal effect $\ucetemp{X}{Y}(4)$. According to Proposition \ref{prop: causal effects process graph} we delete all the links pointing to any of the nodes in $(X(s))_{s \in \Z}$. In Figure \ref{fig: CCF example}.a) we illustrate this by rendering all these deleted links dotted. We obtain the coefficient $\ucetemp{X}{Y}(4)$ by summing the path-coefficients of all directed paths from $X(t-4)$ to $Y(t)$ that do not go through $(X(t-i))_{0 \leq i \leq 3}$. These are all the causal path-ways from $X(t-4)$ to $Y(t)$ on which each link is one of the teal colored links in Figure \ref{fig: CCF example}.a.). 

    Figure \ref{fig: CCF example}.b) provides a graphical interpretation of the CCF $\cetemp{X}{Y}{W}$, specifically, its fourth element $\cetemp{X}{Y}{W}(4)$. In addition to deleting all links that point to any node in $(X(s))_{s \in \Z}$ we delete all links pointing to the nodes in $(W(s))_{s \in \Z}$. These removed links are the dotted links in Figure \ref{fig: CCF example}.b). The coefficient $\cetemp{X}{Y}{W}(4)$ is the sum of path-coefficients of all the paths going from $X(t-4)$ to $Y(t)$ without passing through any $(X(s))_{s \in \Z}$ or $(W(s))_{s \in \Z}$. These are all directed paths that can be composed of the teal colored links in Figure \ref{fig: CCF example}.b.). Note that there is exactly one path from $X(t-4)$ to $Y(t)$ along the teal colored links. 
    
\end{example}

\subsection{Trek rule for process graphs} \label{subsectio: trek-rule}
In Section \ref{section: Preliminaries}, we have mentioned the trek-rule and its importance to several research contributions on SEMs. It is therefore relevant to observe that combining the SEP-representation of a SVAR process with the generalised path-rule (\ref{eq: generalised path-rule}) generalises the trek-rule to the SEP-representation at the level of the projected process graph $G'$. To spell this out, we define the \textit{trek-monomial-filter} of a chosen trek $\pi$ on $G'$ to be 
\begin{align*}
    \C^{(\pi)} &\coloneqq \begin{cases}
        \Lambda^{(\trekL(\pi))} \ast \C_V^\linternal \tcvl \Lambda^{(\trekR(\pi))} & \text{if $\pi$ is of form (\ref{trek-type I})} \\
        \Lambda^{(\trekL(\pi))} \ast \C_{V_0^L, V_0^R}^\linternal \tcvl \Lambda^{(\trekR(\pi))} & \text{if $\pi$ is of form (\ref{trek-type II})}
    \end{cases}.
\end{align*}
\begin{proposition}[Trek-rule for processes] \label{prop: trek-rule processes}
The cross-covariance sequence between two observed processes $V,W \in \mathbf{O}$ is given by the sum over all trek-monomial-filters that are associated with the treks from $V$ to $W$, i.e.,
\begin{align*}
    \C_{V,W} &=\sum_{\pi \in \mathcal{T}(V, W)} \C^{(\pi)}.
\end{align*}
\end{proposition}
\begin{proof}
This follows from Proposition \ref{prop: summary SEM}, the generalised path-rule (\ref{eq: generalised path-rule}) and applying the same arguments as in the proof of \cite[Theorem 4.2]{drton2018algebraic}.  
\end{proof}

\section{Reasoning about structural equation processes in the frequency domain} \label{section: frequency domain}
In this section, we transfer the concepts and observations we developed in the previous section to the frequency domain. This turns out to be convenient, as applying the Fourier-Transformation makes the expressions we encountered in the last section simpler and more compact. For the remainder of this section we fix a latent-component time series graph $\mathcal{G}$ and a SVAR-process specified by $(\Phi, \mathbf{w})$ that is consistent with $\mathcal{G}$ and that satisfies the assumptions stated in Theorem \ref{prop: summary SEM}. As before, we denote the process graph of $\mathcal{G}$ by $G$ and its latent projection by $G'$.

\subsection{The frequency domain SEP representation} Recall that the \textit{Fourier-Transformation} of an absolutely summable filter $\lambda = (\lambda(s))_{s \in \Z}$ is the complex-valued function $\mathcal{F}(\lambda) :[0, 2\pi) \to \mathbb{C}$ defined by 
\begin{displaymath}
    \mathcal{F}(\lambda)(\omega) \coloneqq \sum_{s \in \mathbb{Z}} \lambda(s) \exp(-i \omega s).
\end{displaymath}  
It is well known that $\mathcal{F}(\lambda \ast \mu) = \mathcal{F}(\lambda) \mathcal{F}(\mu)$, where the right-hand side is point-wise multiplication of functions. 
If $\Lambda = (\Lambda(s))_{s \in \mathbb{Z}}$ is a matrix-valued filter, then $\mathcal{F}(\Lambda)$ is the matrix-valued function whose $(i,j)$-th entry is $\mathcal{F}(\Lambda_{i,j})$, i.e., the Fourier-Transformation of the filter whose values are the $(i,j)$-th entries of $\Lambda$. Accordingly, if $\Lambda$ is a $p \times n$-dimensional filter and $\Gamma$ a $n \times m$-dimensional filter, then we have that $\mathcal{F}(\Lambda \ast \Gamma) = \mathcal{F}(\Lambda)\mathcal{F}(\Gamma)$, where the right hand-side is point-wise matrix multiplication. Similarly, Fourier-Transformation takes tilted convolutions to multiplication by complex conjugate, that is, $\mathcal{F}(\Lambda \tcvl \Gamma) = \mathcal{F}(\Lambda) \mathcal{F}(\Gamma) ^\ast$.
\begin{definition}[Direct effect-function]
Let $V,W \in \mathbf{O}$ be a pair of observed processes and $L \in \mathbf{L}$ a latent process, then the \textit{direct effect-function} of $V$ on $W$ resp. $L$ on $V$ is defined as the Fourier-Transformation of the direct effect-filter $\dirtmp{V}{W}$ resp. $\ltmp{L}{W}$, that is,
\begin{align} \label{eq: dir_frq_effect}
    \dirfrq{V}{W} &\coloneqq \mathcal{F}(\dirtmp{V}{W}), & \lfrq{L}{V} &\coloneqq \mathcal{F}(\ltmp{L}{V})\,.
\end{align}
\end{definition}
One advantage of switching to the frequency domain is that the direct effect-functions can be written compactly as rational functions on the complex unit circle whose coefficients are directly related to those of the SVAR parameters $\Phi$. 
\begin{lemma}\label{lemma: analytical expression transfer-function}
    Let $V,W$ be a pair of processes, then the transfer-function $\dirfrq{V}{W}$ as defined in equation (\ref{eq: dir_frq_effect}) is the restriction of a rational polynomial to the complex unit circle, namely 
    \begin{align}\label{eq: direct transfer function}
        \dirfrq{V}{W}(\omega) &= \frac{\sum_{k =0}^p \phi_{V, W}(k) \exp(-i\omega k)}{1- \sum_{k=1}^p \phi_{V, V}(k) \exp(-i\omega k)}.  
    \end{align}
\end{lemma}
\begin{proof}
    Let $V,W \in \mathbf{V}$ and let us abbreviate the coefficients $\phi_{V,W}(k)$ by $b_k$, and $\phi_{W, W}(k)$ by $a_k$. Recall that we defined the link filter $\dirtmp{V}{W}$ recursively as follows
    \begin{align*}
        \dirtmp{V}{W}(k) &= \begin{cases}
            \sum_{j=1}^k \dirtmp{V}{W}(k -j) a_j + b_k & \text{if $0 \leq k \leq p$} \\
            \sum_{j=1}^p \dirtmp{V}{W}(k-j) a_j & \text{if $k > p$}
        \end{cases}.
    \end{align*}
    The formal power series associated to $\dirtmp{V}{W}$ is defined as follows
    \begin{align*}
        h(z) &= \sum_{k \in \Z} \dirtmp{V}{W}(k) \cdot z^k \\
        &= \sum_{k=0}^p \dirtmp{V}{W}(k) \cdot z^k + \sum_{k > p} \dirtmp{V}{W}(k) \cdot z^k \\
        &= \sum_{k= 0 }^p z^k b_k + \sum_{k = 1}^{p} a_k \cdot z^k \cdot h(z) ,
    \end{align*}
    and therefore, 
    \begin{align*}
        h(z) &= \frac{\sum_{k=0}^p b_k \cdot z^k}{1 - \sum_{k =1}^p a_k \cdot z^k }.
    \end{align*}
    Finally, we obtain the link-function $\dirfrq{V}{W}(\omega)$ by setting $z = \exp(-i \omega)$ for which it is well defined because of condition (\ref{condition: stability II}). 
\end{proof}

In the previous section, we assigned to a SVAR process $(\Phi, \mathbf{w})$ its SEP-representation $(\Dirtmp, \C_\mathbf{O}^\linternal)$ with which we could express the ACS of the observed processes $\mathbf{O}$, see equation (\ref{eq: corss covariance process graph}). We now apply the Fourier-Transformation to that pair and to the expression for the ACS of $\mathbf{O}$.

The spectral density of the process $\mathbf{O}$, resp. of their (projected) internal dynamics, is defined to be the Fourier-Transformation of the ACS $\C_\mathbf{O}$, resp. $\C_\mathbf{O}^\internal$ and $\C_\mathbf{O}^\linternal$, namely 
\begin{align*}
    \Sp{\mathbf{O}}& \coloneqq \mathcal{F}(\C_\mathbf{O}), & \Sp{\mathbf{O}}^\internal &\coloneqq \mathcal{F}(\C_\mathbf{O}^\internal), & \Sp{\mathbf{O}}^\linternal &\coloneqq \mathcal{F}(\C_\mathbf{O}^\linternal),
\end{align*}
where the Fourier-Transformation is applied element-wise. The spectral density $\Sp{\mathbf{O}}$ is hermitian, that is, $\Sp{\mathbf{O}} = (\Sp{\mathbf{O}}^\ast)^\top$, where $^\ast$ denotes element-wise complex conjugation. 

We call the Fourier-Transformation $\Dirfrq\coloneqq \mathcal{F}(\Dirtmp)$ the \textit{direct effect-function}. We denote the identity matrix as $\mathcal{I}$, its dimension will be clear from the context. When the Fourier-Transformation is applied to equation (\ref{eq: corss covariance process graph}) one obtains the following expression for the spectral density of the observed processes
\begin{equation} \label{eq: frq sep cross spectrum}
    \begin{split}
        \mathcal{S}_\mathbf{O} &= (\Dirfrq^\infty)^\top \mathcal{S}_\mathbf{O}^\linternal (\Dirfrq^\infty)^\ast \\
        &= (\sum_{k \geq 0} \Dirfrq^k)^\top \mathcal{S}_\mathbf{O}^\linternal (\sum_{k \geq 0} \Dirfrq^k)^\ast \\
        &= (\mathcal{I} - \Dirfrq)^{-\top} \mathcal{S}_\mathbf{O}^\linternal (\mathcal{I} - \Dirfrq)^{-\ast}.
    \end{split}
\end{equation}
The spectral density of the projected internal dynamics $\Sp{\mathbf{O}}^\linternal$ splits into two components, one of which is due to the internal dynamics $\X_\mathbf{O}^\linternal$ and the other one of which is due to the influence of the latent factors. To formulate this decomposition, let us denote the transfer-function associated with the direct effects of the latent processes on the observed processes by $\mathfrak{J} \coloneqq \mathcal{F}(\Gamma)$, so that the spectral density of the projected internal dynamics takes the following form 
\begin{equation} \label{eq: frq sep cross spectrum projected internal}
    \begin{split}
        \mathcal{S}_\mathbf{O}^\linternal &= \mathcal{S}_\mathbf{O}^\internal + \mathfrak{J}^\top \mathcal{S}_{\mathbf{L}} \mathfrak{J}^\ast \\
        &= \mathcal{S}_\mathbf{O}^\internal + \sum_{L \in \mathbf{L}} \mathcal{S}_L \mathfrak{J}^\top_{L \to \mathbf{O}} \mathfrak{J}_{L \to \mathbf{O}}^\ast. 
    \end{split}
\end{equation}
We refer to the tuple $(\Dirfrq, \Sp{\mathbf{O}}^\linternal)$ as the \emph{frequency domain SEP representation} of the observed process $\X_\mathbf{O}$. In the Appendix \ref{sec: commutative diagram}, we spell out the mathematical structures in which the time and frequency SEP representations for a given latent factor process graph $G$ reside. This formalisation makes it explicit that the SEP representations entail the SEM parameterisations of the latent projection $G'$. We organise these formulations and observations using a commutative diagram in the Appendix in Figure \ref{fig: summary main results}.

\begin{example}[The spectral density structure in an instrumental process setting, see Figure \ref{fig:latent_factor_summary_graph}] \label{ex: instrumental process}
    In this example, we spell out the frequency domain dependency structure among the process that are related to each other as depicted in the process graph of Figure \ref{fig:latent_factor_summary_graph}. Specifically, we describe the entries in the spectral density matrix. To this end, we start with the direct transfer function 
    \begin{align*}
        \Dirfrq = \begin{pmatrix}
            0 & \dirfrq{X}{M} & 0 \\
            0 & 0 & \dirfrq{M}{Y} \\
            0 & 0 & 0 
        \end{pmatrix},
    \end{align*} so that  
    \begin{align*}
        (\mathcal{I}-\Dirfrq)^{-1} = \sum_{k=0}^\infty \Dirfrq^k = \begin{pmatrix}
            1 & \dirfrq{X}{M} & \dirfrq{X}{M} \dirfrq{M}{Y} \\
            0 & 1 & \dirfrq{M}{Y} \\
            0 & 0 & 1 
        \end{pmatrix}.
    \end{align*}
    The transfer function specifying the direct effects of the latent process on the observed processes is 
    \begin{align*}
        \mathcal{J} = \begin{pmatrix}
            0 &
            \lfrq{L}{M} &
            \lfrq{L}{Y}
        \end{pmatrix}. 
    \end{align*}
    We apply formula (\ref{eq: frq sep cross spectrum projected internal}) to obtain the cross-dependencies among the projected internal dynamics $\X_\mathbf{O}^\linternal$ in the frequency domain, i.e.,   
    \begin{align*}
        \Sp{\mathbf{O}}^\linternal &= \begin{pmatrix}
            \Sp{X}^\internal & 0 & 0 \\
            0 & \Sp{M}^\internal + |\mathcal{J}_{L \to M}|^2 \Sp{L} & \mathcal{J}_{L \to M} \mathcal{J}_{L \to Y}^\ast \Sp{L} \\
            0 & \mathcal{J}_{L \to Y} \mathcal{J}_{L \to M}^\ast \Sp{L} & \Sp{Y}^\internal + |\mathcal{J}_{L \to Y}|^2 \Sp{L}
        \end{pmatrix}
    \end{align*}
    Using formula (\ref{eq: frq sep cross spectrum}), we get for the diagonal entries in the spectral density of the observed processes the expressions
    \begin{align*}
        \Sp{X} &= \Sp{X}^\internal\\
        \Sp{M} &= |\dirfrq{X}{M}|^2 \Sp{X}^\internal + \Sp{M}^\linternal \\
        \Sp{Y} &= |\dirfrq{X}{M}\dirfrq{M}{Y}|^2 \Sp{X}^\internal + |\dirfrq{M}{Y}|^2\Sp{M}^\linternal + 2 \mathrm{Re}(\dirfrq{M}{Y}^\ast \Sp{Y,M}^\linternal ) + \Sp{Y}^\linternal,
    \end{align*}
    and for the off-diagonal entries we obtain 
    \begin{align*}
        \Sp{X, M} &= \dirfrq{X}{M}^\ast \Sp{X}^\internal\\
        \Sp{X, Y} &= \dirfrq{X}{M}^\ast\dirfrq{M}{Y}^\ast \Sp{X}^\internal \\
        \Sp{M, Y} &= |\dirfrq{X}{M}|^2 \dirfrq{M}{Y}^\ast \Sp{X}^\internal + \dirfrq{M}{Y}^\ast \Sp{M}^\linternal + \Sp{M, Y}^\linternal.
    \end{align*}
    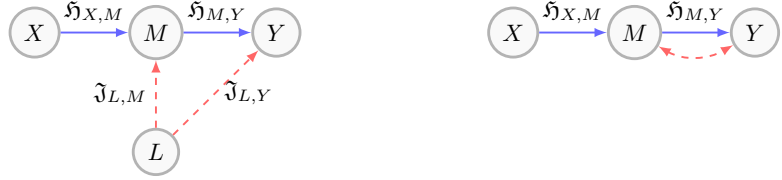
\begin{figure}
    \centering
    \resizebox{\textwidth}{!}{
    \begin{tikzpicture}[
       ts_node/.style={circle, draw=gray, fill=gray!60, thick}, 
       header/.style={rectangle, minimum size=1cm, thick},
       snode/.style={circle, draw=gray!60, fill=gray!5, very thick}
    ]

    \node[header] (summary) at (0,0) {a.) $G=(\mathbf{O} \cup \mathbf{L}, D^\mathbf{O} \cup D^{\mathbf{L}, \mathbf{O}})$};
    \node[snode] (X) [below=0.4cm of summary] {$X$};
    \node[snode] (M) [right=of X] {$M$};
    \node[snode] (L) [below=of M] {$L$};
    \node[snode] (Y) [right=of M] {$Y$};  

    \node[header] (summary_l) [right=3.0cm of summary] {b.) $G' =(\mathbf{O}, D^\mathbf{O}, B)$};
    \node[snode] (X_l) [below=0.4cm of summary_l] {$X$};
    \node[snode] (M_l) [right=of X_l] {$M$};
    \node[snode] (Y_l) [right=of M_l] {$Y$};  

    \draw[->, thick, draw=blue!60] (X) -- node[above] {$\dirfrq{X}{M}$} (M);
    \draw[->, thick, draw=blue!60] (M) -- node[above] {$\dirfrq{M}{Y}$} (Y);
    \draw[->, dashed, thick, draw=red!60] (L) -- node[left] {$\lfrq{L}{M}$} (M); 
    \draw[->, dashed, thick, draw=red!60] (L) -- node[right] {$\lfrq{L}{Y}$} (Y);

    \draw[->, thick, draw=blue!60] (X_l) -- node[above] {$\dirfrq{X}{M}$} (M_l);
    \draw[->, thick, draw=blue!60] (M_l) -- node[above] {$\dirfrq{M}{Y}$} (Y_l);
    \draw[<->, dashed, thick, red!60] (M_l) edge[out=330, in=210] (Y_l);

    \end{tikzpicture}}
    \caption{This figure depicts the process graph of a latent factor-time series graph which we leave unspecified. The observed processes are $\mathbf{O} = \{X,M,Y \}$, and there is one latent process $L$ which influences both $M$ and $Y$.  b.) The latent projection $G'$ of the graph $G$ shown in a.), where the latent confounding of $M$ and $Y$ by the latent process $L$ is represented by a bidirectional edge between $M$ and $Y$.}
    \label{fig:latent_factor_summary_graph}
\end{figure}
\end{example}

\begin{remark}[On identifying transfer-functions in an instrumental process setting, see Figure \ref{fig:latent_factor_summary_graph}]
    Example \ref{ex: instrumental process} is the generic example of a instrumental process setting \cite{thams2022instrumental, mogensen2022instrumental} at the level of the process graph. Although the link $M \to Y$ is confounded by $L$, one can, in principle, recover the function $\dirfrq{M}{Y}$ from the spectral density of the observed processes. To this end, we start by assuming that $(\Phi, \mathbf{w})$ specifies a SVAR-process of order one that is consistent with the process graph $G$ from Figure \ref{fig:latent_factor_summary_graph}. According to the relations from Example \ref{ex: instrumental process}, the link-functions can be calculated from the observed spectral density as:  
    \begin{align*}
        \dirfrq{X}{M} &= \frac{\Sp{M,X}}{\Sp{X}} & \dirfrq{M}{Y} &= \frac{\Sp{Y,X}}{\Sp{M,X}} = \frac{\Sp{Y,X}}{\dirfrq{X}{M}\Sp{X}}.
    \end{align*}
    Since $\Sp{X}$ is non-zero at every frequency $\omega$ due to the stability condition (\ref{condition: stability II}), the first of these equations shows that $\dirfrq{X}{M}$ is always identifiable. The direct transfer-function $\dirfrq{M}{Y}$ is identifiable for every frequency $\omega \in [0, 2 \pi)$ if the function $\dirfrq{X}{M}$ is non-zero at every frequency. To undertsand when $\dirfrq{X}{M}$ is non-zero, recall that $\dirfrq{X}{M}$ is a rational polynomial restricted to the unit circle with numerator
    \begin{align*}
        \phi_{X, M}(0) + \phi_{X, M}(1)z.
    \end{align*}
    There are two cases in which this polynomial is zero. First, $\dirfrq{X}{M}$ is zero at every frequency if $\phi_{X, M}(k)$ is zero for every $k\in \{0,1\}$. Second, $\dirfrq{X}{M}$ is zero if both $\phi_{X,M}(0)$ and $\phi_{X,M}(1)$ are non-zero and 
    \begin{align*}
        \exp(i \omega) &= - \frac{\phi_{X,M}(1)}{\phi_{X, M}(0)}, \text{ and $\omega \in [0, 2\pi)$}.  
    \end{align*}
    Since both coefficients are real, this equality can only hold if $\omega \in \{ 0, \pi\}$, that means $\phi_{X,M}(0) = \pm \phi_{X, M}(1)$. Similar arguments show that for a general SVAR-process of order $p$, the link-function $\dirfrq{M}{Y}$ is not directly identifiable for at most $p$ frequencies. However, if the function  $\dirfrq{X}{M}$ vanishes on only finitely many frequencies $\omega_i$, then we can approximate $\dirfrq{M}{Y}$ arbitrarily well as the limit
    \begin{align*}
        \dirfrq{M}{Y}(\omega_i) &= \lim_{\omega \to \omega_i} \frac{\Sp{Y,X}(\omega)}{\Sp{M,X}(\omega)}.  
    \end{align*}
    This means that the only case in which we cannot infer anything about the transfer-function $\dirfrq{M}{Y}$ is when $\phi_{X, M}(k)=0$ for all $0 \leq k \leq p$, i.e., when $X$ is in fact not an instrumental process for determining the causal effect of $M$ on $Y$.  
\end{remark}

\subsection{Causal effects between processes in the frequency domain} 
Let $X,Y \in \mathbf{O}$ be two observed processes and $\mathbf{Z}\subset \mathbf{O}$ a collection of observed processes. In the previous section, we introduced the filter $\cetemp{\mathbf{X}}{Y}{\mathbf{Z}}$ as the object that quantifies the causal influence of $X$ on $Y$ along all causal path-ways on $G'$ that do not pass through $\mathbf{Z}$ or revisit $X$. In this subsection, we assume that the CCF $\cetemp{X}{Y}{\mathbf{Z}}$ exists and is absolutely summable. This is true, for example, if the SVAR parameter $\Phi$ satisfies the assumption spelled out in the Appendix in Proposition \ref{prop: supplement sep existence condition}.

Suppose $\pi$ is a directed path on $G'$, then we define its \textit{path-function} to be the Fourier-Transformation of its path-filter, i.e., $\Dirfrq^{(\pi)} \coloneqq \mathcal{F}(\Lambda^{(\pi)})$, which is the point-wise product of the
direct effect-functions associated with the links that constitute $\pi$. We call the Fourier-Transformation of $\cetemp{X}{Y}{\mathbf{Z}}$ the \textit{controlled causal transfer function (CCTF)} and denote it as 
\begin{align*}
    \cefrq{X}{Y}{\mathbf{Z}} &\coloneqq \mathcal{F}(\cetemp{X}{Y}{\mathbf{Z}})= \sum_{\pi \in \mathrm{P}_{\mathbf{Z}\cup \{X \}}(X,Y)} \Dirfrq^{(\pi)}.
\end{align*}
If $\mathbf{Z} = \emptyset$, then we abbreviate the CCTF by $\ucefrq{X}{Y}$ and call it the \textit{causal transfer function (CTF)}. Furthermore, applying the Fourier-Transformation to equation (\ref{eq: generalised path-rule}) gives Wright's path-rule in the frequency domain for the process graph.
Note that the CTF $\cefrq{X}{Y}{\mathbf{Z}}$ can be decomposed into its modulus- and phase-function, that is, 
\begin{align*}
    \cefrq{X}{Y}{\mathbf{Z}} &= r_{X \to Y | \Do(\mathbf{Z})=0} \exp(i \theta_{X \to Y | \Do(\mathbf{Z})=0}).
\end{align*}
The modulus function $r_{X \to Y | \Do(\mathbf{Z})=0}$ provides a controlled causal effect quantification in the frequency domain.
\begin{example}[Path rule in the frequency domain]
    \begin{figure}
    \centering
    \resizebox{0.8\textwidth}{!}{
    \begin{tikzpicture}[
       ts_node/.style={circle, draw=gray, fill=gray!60, thick}, 
       header/.style={rectangle, minimum size=1cm, thick},
       snode/.style={circle, draw=gray!60, fill=gray!5, very thick}
    ]

    \node[header] (a) at (0,0) {a.) $G_a$};
    \node[snode] (X_a) [below=0.2cm of a] {$X$};
    \node[snode] (M_a) [below=of X_a] {$M$};
    \node[snode] (Y_a) [below=of M_a] {$Y$};  

    \node[header] (b) [right=3.cm of a] {b.) $G_b$};
    \node[snode] (Z_b) [below=0.2cm of b] {$Z$};
    \node[snode] (X_b) [below= of Z_b] {$X$};
    \node[snode] (Y_b) [below=of X_b] {$Y$};
    \node[snode] (M_b) [right=of X_b] {$M$};

    \node[header] (c) [right=3.cm of b] {c.) $G_c$};
    \node[snode] (Z_c) [below=0.2cm of c] {$Z$};
    \node[snode] (X_c) [below=of Z_c] {$X$};
    \node[snode] (Y_c) [right=of X_c] {$Y$};

    \draw[->, draw=blue!60, thick] (X_a) -- node[left] {$\dirfrq{X}{M}$} (M_a);
    \draw[->, draw=blue!60, thick] (M_a) -- node[left] {$\dirfrq{M}{Y}$} (Y_a);

    \draw[->, draw=blue!60, thick] (Z_b) -- node[left] {$\dirfrq{Z}{X}$} (X_b);
    \draw[->, draw=blue!60, thick] (Z_b) -- node[right] {$\dirfrq{Z}{M}$} (M_b);
    \draw[->, draw=blue!60, thick] (Z_b) to [out=190, in=170] node[left] {$\dirfrq{Z}{Y}$} (Y_b);
    \draw[->, draw=blue!60, thick] (X_b) -- node[left] {$\dirfrq{X}{Y}$} (Y_b);
    \draw[->, draw=blue!60, thick] (X_b) -- node[above] {$\dirfrq{X}{M}$} (M_b);
    \draw[->, draw=blue!60, thick] (M_b) -- node[right] {$\dirfrq{M}{Y}$} (Y_b);

    \draw[->, draw=blue!60, thick] (Z_c) -- node[left] {$\dirfrq{Z}{X}$} (X_c);
    \draw[->, draw=blue!60, thick] (X_c) to [out=30, in=150] node[above] {$\dirfrq{X}{Y}$} (Y_c);
    \draw[->, draw=blue!60, thick] (Y_c) to [out=210, in = 330] node[below] {$\dirfrq{Y}{X}$} (X_c);
    \draw[->, draw=blue!60, thick] (Z_c) to [out=0, in=90] node[right] {$\dirfrq{Z}{Y}$} (Y_c);
    \end{tikzpicture}}
    \caption{This figure shows three example process graphs. In the Appendix \ref{section: paramterisations} we specify for each of the three process graphs $G_{a}, G_b, G_c$ a SVAR process. The edges in each of the process graphs are parameterised by the functions as expressed in Lemma \ref{lemma: analytical expression transfer-function}.}
    \label{fig:summary_graphs_demo}
\end{figure}
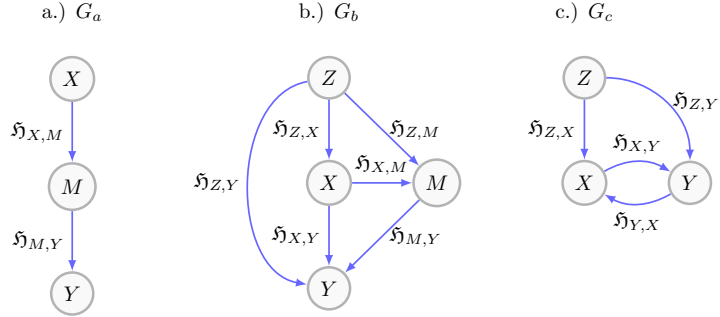 
    Let us illustrate the frequency domain causal effects on the process graphs in Figure \ref{fig:summary_graphs_demo}.
    
    For the process graph $G_a$, the transfer function associated to the causal effect of $X$ on $Y$ is the product of the transfer-functions parameterising the links $X\to M$ and $M \to Y$, i.e. $\ucefrq{X}{Y} = \dirfrq{X}{M} \dirfrq{M}{Y}$, as visualised in Figure \ref{fig:path_rule_figure_a}.
    \begin{figure}
        \centering
        \includegraphics[width=\textwidth]{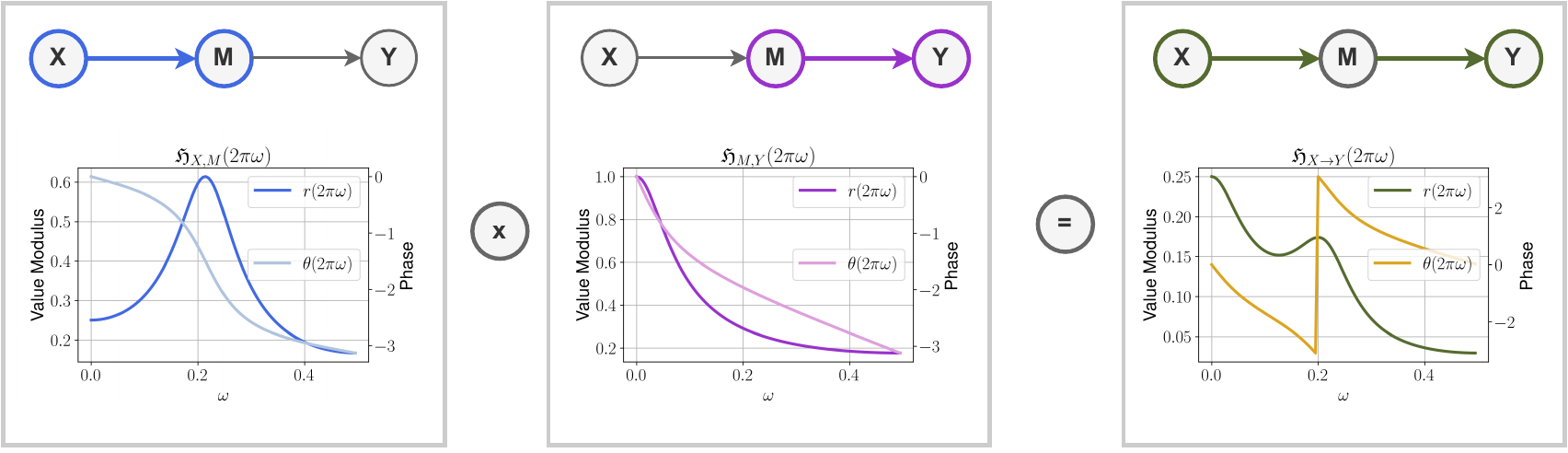}
        \caption{This figure illustrates the path-rule on the process graph $G_a$ from Figure \ref{fig:summary_graphs_demo}. Each of the three figures shows a directed path in $G_a$ and below its associated transfer-function. We plotted both its modulus $r()$ and its phase $\theta()$. The modulus $r()$ is plotted in the same color as the path to which the modulus belongs. The specific forms of the transfer-functions arise from the SVAR process defined in the Appendix \ref{section: paramterisations}.}
        \label{fig:path_rule_figure_a}
    \end{figure}
    
    For the process graph $G_b$, the transfer function associated to the causal effect of $X$ on $Y$, i.e. $\ucefrq{X}{Y}$, is given by the sum $\dirfrq{X}{Y} + \dirfrq{X}{M}\dirfrq{M}{Y}$. Its squared modulus takes the following form
    \begin{align*}
        r_{X\to Y}^2 &= |\dirfrq{X}{Y}|^2 + |\dirfrq{X}{M}\dirfrq{M}{Y}|^2 + 2 \mathrm{Re}(\dirfrq{X}{Y}^\ast \dirfrq{X}{M} \dirfrq{M}{Y}).
    \end{align*}
     In Figure \ref{fig:path_rule_figure_b} we illustrate this decomposition on the SVAR process specified in Appendix \ref{section: paramterisations}. 
    \begin{figure}
        \centering
        \includegraphics[width=\textwidth]{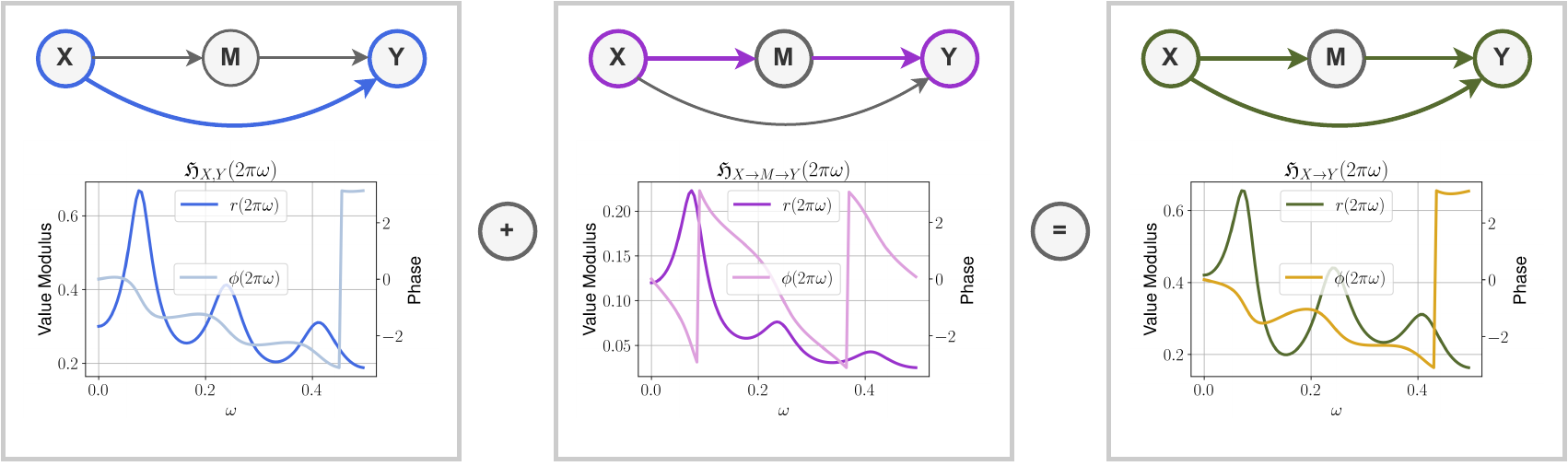}
        \caption{This figure illustrates the path-rule on the process graph $G_b$ for the causal effect $X \to Y$ which is the sum (green) of the transfer-functions associated to the direct (blue) and the indirect effect (violet). The plots follow the same convention as in Figure \ref{fig:path_rule_figure_a}. The specific form of the functions arise from the SVAR process defined in Appendix \ref{section: paramterisations}.}
        \label{fig:path_rule_figure_b}
    \end{figure}
    
    For the process graph $G_c$, we compute the causal effect function for the effect of $Z$ on $Y$. Note that, since there is the cycle $Y \to X \to Y$, a link from $Z$ to $X$ and a link from $Z$ to $Y$, the set of paths $\mathrm{P}_{Z}(Z,Y) = \mathrm{P}(Z,Y)$ is infinite. This set is can be written as the union of sets
    \begin{align*}
        \mathrm{P}(Z,Y) &= \{ Z \to X \to Y \rightleftarrows^{\times l} X \}_{l \geq 0} \\
        &\cup \{ Z \to Y \to X \to Y \rightleftarrows^{\times k} X \}_{k \geq 0}, 
    \end{align*}
    where, for example, $Z \to X \to Y \rightleftarrows^{\times 2}X$ represents the path $Z \to X \to Y \to X \to Y \to X \to Y $. In the following we abbreviate the product $\dirfrq{Y}{X} \dirfrq{X}{Y}$ by $\circfrq{Y}{X}$. Hence, the function capturing the causal effect of $Z$ on $Y$ is
    \begin{align*}
        \ucefrq{Z}{Y} &= \dirfrq{Z}{X}\dirfrq{X}{Y}\sum_{k \geq 0} \circfrq{Y}{X}^k + \dirfrq{Z}{Y}\sum_{l \geq 0} \circfrq{Y}{X}^l\\
        &= \dirfrq{Z}{X}\dirfrq{X}{Y} (1-\circfrq{Y}{X})^{-1} + \dirfrq{Z}{Y}(1- \circfrq{Y}{X})^{-1} \\
        &= \frac{\dirfrq{Z}{X}\dirfrq{X}{Y} + \dirfrq{Z}{Y}}{1- \circfrq{Y}{X}},
    \end{align*}
    if $|\circfrq{Y}{X}(\omega)| < 1$ for all $\omega \in [0, 2\pi)$. Then its squared modulus is 
    \begin{align*}
        |r_{Z \to Y}|^2 &= \frac{|\dirfrq{Z}{X} \dirfrq{X}{Y}|^2 + |\dirfrq{Z}{Y}|^2 + 2\mathrm{Re}(\dirfrq{Z}{X}\dirfrq{X}{Y}\dirfrq{Z}{Y}^\ast)}{|1 - \circfrq{Y}{X}|}. 
    \end{align*}
    In Figure \ref{fig:path_rule_figure_c} we illustrate this transfer function on the SVAR defined in Appendix \ref{section: paramterisations}.
    \begin{figure}
        \centering
        \includegraphics[width=\textwidth]{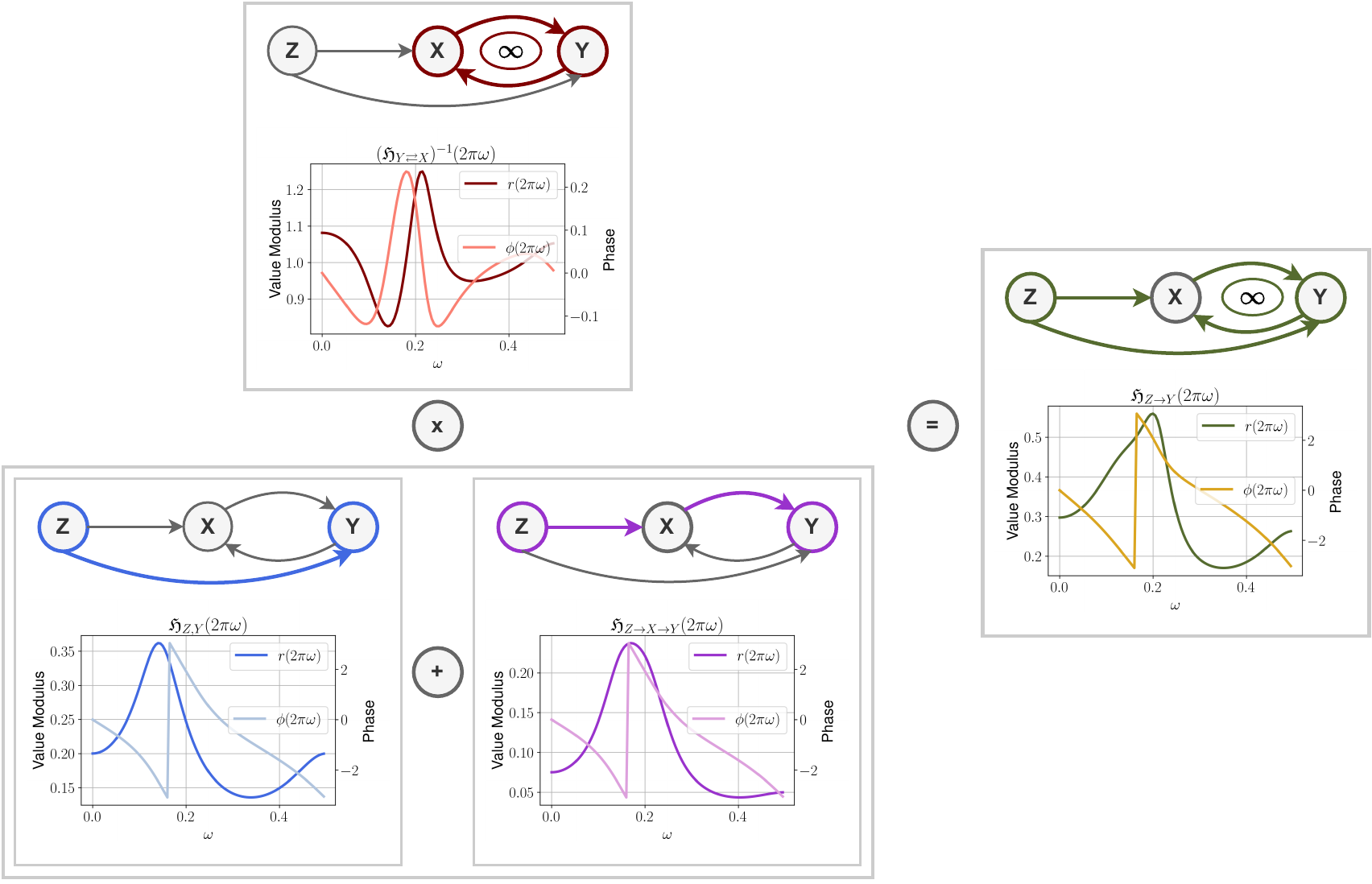}
        \caption{This Figure illustrates the path-rule for the cyclic process graph $G_c$ in Figure \ref{fig:summary_graphs_demo} for the causal effect of $Z$ on $Y$. The transfer-function in the box on the top is $(1- \circfrq{Y}{X})^{-1}$. The transfer-function associated to the causal effect of $Z$ on $Y$ (green) is given by multiplying the sum of the two transfer-functions on the bottom with the one on top. The specific form of the transfer-functions arise from evaluating the expression (\ref{eq: direct transfer function}) on the coefficients of the SVAR processes defined in Appendix \ref{section: paramterisations}.}
        \label{fig:path_rule_figure_c}
    \end{figure}
\end{example}
\subsection{Trek rule in the frequency domain}
Let $\pi$ be a trek on the latent projection $G'$. Then we define the \textit{trek-monomial-function} associated with $\pi$ to be the Fourier-Transformation of the corresponding trek-monomial-filter, i.e., 
\begin{align*}
    \Sp{}^{(\pi)} &\coloneqq \mathcal{F}(\C^{(\pi)}) = \begin{cases}
        \Dirfrq^{(\trekL(\pi))} \Sp{V}^\linternal  (\Dirfrq^{(\trekR(\pi))})^\ast & \text{if $\pi$ is of form (\ref{trek-type I})} \\
        \Dirfrq^{(\trekL(\pi))} \Sp{V_0^L, V_0^R}^\linternal (\Dirfrq^{(\trekR(\pi))})^\ast & \text{if $\pi$ is of form (\ref{trek-type II})}
    \end{cases}.
\end{align*} An immediate implication is that the generalised trek-rule as stated in Proposition \ref{prop: trek-rule processes} has a corresponding formulation in the frequency domain. Namely, the entry in the spectral density corresponding to a pair of process $V,W \in \mathbf{O}$ is given as the sum of trek-monomial-functions
\begin{align}\label{eq: trek-rule frequency domain}
    \Sp{V,W} = \sum_{\pi \in \T{}(V,W)} \Sp{}^{(\pi)}.
\end{align}

In particular, the trek-rule lets us quantify how a process $V$ contributes to the spectral density of another process $W$. The contribution of $V$ to the variability of $W$ is the sum over the trek-monomial-functions of those treks $\pi \in \T{}(W,W)$ for which $V \in \trekL(\pi) \cup \trekR(\pi)$. To this set of treks we will refer as $\T{V}(W,W)$. The part of the spectrum of $W$ to which $V$ contributes, decomposes into a causal $\spcausal{V}{W}$ and a confounding factor $\spconf{V}{W}$. Finally, we represent as $\spresidual{V}{W}$ the sum of trek-monomial-filters of the treks in $\T{}(W,W)\setminus \T{V}(W,W)$. Hence, $\spresidual{V}{W}$ is the part of the spectral density $\Sp{W}$ that is neither directly nor indirectly affected by $V$. These three factors decompose the spectral density of $W$ as
\begin{align} \label{eq: spectral density decomposition}
    \Sp{W} &= \spcausal{V}{W} + \spconf{V}{W} + \spresidual{V}{W}. 
\end{align}

We define the summand $\spcausal{V}{W}$ as the spectral density of the process $\X_{W_V} \coloneqq \ucetemp{V}{W}\ast \X_V$, which is the part of the process $W$ that is causally determined by $V$. Using Lemma \ref{lemma: filtering covariance}, the definition of causal effect functions and the trek rule, we express this spectral density in terms of trek-monomial-functions, i.e.,   
\begin{equation}\label{eq: spectral decomposition causal}
    \begin{split}
        \spcausal{V}{W} &\coloneqq \ucefrq{V}{W} \Sp{V}  \ucefrq{V}{W}^\ast \\
        &= \sum_{\pi_{V\to W} \in \mathrm{P}_V(V,W)} \Dirfrq^{(\pi_{V\to W})} \sum_{\pi_{V\leftrightarrow V \in \T{}(V,V)}} \Sp{}^{(\pi_{V \leftrightarrow V})} \sum_{\pi_{V\to W}' \in \mathrm{P}_V(V,W)} (\Dirfrq^{(\pi_{V\to W}')})^\ast \\
        &= \sum_{\pi \in \T{V \to W} } \Sp{}^{(\pi)},
    \end{split}
\end{equation}
where $\mathcal{T}_{V\to W}$ represents the set of treks $\pi \in \T{}(W,W)$ on which $V \in \trekL(\pi) \cap \trekR(\pi)$. 

The second factor in equation (\ref{eq: spectral density decomposition}), i.e., $\spconf{V}{W}$ indicates the possible confounding of the causal effect $\ucefrq{V}{W}$ and is defined as the sum over all trek-monomial-functions associated with the treks in $\T{V \leftrightarrow W} \coloneqq \T{V}(W,W) \setminus \T{V\to W}$. In Appendix \ref{section: spectrum decomposition} we use the trek-rule in combination with the definition of $\ucefrq{V}{W}$ to arrive at the following expression for the confounding contribution
\begin{align}\label{eq: spectral decomposition conf}
    \spconf{V}{W} &\coloneqq \sum_{\pi \in \T{V\leftrightarrow W}} \Sp{}^{(\pi)}  = 2\mathrm{Re}(\ucefrq{V}{W}\Sp{V,W}) - 2|\ucefrq{V}{W}|^2\Sp{V}.
\end{align}  
\begin{example}[Spectral density decomposition]\label{ex: specdtral-density decomposition} 
In Figure \ref{fig: spectrum decomposition}, we illustrate the spectral density decomposition into causal, confounding and residual contributions on the process graph $G_c$ from Figure \ref{fig:summary_graphs_demo}. 
Since each factor is a sum of trek-monomial-functions and there is no latent confounding (i.e., every trek is of type (\ref{trek-type I}) and therefore has a top node), we can partition each of the sets $\T{V \to W}$, $\T{V \leftrightarrow W}$ and $\T{W \setminus V}$ by grouping the treks by their top node. 
Specifically, we show how the process $X$ contributes to the spectrum of $Y$ by determining its causal contribution $\spcausal{X}{Y}$, see Figure \ref{fig: spectrum decomposition}.a.), the contribution due to confounding $\spconf{X}{Y}$, see Figure \ref{fig: spectrum decomposition}.b.), and the residual $\spresidual{X}{Y}$, Figure \ref{fig: spectrum decomposition}. c.). The plots in Figure \ref{fig: spectrum decomposition} are based on the SVAR process c.) as specified in Appendix \ref{section: paramterisations}. We provide detailed computations of these factors in Appendix \ref{section: spectrum decomposition}. 

\begin{figure}
    \centering
    \includegraphics[width=\textwidth]{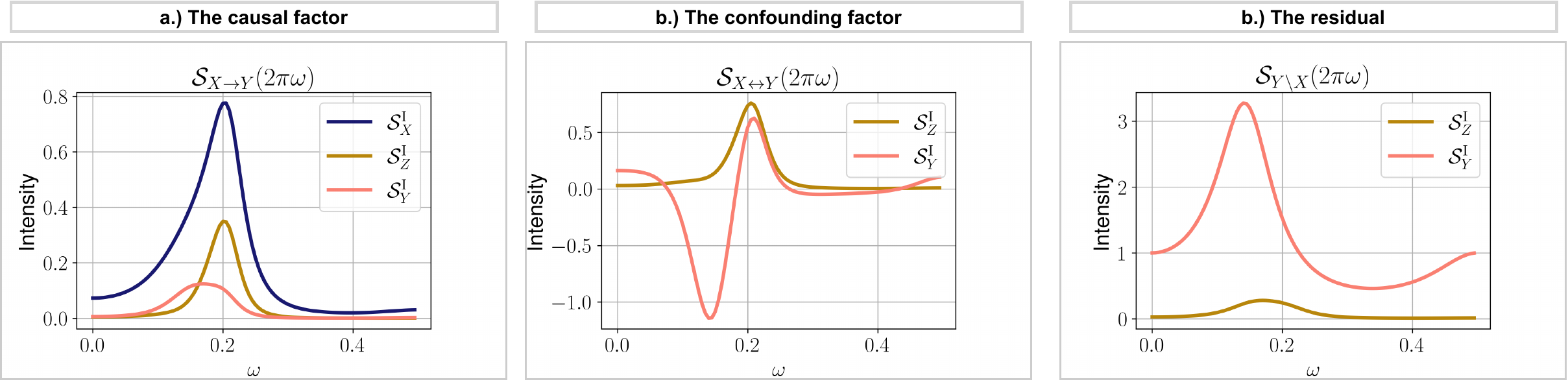}
    \caption{This figure illustrates an example (process graph $G_c$ in Figure \ref{fig:summary_graphs_demo}) of the spectral density decomposition of $Y$ with respect to its ancestor $X$. The underlying SVAR process is specified in Appendix \ref{section: paramterisations}. a.) This figure displays the causal contribution of $X$ to the spectral density of $Y$. The process $X$ is contributing causally to the spectral density of $Y$ by sending its internal dynamics $\X_X^\internal$ along all possible causal paths from $X$ to $Y$ (blue line). Furthermore, it contributes to $\mathcal{S}_Y$ by mediating the internal dynamics of $Z$ (yellow line) and the internal dynamics of $Y$ (salmon colored line). The sum over these components is $\spcausal{X}{Y}$. b.) This figure shows the spectral representation of the confounding between $X$ and $Y$. Both of the processes $Z$ and $Y$ act as confounders. Hence, the contribution due to confounding $X \leftrightarrow Y$ to the spectral density of $Y$ can be expressed in terms of internal dynamics of both $Z$ (yellow line) and $Y$ (salmon colored line). c.) This figure shows the residual spectral density, which is due to the internal dynamics of $Z$ (yellow line) along the direct link from $Z$ to $Y$, and the internal dynamics of $Y$ (salmon colored line).}
    \label{fig: spectrum decomposition}
\end{figure}
\end{example}

\section{Summary and outlook} \label{section: outlook}
In this work, we investigated the causal structure of linear SVAR processes. A SVAR process comes with two graphical objects that encode the dynamic causal structure among its components. The first graphical representation is the infinite time series graph, which qualitatively describes how the present and past states of the SVAR process determine the current state of the process up to noise. The time series graph has two sources of complexity: (1) the complexity arising from the causal structure among the processes and (2) the complexity stemming from the temporal structure, i.e., the time-lags at which these causal effects are realised. If we disregard the temporal complexity of the time series graph, then we obtain the finite and generally less complex process graph, which is the second graphical object linked to an SVAR process.   

\subsection{Summary} In this paper, we provided a formal SCM perspective that equips the process graph with a quantitative model for both the time and frequency domain. The basis of this perspective is a notion of (direct) causal effect objects between processes at the level of the process graph. We formulated these objects both in the time and frequency domain.
With these direct causal effects we established a set of linear equations encoding the relations amongst the processes up to a noise process, which we termed the projected internal dynamics. These equations can be seen as a parameterisation of the process graph by a generalised linear SCM. From this parameterisation, we derived a description of the spectral density of the SVAR process in terms of the transfer-function and the spectral density of the noise processes, i.e., projected internal dynamics. This description revealed that the spectral density of a SVAR process whose latent projected process graph is a mixed graph $G'$, has the same structure as the covariance matrix of a linear Gaussian SEM parameterising the mixed graph $G'$. Consequently, the process graph structures the spectral density in such a way that, depending on the process graph, the transfer-function of the associated SVAR process might be extracted only based on its process graph and spectral density. This extraction of the transfer-function is always possible, for example, if there is no latent confounding. In this case one obtains the transfer-function, frequency-wise, from the observational spectral density along the same lines as one would extract the linear model from the covariance of a fully observed SEM. Therefore, provided that the process graph is known and that the transfer-function for that graph is identifiable, there are two ways to estimate the transfer-function: (1) one can estimate the transfer-function directly from the (estimated) spectral density or (2) one can first estimate the full SVAR parameters and then use Lemma \ref{lemma: analytical expression transfer-function} to compute the transfer function. A comparison of these two approaches in terms of their statistical performance seems to be an interesting direction for future work. 

\subsection{Outlook} The formal perspective on the causal structure of the process graph suggests several interesting avenues for future research.  

\emph{Identification and estimation of causal effects}. 
One problem might be to characterise how much information is lost after disregarding the temporal structure. Specifically, one might try to determine which causal effects on the process graph can be identified from observations knowing the time series graph but cannot be identified when all we know is the process graph. This identifiability loss might occur, for example, when an acyclic time series graph collapses into a cyclic process graph. In such cases, it seems unlikely that one can generally avoid to exploit information on the lag-structure if the objective is to recover the direct effect-functions of links that participate in a cycle on the process graph.

A typical problem in Causal Inference is to estimate causal effects, given that one has at least partial knowledge of the causal graph. In SCM’s one can estimate these effects with so-called adjustment sets \cite{perkovic2018complete, pearl2009causality, shipster2010adjustment, smucler2021efficientadjustment, maathius2009interventioneffects}. The problem of graphically characterising efficient or optimal adjustment sets has been studied in e.g. \cite{henckel2021GraphicalTools, runge2021optimal, smucler2021efficientadjustment, henckel2022totaleffect}. A possible target for future work could be to follow up the work on efficient causal effect estimation at the level of the process graph \cite{reiter2024asymptotic}, where we identify the asymptotically optimal estimator for frequency domain causal effects under the assumption that no latent processes are present and that the process graph is known. 

\emph{Causal discovery of the process graph}. Recently, there have been several developments towards formalising causal discovery of the time series graph \cite{eichler2010graphical, dahlhaus2003causality, eichler2012graphical, runge2019detecting, runge2019inferring, gerhardus2020high, gerhardus2021characterization}. These concepts and methods aim at discovering the Markov equivalence class of the full time series graph. The framework in which they are phrased is that of SCMs. 
For the case of SVAR processes, our contribution sets grounds on which the discovery of the process graph can be formalised in the framework of SEP’s - a generalised version of linear SCM’s. If the process graph is a directed acyclic graph, then each $d$-separation statement about the process graph corresponds to a rank condition on the spectral density matrix, as we show in \cite{reiter2024causal}. This means that the Markov equivalence class of the process graph of a SVAR processes can be identified in two ways: (1) We could first discover the time series graph and then reduce it to its process graph (note that this graph can contain more information than the Markov equivalence class of the process graph). (2) We could use the SEP approach to discover the process graph directly. This poses two related question. First, is the estimation of the Markov equivalence class of the process graph along procedure (2) more reliable than discovery along procedure (1)? Second, if that turned out to be the case, can we use the information about the process graph to improve the result of discovering the full time series graph? 

\emph{Towards applications.} We believe that the formalisation of causal inference on process graphs for SVAR processes is also  relevant and useful in applications where one is often interested in systems of quantities that interact over time. Observations of such systems are given as time series, and from these data we wish to infer how strongly these quantities drive each other. In practice, we may already have qualitative knowledge of which quantity causally precedes which. However, we may not know the time lags after which these effects take place, nor may we be willing to make assumptions about them. Using the SEP representation of SVAR processes, one could ignore these lags and instead use only the process graph to compute causal effects. In particular, this representation avoids possible misspecification of the generally more complex time series graph. On the contrary, the causal effects introduced in this paper are much more complex than causal effects in the usual sense, and estimating them is likely to involve numerous statistical challenges. Future research should therefore identify and address these challenges. In \cite{reiter2024asymptotic} we take first steps in this direction.   
However, even if the time series graph is known or has been identified by a causal discovery algorithm, it may still be of interest to apply our findings, as the frequency domain causal effects as well as the spectral density decomposition allow for compact and informative representations of the causal structure in SVAR processes, as we demonstrate in \cite{reiter2024asymptotic}. 

\section*{Acknowledgments}
J.W. and J.R. received funding from the European Research Council (ERC) Starting Grant CausalEarth under the European Union’s Horizon 2020 research and innovation program (Grant Agreement No. 948112).
N.R., A.G. and J.R.  received funding from the European Union’s Horizon 2020 research and innovation programme under Marie Skłodowska-Curie grant agreement No 860100 (IMIRACLI).



\bibliographystyle{plain}
\bibliography{refs.bib}

\newpage

\appendix
\section{A commutative diagram on the relation between the SEP representations and how they generalise the SEM parameterisation} \label{sec: commutative diagram}
In this section, we provide a summary of the SEP representations constructed in this paper and relate them back to the parameterisation of SEMs. Specifically, we define the spaces in which these parameterisations "live" and the maps connecting them so that we can organise our constructions in a commutative diagram, see Figure \ref{fig: summary main results}. This perspective could become helpful when formalising algebraic identifiability theory for the process graph of SVAR processes. In the following, we fix a latent component process graph $G=(\mathbf{O} \cup \mathbf{L}, D)$ over $|\mathbf{O}| =m$ observed and $|\mathbf{L}|=d$ latent processes. We denote by $G'=(\mathbf{O}, D^\mathbf{O}, B)$ the latent projection of $G$.
\begin{figure}
    \centering
    \resizebox{\textwidth}{!}{
    \begin{tikzpicture}[sp_node/.style={rectangle, minimum size=1cm}]
        
        \node[sp_node] (SVAR) at (0,0) {$\SVAR(G)$};
        \node[sp_node] (SEP-time-par) [below=0.6cm of SVAR] {$\ell^1_\reg(\R^{D^\mathbf{O}}) \times \ell^\infty_{\PD(B)}(\Z, m)$};
        \node[sp_node] (SEM-par) [left=of SEP-time-par] {$\R^{D^\mathbf{O}}_{\reg} \times \PD(B)$};
        \node[sp_node] (SEP-frq-par) [right=of SEP-time-par] {$\Rfnc(\Cplx^{D^\mathbf{O}}_\reg) \times \Rfnc(\PD(B))$}; 

        \node[sp_node] (SEM-cov) [below=0.6cm of SEM-par] {$\PD_m$};
        \node[sp_node] (SEP-time-cov) [below=0.6cm of SEP-time-par]{$\ell^\infty_\PD(\Z, m)$};
        \node[sp_node] (SEP-frq-cov) [below=0.6cm of SEP-frq-par]{$\Rfnc(\PD_m)$};

        \node[sp_node] (SEM-title) [above=1.5cm of SEM-par] {\textbf{SEM}};
        \node[sp_node] (SEP-time-title) [above=1.5cm of SEP-time-par] {\makecell{\textbf{SEP} \\ \textbf{(time-domain)}}};
        \node[sp_node] (SEP-frq-title) [above=1.5cm of SEP-frq-par] {\makecell{\textbf{SEP} \\ \textbf{(frequency domain)}}};

        \draw[->](SVAR) -- node[left] {Theorem 1} (SEP-time-par);
        \draw[->](SVAR) -- node[above right] {$\mathcal{F}$(Theorem 1)} (SEP-frq-par);

        \draw[->](SEM-par) -- node[above] {$(\cdot \iota, \cdot\iota)$} (SEP-time-par);
        \draw[->](SEP-time-par) -- node[above] {$(\mathcal{F}, \mathcal{F})$} (SEP-frq-par);

        \draw[->](SEM-par) -- node[left] {$\Sigma$} (SEM-cov);
        \draw[->](SEP-time-par) -- node[left] {$\C_\mathbf{O}$} (SEP-time-cov);
        \draw[->](SEP-frq-par) -- node[left] {$\Sp{\mathbf{O}}$} (SEP-frq-cov);

        \draw[->](SEM-cov) -- node[above] {$\cdot \iota$} (SEP-time-cov);
        \draw[->](SEP-time-cov) -- node[above] {$\mathcal{F}$} (SEP-frq-cov);

        \draw[->](SEM-title) -- node[above] {White noise} (SEP-time-title);
        \draw[->](SEP-time-title) -- node[above] {Fourier} (SEP-frq-title);
    \end{tikzpicture}}
    \caption{This figure shows how the Fourier-Transformation (as represented by the two horizontal arrows on the right) connects the time and frequency SEP representations of a SVAR process $(\Phi, \mathbf{w})$, and how the SEP parameterisations generalise the SEM parameterisation. The top most vertical arrow employs Theorem 1 of the main paper to associate $(\Phi, \mathbf{w})$ with its time domain SEP representation $(\Dirtmp, \C^\linternal_\mathbf{O})$. The diagonal arrow on the top maps the pair $(\Phi, \mathbf{w})$ to the Fourier transformed SEP representation, represented as $(\Dirfrq, \Sp{\mathbf{O}}^\linternal)$ in the main text. The vertical arrows in the middle and on the right, connecting the second and third row of this diagram, indicate the computation of the ACS (middle) and the spectral density (right) in terms of the respective SEP representation. To see that the square involving these computations commutes, recall that the Fourier-Transformation takes convolution to point-wise matrix multiplication and tilted convolution to point-wise multiplication by the complex conjugate. The square on the left of this diagram shows the Gaussian white noise embedding of SEMs parameterising $G'$.}
    \label{fig: summary main results}
\end{figure}
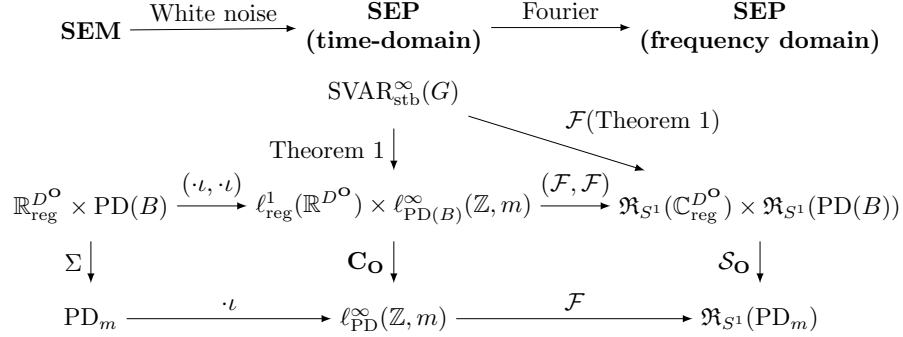 

Recall from \cite{10.1214/22-AOS2221, foygel2012half} that $\R^{D^\mathbf{O}}_\reg$ denotes the space of all $m\times m$-dimensional matrices $A$ such that $A_{V,W} = 0$ if $V \not \to W$ for the latent projection $G'$. Furthermore, we denote by $\PD_m$ the set of all positive definite symmetric (hermitian if considered over the complex numbers) $m\times m$ matrices and by $\PD(B)$ the set of matrices $\Omega=(\Omega_{V,W})_{V,W \in \mathbf{O}}\in \PD_m$ such that $\Omega_{V,W}= 0$ if $v \neq w$ and $v \not \leftrightarrow w$. A SEM parameterising the mixed graph $G'$ is given by a tuple of matrices $(A, \Omega) \in \R^D_\reg \times \PD(B)$ and has as observational covariance $\Sigma = (I-A)^{-\top} \Omega (I- A)^{-1}$. The SEM parameterisation of $G'$ is depicted in the first column of Figure \ref{fig: summary main results}.  

We now generalise these notations to sequence and function spaces in which the time resp. frequency domain SEP representations of SVAR processes that are consistent with $G$ are defined. Let us therefore denote by $\SVAR(G)$ the set of all parameter pairs $(\Phi, \mathbf{w})$ that satisfy the assumptions of Theorem 1 in the main paper, and by $\SVARs(G)$ the set of parameter pairs $(\Phi, \mathbf{w})$ that are consistent with $G$ and that satisfy
\begin{align}\label{eq: absolute stability condition}
    \sum_{V,W \in \mathbf{V}} \sum_{k = 0}^p |\phi_{V,W}(k)| & <1.
\end{align} 
In Proposition \ref{prop: supplement sep existence condition} we show that $\SVARs(G) \subset \SVAR(G)$.

\subsection{The sequence space of time domain SEP representations}
Let us denote by $\ell_\reg^{1}(\Z, \R^{D^\mathbf{O}})$ the space of all $m \times m$-dimensional absolutely summable filters $\Lambda$ such that $\Lambda(s) \in \R^{D^\mathbf{O}}$ for every $s \in \Z$, and $\Lambda^\infty$ exists and is entry-wise absolutely summable. We represent as $\ell^\infty_\PD(\Z, m)$ the space of bounded sequences $\C=(\C(\tau))_{\tau \in \Z}\subset \R^{m \times m }$ such that for every integer $N\geq1$ the $m\cdot N \times m \cdot N $ block-matrix $(\C(i-j))_{1 \leq i,j \leq N} \in \PD_{m \cdot N}$. Finally, we denote by $\ell^1_{\PD(B)}(\Z, m) \subset \ell^1_\PD(\Z, m)$ the subset consisting of sequences $\C= (\C(\tau))_{\tau \in \Z}$ such that for all $\tau \in \Z$ the entry $\C_{V,W}(\tau) = 0$ if $V \not \leftrightarrow W$. With these notations we summarise Theorem 1 of the main paper as the composition of maps
\begin{equation}\label{eq: sep parameterisation map}
     \SVAR(G) \to \ell^1_\reg(\Z, \R^{D^\mathbf{O}}) \times \ell^\infty_{\PD(B)}(\Z, m) \to \ell^\infty_{\PD}(\Z, m), 
\end{equation}
where the first arrow maps $(\Phi, \mathbf{w})$ to its time domain SEP representation $(\Dirtmp, \C_\internal^\mathbf{O})$ and the second denotes the computation of the ACS $\C_\mathbf{O}$. This is the second column in the commutative diagram of Figure \ref{fig: summary main results}.

\subsection{The rational function space of frequency domain SEP representations}
Let us denote by $\Rfnc$ the space consisting of all functions $f: D^1 \subset \Cplx \to \Cplx$ such that $f(z) =\frac{p(z)}{q(z)}$ for every $z\in S^1\subset \Cplx$, where $p,q$ are polynomials with real coefficients. Accordingly, we denote by $\Rfnc(\Cplx^{D^\mathbf{O}}_\reg)$ the set of all functions $f=(f_{V,W})_{V,W \in \mathbf{O}}: D^1 \to \Cplx^{m \times m}$ such that $f_{V,W} \in \Rfnc$ for every $V,W\in \mathbf{O}$ and that $f(z) \in \Cplx^{D^\mathbf{O}_\reg}$ for all $z\in D^1$. Similarly, we define $\Rfnc(\PD_m)$ to be the set of functions $f=(f_{V,W})_{V,W\in \mathbf{O}}:S^1 \to \Cplx^{m\times m}$ where $f_{V,W} \in \Rfnc$ and $f(z) \in \PD(m)$ for all $z\in S^1$. Eventually, we represent by $\Rfnc(\PD(B)) \subset \Rfnc(\PD_m)$ the set of matrix-valued functions $f$ such that $f(z)\in \PD(B)$ for all $z \in S^1$. Application of the Fourier-Transformation to (\ref{eq: sep parameterisation map}) together with Lemma 1 of the main paper and the observation that $z^\ast = z^{-1}$ for every $z\in S^1$ gives the third column of the diagram depicted in Figure \ref{fig: summary main results}.

\subsection{The Gaussian white noise embedding}
Note that any SEM $(A,\Omega) \in \R^{D^\mathbf{O}}_\reg \times \PD(B)$ paramterising the mixed graph $G'$ defines a tuple $(A \cdot \iota, B\cdot \iota )\in \ell^1_\reg(\Z, \R^{D^\mathbf{O}}) \times \ell^\infty_{\PD(B)}(\Z, m)$ where $A \cdot \iota $ is the $\Z$-indexd filter that has the matrix $A$ in its zeroth entry and the zero matrix in all other entries. We define the filter $\Omega \cdot \iota$ in an analogous manner. Note that $(A \cdot \iota, \Omega \cdot \iota)$ is the SEP representation of a SVAR white noise process, i.e., a SVAR process parameterising a time series graph with only contemporaneous links. This representation of SEMs is denoted by the horizontal arrows connecting the first and second column in the diagram of Figure \ref{fig: summary main results}. 

Finally, observe that the frequency domain SEP representation of the SEM $(A, \Omega)$ is the tuple of constant functions $(A \cdot \mathcal{I}, \Omega \cdot \mathcal{I}) \in \Rfnc(\Cplx^{D^\mathbf{O}}_\reg) \times \Rfnc(\PD(B))$, where $A \cdot \mathcal{I}(z) = A$ resp. $\Omega \cdot \mathcal{I}(z) = \Omega$ for every $z \in D^1$. 

\section{Stable SVAR processes}\label{section: stable VAR processes}
In this section, we recall some notions about Gaussian Processes and stable SVAR processes in order to facilitate the proof of theorem 1 of the main paper. 
\begin{definition}[Gaussian Process]
    Let $\X_\mathbf{V} = (\X_V(t))_{t \in \Z, V \in \mathbf{V}}$ be a $m=|\mathbf{V}|$-dimensional stochastic process. It is called a \textit{Gaussian Process} if for any finite sub-collection $\{(V_i,t_i) \}_{1\leq i \leq n} \subset \mathbf{V} \times \Z$ the random vector $(\X_{V_i}(t_i))_{1 \leq i \leq n}$ follows a multi-variate normal distribution. If furthermore, for all $\tau \in \Z$ one has that $(\X_{V_i}(t_i + \tau))_{1 \leq i \leq n}$ follow the same distribution, then the Gaussian Process $\X_\mathbf{V}$ is called stationary.
\end{definition}
A Gaussian Process $\X_\mathbf{V}$ is entirely specified by its mean sequence $\mu_\mathbf{V} =(\Ep[\X_\mathbf{V}(t)])_{t \in \Z}$ and its ACS $\C_\mathbf{V}(t,\tau) = (\Ep[(\X_\mathbf{V}(t) - \mu_\mathbf{V}(t)) (\X_\mathbf{V}(t-\tau)- \mu_\mathbf{V}(t-\tau))^\top])_{t, \tau \in \Z}$. If $\X_\mathbf{V}$ is stationary, then both its mean-sequence and ACS are independent of $t$. 

A SVAR process $\X_\mathbf{V}$ specified by $(\Phi, \mathbf{w})$ is called \textit{stable} if its \textit{reverse characteristic polynomial} \cite[Chapter 2]{lutkepohl2005new}
\begin{align}\label{condition: stability I}
    \mathrm{det}(I - z (I -\Phi(0)^{-1})\Phi(1)- \cdots -z^p(I - \Phi(0)^{-1})\Phi(p))
\end{align}
is non-zero on the complex unit disk, i.e., for every $z \in \mathbb{C}$ s.t. $|z| \leq 1$. 

Suppose $\X_\mathbf{V} = (\X_V)_{V \in \mathbf{V}}$ is a stable $m = |\mathbf{V}|$-dimensional SVAR process specified by $(\Phi, \mathbf{w})$. Then by \cite{lutkepohl2005new}[Chapter 2] or \cite{brockwell2009time}[Chapter 11] there is an entry-wise absolutely summable $m \times m $-dimensional filter $\Psi = (\psi_{V,W})_{V,W \in \mathbf{V}}$ such that
\begin{align*}
    \X_\mathbf{V} &= \Psi^\top \ast \eta',
\end{align*}
where $\eta' = ((I- \Phi(0)^{-\top}) \eta(t))_{t \in \Z} = (\eta'_{V}(t))_{V \in \mathbf{V}, t \in \Z}$ is a Gaussian white noise process, i.e., any finite collection $(\eta_{V_1}(t_1), \cdots, \eta_{V_m}(t_m))$ follows a joint Gaussian distribution, such that
\begin{align*}
    \Ep[\eta'_{V_i}(t)\eta'_{V_j}(s) ] &= \begin{cases}
        [(I- \Phi(0))^{-\top} \mathbf{w} (I - \Phi(0))^{-1}]_{V_i,V_j} & \text{if $s = t$} \\
        0 & \text{otherwise}
    \end{cases}
\end{align*}
Recall that $\mathbf{w}$ is a diagonal matrix, meaning that the Gaussian white noise process $\eta = (\eta_V)_{V\in \mathbf{V}}$ is both serially and mutually independent.

The elements of the multivariate filter $\Psi$ are defined recursively as 
\begin{align*}
    \Psi(k)= \begin{cases}
        \mathbf{0}& \text{if $k < 0$} \\
        I & \text{ if $k=0$} \\
        \Phi'(k) + \sum_{l= 1}^{k-1}  \Psi(k-l) \Phi'(l)& \text{if $1 \leq k \leq p$} \\
        \sum_{l= 1}^{p}  \Psi(k-l) \Phi'(l) & \text{otherwise}
    \end{cases},
\end{align*}
where $\Phi'(l) = \Phi(l)(I - \Phi(0))^{-1}$ and $1 \leq l \leq p$. Then $\X_\mathbf{V}$ is a stationary multivariate Gaussian Process, specified by the mean resp. covariance function 
\begin{align*}
    \mu_\mathbf{V}(t) &= \sum_{k \geq 0} \Psi(k)^\top \mu' \\
    \C_\mathbf{V}(\tau) &= \sum_{k \geq 0} \Psi(k + \tau)^\top \mathbf{w}' \Psi(k). 
\end{align*}

To aid the proof of the main theorem of this work, we recall some notions about convergence of sequences of filters and filtered Gaussian Processes. In the following we denote by $\ell^1(\Z, \mathbb{R}^{m \times n})$ the space of absolutely summable $m \times n$-dimensional filters. We just write $\ell^1$ if the specification of the dimension is clear from the context. Furthermore, if $A \in \mathbb{R}^{m \times n}$, then we set 
\begin{align}
    \lVert A \rVert \coloneqq (\mathrm{tr}(A A^\top))^{\frac{1}{2}}
\end{align}
to be the Hilbert-Schmitt norm of $A$.
If $\Lambda= (\Lambda(j))_{j \in \Z}$ is an $m \times n$-dimensional filter, we also consider the norms
\begin{align*}
    \lVert \Lambda \rVert_1  & \coloneqq \sum_{j \in \Z} \lVert \Lambda(j) \rVert \\
    \lVert \Lambda \rVert_\infty & \coloneqq \sup_{j \in \Z} \lVert \Lambda(j) \rVert. 
\end{align*}
A filter $\Lambda$ is entry-wise absolutely summable if and only if $\lVert \Lambda \rVert_1 < \infty$. A sequence of $\ell^1$-filters $(\Lambda^{(n)})_{n \geq 0}$ is said to converge to a filter $\Lambda \in \ell^1$ in $\ell^1$ if 
\begin{align*}
    \lim_{n \to \infty} \lVert \Lambda^{(n)} - \Lambda \rVert_1 &= 0.
\end{align*}
In our context, we are mainly interested in filters $\Lambda=(\Lambda(j))_{j \in Z}$ that have non-zero elements only at non-negative indices. We therefore denote by $\ell^1_{\geq 0}$ the subspace of $\ell^1$ consisting of those filters that are non-zero only at non-negative indices. 

Furthermore, let $(\X_\mathbf{V}^{(n)})_{n\geq 1} = ((\X^{(n)}_V)_{V \in \mathbf{V}})_{n\geq 0}$ be a sequence of Gaussian Processes. Then the sequence $(\X_\mathbf{V}^{(n)})_{n\geq 1}$ is said to converge in distribution to a Gaussian Process $\X_\mathbf{V}$ if for any finite collection $\{(V_i,t_i) \}_{1 \leq i \leq q} \subset \mathbf{V} \times \Z$ the sequence of random vectors $((\X^{(n)}_{V_i}(t_i))_{1\leq i\leq q})_{n\geq 1}$ converges in distribution to the random vector $(\X_{V_i}(t_i))_{1\leq i \leq q}$. It follows from Levy's convergence theorem, see e.g. \cite{williams_1991}, that $(\X_\mathbf{V}^{(n)})_{n \geq 1}$ converges to $\X_\mathbf{V}$ in distribution if the sequences $(\mu_\mathbf{V}^{(n)})_{n\geq 1}$ and $(\C_\mathbf{V}^{(n)}(t, \tau))_{n \geq 1}$ converge point-wise to $\mu_\mathbf{V}$ resp. $\C_\mathbf{V}$.

\begin{lemma}\label{lemma: supplement continuity convolution}
    Let $(\Lambda^{(n)})_{n \geq 1}$ be a sequence of filters in $\ell^1_{\geq 0}(\Z, \mathbb{R}^{m \times n})$, such that they converge to $\Lambda \in \ell^1_{\geq 0}(\Z, \mathbb{R}^{m \times n})$ in $\ell^1$. Suppose $\Psi \in \ell^1_{\geq 0}(\Z, \mathbb{R}^{n \times q})$, then $(\Lambda^{(n)}\ast \Psi)_{n \geq 1}$ converges to $\Lambda \ast \Psi$ in $\ell^1$. That means, convolution by $\Psi$ defines a continuous operator on $\ell^1$. 
\end{lemma}
\begin{proof}
     The claim follows from the following computation 
    \begin{align*}
        \sum_{j \in \Z} \lVert(\Lambda \ast \Psi)(j)- (\Lambda^{(n)} \ast \Psi(j))\rVert &= \sum _{j \geq 0} \lVert\sum_{l = 0}^j \Lambda(l) \Psi(j-l) - \Lambda^{(n)}(l)\Psi(j-l) \rVert \\
        & \leq \sum_{j \geq 0} \sum_{l =0}^j\lVert \Lambda(l) \Psi(j-l) - \Lambda^{(n)}(l) \Psi(j-l) \rVert \\
        & \leq \sum_{j \geq 0}\sum_{l \geq  0} \lVert \Lambda(l) - \Lambda^{(n)}(l) \rVert \lVert \Psi(j-l) \rVert \\
        & = (\sum_{l \in \Z} \lVert \Lambda(l) - \Lambda^{(n)}(l) \rVert ) (\sum_{j \in \Z} \lVert \Psi(j) \rVert).  
    \end{align*}
    The expression on the last line goes to zero as $(\Lambda^{(n)})$ converges to $\Lambda$ in $\ell^1$ and $\Psi$ is absolutely summable. This shows the claim.  
\end{proof}
\begin{lemma}\label{lemma: supplement GP convergence}
    Let $\Psi = (\psi_{V,W})_{V, W \in \mathbf{V}} \in \ell^{1}_{\geq 0}$ be a matrix-valued filter. Suppose that $\eta = (\eta_V(t))_{V\in \mathbf{V}, t \in \Z}$ is a multivariate Gaussian white noise process whose instantaneous covariance structure shall be denoted by $\mathbf{w} = \Ep[\eta(t)^\top \eta(t)]$. Let furthermore $(\Psi^{(n)})_{n \geq 1}$ be a sequence in $\ell^1_{\geq 0}$ filters having the same dimension as $\Psi$ and such that $\Psi^{(n)}$ converges to $\Psi$ in $\ell^1$. Then the sequence of Gaussian Processes $((\Psi^{(n)})^\top \ast \eta)_{n \geq 1}$ converges in distribution to the Gaussian Process $\Psi^\top \ast \eta$.  
\end{lemma}
\begin{proof} 
    Let us denote by $\C$ and $\C^{(n)}$ the ACS of the stationary Gaussian Process $\X_\mathbf{V}^{(n)} = (\Psi^{(n)})^\top \ast \eta$ resp. $\X_\mathbf{V} = \Psi^\top \ast \eta$. For any $\tau \in \Z$ their respective elements are given by
    \begin{align*}
        \C^{(n)}(\tau) &= \sum_{j \geq 0} (\Psi^{(n)})^\top(\tau +j) \mathbf{w} \Psi^{(n)}(j) \\
        \C(\tau) &= \sum_{j \geq 0} \Psi^\top(\tau + j) \mathbf{w} \Psi(j). 
    \end{align*}
    Following the previous discussion on convergence in distribution of Gaussian Processes, we need to show that $\lVert\C^{(n)}(\tau) - \C(\tau)\rVert$ converges to zero as we let $n$ go to infinity, for every $\tau \in \Z$. We restrict to the case where the covariance matrix $\mathbf{w}$ is the identity matrix. Let us compute 
    \begin{align*}
        \lVert \C^{(n)}(\tau) - \C(\tau)\rVert & =\lVert \sum_{j \geq 0} (\Psi^{(n)})^\top(\tau +j)\Psi^{(n)}(j)-  \sum_{j \geq 0} \Psi^\top(\tau + j) \Psi(j)\rVert \\
        & \leq \sum_{j \geq 0} \lVert \Psi^{(n)}(\tau + j) - \Psi(\tau + j) \rVert \lVert \Psi^{(n)}(j)\rVert \\
        &+ \sum_{l \geq 0} \lVert \Psi(\tau +l)\rVert (\Psi^{(n)}(l) - \Psi(l))\rVert \\
        &\leq (\alpha + \lVert \Psi \rVert_\infty) \lVert \Psi^{(n)} - \Psi \rVert_1,  
    \end{align*}
    where $\alpha = \sup_{n \geq 0} \lVert \Psi^{(n)} \rVert_1 < \infty$. Hence, the point-wise convergence of the ACS's follows. In a similar manner one shows the convergence of the mean sequences. 
\end{proof}

\begin{corollary}\label{corrollary: supplement convergence SVAR processes}
    Let $\X_\mathbf{V}$ be a stable SVAR process specified by $(\Phi,\mathbf{w})$ and suppose $(\Lambda^{(n)})_{n \geq 1}$ is a sequence in $\ell^{1}_{\geq 0}(\Z, \mathbb{R}^{|\mathbf{V}| \times |\mathbf{V}}|)$ that converges in $\ell^1$ to a filter $\Lambda \in \ell^1_{\geq 0}(\Z, \mathbb{R}^{|\mathbf{V}| \times |\mathbf{V}}|)$. Then the sequence of Gaussian Processes $(\Lambda^{(n)} \ast \X_\mathbf{V})_{n \geq 1}$ converges in distribution to the Gaussian Process $\Lambda\ast \X_\mathbf{V}$.   
\end{corollary}
\begin{proof}
    This follows by combining Lemma \ref{lemma: supplement continuity convolution} and Lemma \ref{lemma: supplement GP convergence}. 
\end{proof}

\section{Proof of Theorem 1: structural equation process representation of SVAR processes}\label{section: proof main theorem}
\begin{theorem}[Structural equation process representation, Theorem 1 in the main paper]\label{thm: supplement summary SEM}
     Let $(\Phi, \mathbf{w})$ be a parameter pair that specifies a stable SVAR process $\X_{\mathbf{O} \cup \mathbf{L}}$ consistent with $\mathcal{G}$ such that for every $V \in \mathbf{V}= \mathbf{O} \cup \mathbf{L}$ the following stability condition is satisfied
    \begin{align} \label{condition: supplement stability II}
        \sum_{k=1}^p |\phi_{V,V}(k)| < 1, 
    \end{align} and the power-series of filters $\Dirtmp^\infty \coloneqq\sum_{k = 0}^\infty \Dirtmp^k$ exists and is an entry-wise absolutely summable filter. Then the observed processes satisfy the linear relation
    \begin{equation} \label{eq: supplement processes equation}
        \begin{split}
             \X_\mathbf{O} &= \Dirtmp^\top \ast \X_\mathbf{O} + \Ltmp^\top \ast \mathbf{L} + \X_\mathbf{O}^\internal \\
            &=(\Dirtmp^\infty)^\top \ast \X_\mathbf{O}^\linternal, 
        \end{split}
    \end{equation}
    so that the auto-covariance sequence of the obsered process is 
    \begin{equation}
        \begin{split} \label{eq: supplement corss covariance process graph}
            \C_{\mathbf{O}} &= (\Dirtmp^\infty)^\top \ast (\C_\mathbf{O}^\linternal) \tcvl \Dirtmp^\infty, \text{ where}\\
            \C_\mathbf{O}^\linternal &= \C_\mathbf{O}^\internal + \Gamma^\top \ast \C_\mathbf{L} \tcvl \Gamma.
        \end{split}
    \end{equation} 
\end{theorem}

\begin{lemma}\label{lemma: filtering covariance}
    Let $(\mathbf{X}, \mathbf{Y})$ be a stationary stochastic process with bounded auto-covariance sequence, $\Lambda$ a $p \times|\mathbf{X}|$-dimensional filter and $\Gamma$ a $q \times |\mathbf{Y}|$-dimensional filter such that both are element-wise absolutely summable. Then 
    \begin{equation}
        \mathbf{C}_{\Lambda\ast \mathbf{X}, \Gamma \ast \mathbf{{Y}}} = \Lambda \ast \mathbf{C}_{\mathbf{X}, \mathbf{Y}} \hat{\ast} \Gamma^\top
    \end{equation}
\end{lemma}
\begin{proof}
We first compute 
    \begin{align*}
         \mathbf{C}_{\Lambda \ast \mathbf{X},\mathbf{Y}}(\tau) &= \Ep[(\Lambda\ast\mathbf{X})(t) (\mathbf{Y})(t-\tau))^\top] \\
         &= \sum_{s \geq 0}\Ep[\Lambda(s)\mathbf{X}(t-s)(\mathbf{Y}(t-\tau))^\top] \\
         & =\sum_{s \geq 0}\Lambda(s) \mathbf{C}_{\mathbf{X}, \mathbf{Y}}(\tau -s)\\
         & = \Lambda\ast\mathbf{C}_{\mathbf{X}, \mathbf{Y}} (\tau)
    \end{align*}
Now we compute $\mathbf{C}_{\mathbf{X}, \Gamma\ast \mathbf{Y}}$
\begin{align*}
    \mathbf{C}_{\mathbf{X}, \Gamma\ast\mathbf{Y}}(\tau) &= \Ep[\mathbf{X}(t) (\Gamma\ast\mathbf{Y}(t-\tau))^\top] \\
    &= \sum_{s \geq 0}\Ep[\mathbf{X}(t)(\mathbf{Y} (t-\tau-s))^\top](\Gamma(s))^\top \\
    &= \sum_{s \geq 0} \mathbf{C}_{\mathbf{X}, \mathbf{Y}}(\tau +s) (\Gamma(s))^\top \\
    &= \mathbf{C}_{\mathbf{X}, \mathbf{Y}} \hat{\ast} (\Gamma)^\top(\tau)
\end{align*}
The lemma follows by combining these two computations. 
\end{proof}
\begin{proof}[Proof of Theorem \ref{thm: supplement summary SEM}]
    We start by showing (\ref{eq: supplement processes equation}). For its proof it is not necessary to distinguish between latent and observed processes. We therefore refer to $\mathbf{V}= (\mathbf{O}, \mathbf{L})$ as the process consisting of the observed and latent processes. Put differently, we neglect the latent structure to avoid unnecessarily complicated notation in the proof. Let us fix two processes $U,V \in \mathbf{V}$. In the following we will denote the coefficients parameterising the auto-dependencies in $V$ by $a_k \coloneqq \phi_{V,V}(k)$. With the recursive filter
    \begin{align*}
        f_V(j) &\coloneqq \begin{cases}
            0 & \text{if $j < 0$} \\
            1 & \text{if $j = 0$} \\
            \sum_{k =1}^{p} a_k f_V(j-k) & \text{otherwise}
        \end{cases}.
    \end{align*}
    we can express the internal dynamics as $\X_V^\internal = f_V \ast \eta_V$. Note that the entry $f_V(j)$ is the sum of all path-coefficients of the paths from $V(t-j)$ to $V(t)$ that only consist of auto-dependency links. 
    
    Furthermore, we rewrite the direct effect filter as
    \begin{align}\label{eq: direct effect internal dynamics}
        \dirtmp{U}{V}(j) &= \sum_{k = 0}^p \phi_{U,V}(k) f_V(j-k),
    \end{align}
    which sums all the path-coefficients of those paths from $U(t-j) \to V(t)$ on which the first link is an edge $U(t-j) \to V(t-j+l)$ and all subsequent links are auto-dependencies in $V$.

    The state of the considered process $V$ at time $t$ is defined by the SVAR-equations, namely
    \begin{align*}
        \X_V(t) &= \sum_{U \in \Pa(V)} \sum_{k = 0}^p \phi_{U,V}(k) \X_U(t-k) + \sum_{k =1}^p a_k \X_V(t-k) + \eta_V(t). 
    \end{align*}
    In a recursive manner, we expand the states $\X_V(t-j)$ while leaving all other state variables that appear during this procedure untouched. For arbitrary recursion step $n$ this yields  
    \begin{equation} \label{eq: recursive expansion}
        \begin{split}
              \X_V(t) &= \sum_{U \in \Pa(V)} \sum_{j \geq 0} \dirtmp{U}{V}^{(<n+1)}(j) \X_U(t-j) + \sum_{j\geq 0 } f_V^{(n)}(j) \X_V(t-j) + \sum_{j \geq 0}f_V^{(<n)}(j)\eta_V(t-j) 
        \end{split}
    \end{equation}
    where 
    \begin{align*}
        f_V^{(1)}(j) &= \begin{cases}
        a_j & \text{if $1 \leq j \leq p$} \\
        1 & \text{if $j =0$} \\
        0 & \text{otherwise}
        \end{cases}
    \end{align*}
    and 
    \begin{align*}
        f_V^{(n+1)}(j) &= \sum_{k=1}^p \phi_{V, V}(k) f_V^{(n)}(j-k), \\
        f_V^{(<n+1)}(j) &= \sum_{l=1}^n f_V^{(l)}(j), \\
        \dirtmp{U}{V}^{(<n+1)}(j) &= \sum_{l=1}^n \sum_{k=0}^p \phi_{U, V}(k) f_V^{(l)}(j-k). 
    \end{align*}
    
    Note that $f_V^{(n)}(k)$ is the sum of path-coefficients that are defined by auto-dependency paths of length $n$ going from $V(t-k) \to V(t)$, which implies $f^{(n)}(k) = 0$ whenever $k < n $ or $k \geq p\cdot n$. 
    If $n \geq j$, then $f_V^{(<n+1)}(j) = f(j)$, as $f_V^{(<n+1)}(j)$ is the sum over all path-coefficients of paths from $V(t-j)$ to $V(t)$ that consist of at most $n+1$ auto-dependency links. Since an auto-dependency link is lagged by at least one time-step, a path composed of auto-dependency links from $V(t-j) \to V(t)$ has at most $j$ edges. Since the set of directed auto-dependency paths can be partitioned by (path-) length we can decompose
    \begin{align}\label{eq: internal dynamics decomposition}
        f_V = \sum_{n \geq 1} f_V^{(n)}.
    \end{align}

    Combining equations (\ref{eq: direct effect internal dynamics}) and (\ref{eq: internal dynamics decomposition}) yields the following expression for the direct effect filter
    \begin{align*}
        \dirtmp{U}{V} &= \sum_{k = 0}^p \phi_{U,V}(k) \sum_{n\geq 1} f_V^{(n)}.
    \end{align*}

    With matrix-valued filters we summarise the process-wise expansion (\ref{eq: recursive expansion}) as follows
    \begin{align}
        \X_\mathbf{V} &= (\Dirtmp^{(<n+1)} + f^{(n)} )^\top \ast \X_\mathbf{V} + f^{(<n)} \ast \eta,  
    \end{align}
    where $\Dirtmp^{(<n+1)}$ is the matrix-valued filter $(\dirtmp{U}{V}^{(<n+1)})_{U,V \in \mathbf{V}}$. The filters $f^{(n)}$ and $f^{(<n)}$ are defined analogously. Note that $f^{(n)}(k)$ resp. $f^{(<n+1)}(k)$ are diagonal matrices for every $k \in \Z$.  
    We prove (\ref{eq: supplement processes equation}) by showing the following convergence in distribution statements
    \begin{align} \label{eq: supplement main thm convergence in distribution}
        \lim_{n \to \infty} (\Dirtmp^{(<n+1)} + f^{(n)})^\top \ast \X_\mathbf{V} &= \Dirtmp^\top \ast \X_\mathbf{V} \\
        \lim_{n \to \infty} f^{(<n+1)} \ast \eta &= \X_\mathbf{V}^\internal.
    \end{align}
    
    Since $\X_\mathbf{V}$ is a stable SVAR process it suffices to check that $\Dirtmp^{(<n)}$ converges to $\dirtmp{U}{V}$, $f^{(<n)}$ converges to $f$, and that $f^{(n)}$ converges to the zero element in $\ell^1$. We show these convergence statements entry-wise. Let us therefore return to our pair of processes $U,V$. 

    To prove these three $\ell^1$ convergence statements we show that
    \begin{align}\label{eq: power-series}
        \sum_{n \geq 1} \sum_{k \geq 0} |f_V^{(n)}(k)| < \infty,
    \end{align}
    by employing the stability condition (\ref{condition: supplement stability II}), which ensures that
    \begin{align} \label{eq: sum of coefficients}
        |a_1| + \cdots + |a_p| < 1. 
    \end{align}
    To do so, consider the formal power series associated with the sequence $f_V$, i.e., 
    \begin{align*}
        F_V(z) &=\sum_{j \geq 1} f_V(j) z^j.
    \end{align*}
    We show that the formal power-series defined by
    \begin{align*}
        G_V(z) &= \sum_{j \geq 1} (a_1 z + \cdots + a_p z^p)^j \\
        &= \sum_{k \geq 1} g_V(j) z^j 
    \end{align*}
    is equal to the formal power-series $F(z)$ by showing that $g_V(j) = f_V(j)$ for all $j \geq 0$. In order to express the coefficient $g_V(j)$ in terms of the $(a_k)_{1\leq k \leq p}$ we introduce the set of tuples
    \begin{align*}
        S_j &= \{ \pi = (k_1, \dots, k_m) | 1 \leq k_l \leq p, \text{ and } \sum_{l=1}^m k_l = k \}.
    \end{align*}
    Furthermore, if $\pi$ is a tuple in $S_k$, then we define
    \begin{align*}
        a^{(\pi)} &= a_{k_1} \cdots a_{k_m}. 
    \end{align*}
    The definition of $G_V$ makes apparent that 
    \begin{align*}
        g_V(j) &= \sum_{\pi \in S_j} a^{(\pi)}.
    \end{align*}
    Using the observation
    \begin{align}
        S_j = \bigcup_{l=1}^{\min\{j-1,p\}} \{(l, \pi'): \pi' \in S_{j-l} \}
    \end{align}
    together with an inductive argument one concludes that $f_V(j) = g_V(j)$ for all $j$. Obviously, $f(j) = g(j)$ if $j \leq 1$. Now assume that $f(l) = g(l)$ for $1 \leq l < j+1$, then 
    \begin{align*}
        f_V(j+1) &= \sum_{k= 1}^p a_k f_V(j+1-k) \\
        &= \sum_{k=1}^p a_k g(j+1-k) \\
        &= \sum_{l=1}^p a_k \sum_{\pi' \in S_{j+1-k}} a^{(\pi')} \\
        &= \sum_{\pi \in S_{j+1}} a^{(\pi)} \\
        &= g_V(j+1).
    \end{align*}
    In order to demonstrate the absolute convergence of $f$ we introduce the sequence $|f|_V = (|f|_V(k))_{k \in \Z}$, which is recursively defined as follows 
    \begin{align*}
        |f|_V(j) &= \begin{cases}
            0 & \text{if $j < 0$} \\
            1 & \text{if $j = 0$} \\
            \sum_{k =1}^{p} |a_k| |f|_V(j-k) & \text{otherwise}
        \end{cases}.
    \end{align*}
    Similarly, $|f|_V^{(n)}$ shall denote the filter whose values follow the same recursion as $f^{(n)}$ but with the absolute values of the auto-dependency coefficients $|a_k|$ instead of $a_k$. One readily checks for every $j \in \Z$ the inequality
    \begin{align}
        |f^{(n)}_V(j)| & \leq |f|_V^{(n)}(j).
    \end{align}
    Using this inequality, we compute 
    \begin{align*}
       \sum_{j \geq 0} \sum_{n \geq 1} |f^{(n)}_V(j)| &\leq \sum_{j \geq 0} \sum_{n \geq 1 } |f|_V^{(n)}(j) \\
        &= \sum_{j \geq 0} |f|_V(j) \\
        &= \sum_{j \geq 1}(|a_1| + \cdots + |a_p|)^j \\
        &= \frac{1}{1 - \sum_{k = 0} |a_k|},
    \end{align*}
    where the first equality is due to equation (\ref{eq: internal dynamics decomposition}) and the last expression is well defined because of (\ref{condition: supplement stability II}). This allows us to conclude with (\ref{eq: power-series}) and hence, with the convergence of $f^{(<n)}_V $ to $f$ in $\ell^1$, which can be seen by combining (\ref{eq: power-series}) with the inequalities 
    \begin{align*}
        \sum_{j \geq 0} |f_V(j) - f_V^{(<n)}(j)| &= \sum_{j \geq 0} |\sum_{m\geq 1} f_V^{(m)}(j) - \sum_{l=1}^{n-1}f^{(l)}_V(j)| \\
        &= \sum_{j \geq 0} |\sum_{m \geq n} f^{(m)}_V(j)| \\
        &\leq \sum_{j \geq 0} \sum_{m \geq n} |f^{(m)}_V(j)|.
    \end{align*}
    The last expression goes to zero as we let $n$ go to infinity because of inequality (\ref{eq: power-series}).
    Analogously, one argues for the $\ell^1$ convergence of $f^{(n)}$ to the zero element and of $\dirtmp{U}{V}^{(n)}$ to $\dirtmp{U}{V}$. 
    Corollary \ref{corrollary: supplement convergence SVAR processes} allows us to conclude with the convergence in distribution (\ref{eq: supplement main thm convergence in distribution}) and thus (\ref{eq: supplement processes equation}). Equation (\ref{eq: supplement corss covariance process graph}) follows by combining the second line of Equation (\ref{eq: supplement processes equation}) and Lemma \ref{lemma: filtering covariance}.
\end{proof}

Note that condition (\ref{condition: supplement stability II}) implies that $\Dirtmp^\infty$ and every CCF $\cetemp{\mathbf{X}}{Y}{\mathbf{Z}}$ are absolutely summable filters whenever the process graph $G$ is acyclic because then the sets of $\mathbf{Z}$-avoiding directed paths are finite sets. However, in case $G$ has cycles the sets of $\mathbf{Z}$-avoiding directed paths can be infinite. Then condition (\ref{condition: supplement stability II}) may not suffice to ensure absolute summability. Nonetheless, there is a simple condition on the SVAR-parameter $\Phi$ granting that the power-series $\Dirtmp^\infty$ (and all CCFs) exists and is absolutely summable.
\begin{proposition} \label{prop: supplement sep existence condition}
    Let $\X_\mathbf{V}$ be a SVAR process of order $p$, specified by the parameter pair $(\Phi, \mathbf{w})$, such that 
    \begin{align} \label{strict stability condition}
        \sum_{V, W \in \mathbf{V}} \sum_{k=0}^p  |\phi_{V,W}(k)| & < 1.
    \end{align}
    Then its direct effect filter $\Dirtmp$ as well as the power-series of filters $\Dirtmp^\infty$ are entry-wise absolutely summable. Furthermore, for any pair of processes $X,Y \in \mathbf{V}$ and collection of processes $\mathbf{Z} \subset \mathbf{V}$ the CCF $\cetemp{X}{Y}{\mathbf{Z}}$ is absolutely summable. 
\end{proposition}

Before turning to its proof we introduce some notations and an auxillary lemma.
If $(\Phi, \mathbf{w})$ is a parameter pair specifying a SVAR model of order $p$ that is consistent with the process graph $G$, then $|\Phi|= (|\Phi(i)|)_{0\leq i \leq p}$ is the sequence of matrices that are the entry-wise absolute values of the matrices in $\Phi$. For any $V,W \in \mathbf{V}$ we denote by $\dirtmpabs{V}{W}$ resp. $\dirfrqabs{V}{W}$ the direct effect-filter resp. direct transfer function associated with the SVAR model specified by $(|\Phi|, \mathbf{w})$. Observe that 
\begin{equation} \label{eq supplement: sum of coefficients}
    \begin{split}
        \sum_{j \geq 0} \dirtmp{V}{W}(j) & =\dirfrq{V}{W}(0) \\
        &= \frac{\sum_{k=0}^p\phi_{V,W}(k)}{1- \sum_{k=0}^p \phi_{W,W}(k)},
    \end{split}
\end{equation}
which is well defined because of (\ref{condition: supplement stability II}). 

Accordingly, if $\pi$ is a directed path on the process graph $G$, then $|\Dirtmp|^{(\pi)}$ resp. $|\Dirfrq|^{(\pi)}$ represent the path-filter resp. path-function that is linked to the path $\pi$ and the SVAR model specified by $(|\Phi|, \mathbf{w})$. 

\begin{lemma} \label{lemma: supplement absolute summability}
    Let $\pi$ be a directed causal path on $G$, then we have that 
    \begin{align*}
        \sum_{j \geq 0 } |\Dirtmp^{(\pi)}(j)| & \leq \sum_{j \geq 0} |\Dirtmp|^{(\pi)}(j) \\
        & = |\Dirfrq|^{(\pi)}(0).
    \end{align*}
\end{lemma}
\begin{proof}
    We show this lemma by simple induction on the path-length $|\pi|$. We start with directed paths of length one. Let therefore $\pi = V \to W$ be a link on the process graph $G$. Then we have that $|\dirtmp{V}{W}(0)| = \dirtmpabs{V}{W}(0)$. Assuming that $|\dirtmp{V}{W}(l)| \leq \dirtmpabs{V}{W}(l)$ for $l \leq j-1$, it then follows for $j \leq p$ that 
    \begin{align*}
        |\dirtmp{V}{W}(j)| &= |\sum_{l=1}^j \dirtmp{V}{W}(j-l) \phi_{W,W}(l) + \phi_{V,W}(j)|\\
        &\leq \sum_{l=1}^j  \dirtmpabs{V}{W}(j-l) |\phi_{V,W} (l)| + |\phi_{V,W}(j)| \\
        &=\dirtmpabs{V}{W}(j).
    \end{align*}
    The case $j > p$ works analogously, so that we conclude with $|\dirtmp{V}{W}(j)| \leq \dirtmpabs{V}{W}(j)$ for all $j \in \Z$. 
    
    Suppose $\pi = (V\to W_1, \pi') $ is a path of length $n$ from $V$ to $W$, where $\pi'$ is a path of length $n-1$ from $W_1$ to $W$, then using the induction hypothesis we see that 
    \begin{align*}
        |\Dirtmp^{(\pi)}(j)| &= |\sum_{l= 0}^j \dirtmp{V}{W}(l) \Dirtmp^{(\pi')}(j-l)| \\
        & \leq \sum_{l=0}^j \dirtmpabs{V}{W}(l) |\Dirtmp|^{(\pi')}(j-l) \\
        &= |\Dirtmp|^{(\pi)}(j).
    \end{align*}
    This shows the claim. 
\end{proof}

\begin{proof}[Proof of Proposition \ref{prop: supplement sep existence condition}]
    Let $\Phi$ be an SVAR parameter such that (\ref{strict stability condition}) is satisfied. Then we compute 
    \begin{align*}
        \lVert \Dirfrq(z) \lVert_1 &= \sum_{V\neq W \in \mathbf{V}} \frac{|\sum_{k=0}^p \phi_{v,w}(k)z^k|}{|1- \sum_{j=1}^p \phi_{w,w}(k)z^k|}\\
        &\leq \sum_{V\neq W \in \mathbf{V}} \frac{\sum_{k=0}^p |\phi_{v,w}(k)z^k|}{1- \sum_{j=1}^p |\phi_{w,w}(k)z^k|} \\ 
        &= \lVert |\Dirfrq|(1) \lVert_1 \\
        & \leq (1- \alpha)^{-1} \sum_{v \neq w \in V} \sum_{k = 0}^p |\phi_{v,w}(k)|, 
    \end{align*}
    where 
    \begin{align*}
        \alpha &\coloneqq \max_{V \in \mathbf{V}} \sum_{j=1}^p |\phi_{v,v}(j)|
    \end{align*}
    Due to assumption (\ref{strict stability condition}) it holds that 
    \begin{align*}
        \alpha + \sum_{V\neq W \in \mathbf{V}} \sum_{k=0}^p |\phi_{v,w}(k)| < 1.
    \end{align*}
    From this we conclude that $\lVert \Dirfrq(z)\lVert_1 < \lVert |\Dirfrq|(z)\lVert_1 < 1$  for all $z\in \C$ such that $|z|\leq 1$. 

    This means in particular that $\Dirfrq^\infty(z)$ and also $|\Dirfrq|^\infty(z)$ are well defined for every $z\in S^1$ and furthermore
    \begin{align*}
        \Dirfrq^\infty(z) &= (I - \Dirfrq(z))^{-1} & |\Dirfrq|^\infty(z) &= (I - |\Dirfrq|(z))^{-1}
    \end{align*}
    So the entries in $\Dirfrq^\infty(z)$ are rational expressions in terms of the SVAR-parameters $\Phi$, and the filter $\Dirtmp^\infty$ exists and is absolutely summable, since 
    \begin{align*}
        \lVert \Lambda^\infty \lVert_1 < \lVert |\Lambda|^\infty \rVert_1 = \lVert |\Dirfrq|^\infty(1) \lVert_1 < \infty
    \end{align*}

    Now we show that any controlled causal effect is defined if $\Phi$ satisfies (\ref{strict stability condition}). Let $\mathbf{X}, \mathbf{Z}, \mathbf{Y}$ be mutually disjoint sets of variables. We are interested in determining the controlled causal effect $\cefrq{\mathbf{X}}{\mathbf{Y}}{\mathbf{Z}}(z)$. To compute this we introduce the auxillary matrix $\mathfrak{F}(z)\in \C^{\mathbf{V}\times \mathbf{V}}$ which is entry-wise defined as follows 
    \begin{align*}
        \mathfrak{F}_{V,W}(z) &= \begin{cases}
            \Dirfrq_{V,W}(z) & W \notin \mathbf{X}\cup \mathbf{Z} \\ 
            0 & W \in \mathbf{X}\cup \mathbf{Z}
        \end{cases}
    \end{align*}
    Furthermore, let $\mathfrak{F}^\infty = \sum_{k \geq 0} \mathfrak{F}^k(z)$ be the associated power sum of matrices. From the definition of controlled causal effects it follows that for any $X\in \mathbf{X}$ and $Y \in \mathbf{Y}$ the associated causal effect is computed as follows 
    \begin{align*}
        \cefrq{X}{Y}{\mathbf{Z}}(z) &=  \mathfrak{F}^\infty_{X,Y}(z).
    \end{align*}
    The right-hand side exists because $\lVert \mathfrak{F}(z)\lVert_1 \leq \lVert \Dirfrq(1) \rVert_1 < 1$, as shown above.

    The associated causal effect filter $\cetemp{\mathbf{X}}{\mathbf{Y}}{\mathbf{Z}}$ exists and is absolutely summable
    \begin{align*}
        \lVert \Dirtmp_{X\to Y | \Do(\mathbf{Z})=0} \rVert_1 \leq \lVert |\Dirtmp|_{X\to Y | \Do(\mathbf{Z})=0} \rVert_1  = |\mathfrak{F}|_{X,Y}^\infty(1) \leq   |\Dirfrq|_{X,Y}^\infty(1) < \infty,
    \end{align*}
    which follows by application of Lemma \ref{lemma: supplement absolute summability}. 
    
\end{proof}

\section{Proof of Proposition 2 in the main paper: the generalised path-rule}\label{section: proof generalised path rule}

\begin{proposition}[Proposition 1 in the main paper] 
    Let $X,Y, \mathbf{Z}$ be as in Definition 6 of the main paper and $\mathcal{G}'$ the time series subgraph of $\mathcal{G}$ over the nodes $\mathbf{V} \times  \Z$ such that its set of directed edges is 
    \begin{displaymath}
        \mathcal{D}' \coloneqq\{ V(t-k) \to W(t) \in \mathcal{D}: W \notin \mathbf{Z} \cup \{X\} \} \subset \mathcal{D}.
    \end{displaymath}
    This graph $\mathcal{G}' = (\mathbf{V}, \mathcal{D}')$ is the graph one obtains by deleting all edges in $\mathcal{G}$ which point to any of the nodes in $(X(s), \mathbf{Z}(s))_{s \in \Z}$.
    Further, we denote by $\mathcal{P}'(X(t-s), Y(t))$ the set of all directed paths on $\mathcal{G}'$ that start at $X(t-s)$ and end at $Y(t)$. Then, the $s$-th element of the CCF $\cetemp{X}{Y}{\mathbf{Z}}$ is given by the sum of the path-coefficients of the directed paths from $X(t-s)$ to $Y(s)$ on the subgraph $\mathcal{G}'$, i.e.,   
    \begin{align*}
        \cetemp{X}{Y}{\mathbf{Z}}(s) &=\sum_{\rho  \in \mathcal{P}'(X(t-s), Y(t))} \Phi^{(\rho)}. 
    \end{align*}
\end{proposition}
For the sake of proving Proposition 1 in the main paper, we define the map of sets
\begin{align}\label{eq: path projection}
    F: & \bigcup_{X,Y \in \mathbf{V}} \bigcup_{s\geq 0}  \mathcal{P}'(X(t-s), Y(t)) \to \bigcup_{X,Y\in \mathbf{V}}\mathrm{P}(X,Y) \times \Z,
\end{align}
where $\mathcal{P}'(X(t-s), Y(t))$ is the set of all paths of the form $\rho = X(t-s) \to W(t-s + r) \to \cdots \to Y(t)$ in the time series graph $\mathcal{G}$ such that $W \neq X$. Note that we can write $\rho$ as a concatenation of paths, written $\rho = \rho_1 \to \cdots \to \rho_m$, such that (i) the links in each $\rho_i$ are only auto-dependencies, that is, each $\rho_i$ is associated to a unique $X_i \in \mathbf{V}$, and (ii) for every $i$ the end-point of $\rho_{i}$ is different from the starting point of $\rho_{i+1}$, that is, $X_i \neq X_{i+1}$. This path-decomposition defines the mapping $F$, specifically,  $F(\rho) = (X_1 \to \cdots \to X_m, s)$.  

\begin{lemma}\label{lemma: path-rule}
    If $\pi$ is a directed path in the process graph $G$, then its path-filter can be expressed as a sum of path-coefficients in $\mathcal{G}$ as follows
    \begin{align*}
        \Lambda^{(\pi)}(s) &= \sum_{\rho \in F^{-1}(\pi, s)} \Phi^{(\rho)}.
    \end{align*}
\end{lemma}
\begin{proof}
    We prove the statement by induction and start with the case where $\pi$ has only one edge connecting two distinct nodes $V,W$. We need to show for every $s \geq 0$ that
    \begin{align*}
        \dirtmp{V}{W}(s) &= \sum_{\rho \in F^{-1}(V\to W, s)} \Phi^{(\rho)}.
    \end{align*}
    The equality holds if $s =0$. Let $s > 1$ and suppose we have proven the statement for all $0 \leq s' < s$. 
    The set of paths $F^{-1}(V \to W, s)$ is that of all paths on $\mathcal{G}$ of form $V(t-s) \to W(t-s+r_1) \to \cdots \to W(t-s+r_m) \to W(t)$, where $0 \leq r_i < r_{i+1}$, and $1 \leq i \leq m-1$. If such a path $\rho$ has more than one edge, then we denote by $\Delta(\rho) \coloneqq r_{m}$ the last time-point through which $\rho$ passes before it arrives at its target node $W(t)$. The set $F^{-1}(V \to W, s)$ can be partitioned in the set of paths which contain exactly one edge and the set of paths which contain more than one edge. The set of paths which contain more than one edge can be partitioned further into the following $s-1$ sets
    \begin{align*}
        F_k &= \{\rho \in F^{-1}(V \to W, s): \text{ }\Delta(\rho) = k \} \\
        &= \{ \rho' \to W(t): \text{ } \rho' \in F^{-1}(V \to W, s -k ) \}. 
    \end{align*}
    By induction we conclude that the sum of path-coefficients associated to the paths in $F_k$ satisfies the equality 
    \begin{align*}
        \sum_{\rho \in F_k} \Phi^{(\rho)} &= \phi_{W, W} (k) \dirtmp{V}{W}(s-k).
    \end{align*}
    Combining this equality with the partition of $F^{-1}(V \to W, s)$ gives 
    \begin{align*}
        \sum_{\rho \in F^{-1}(V \to W, s)} \Phi^{(\rho)} &= \phi_{V, W}(k) + \sum_{k=1}^{s-1} \phi_{V, W}(k) \dirtmp{V}{W}(s-k) \\
        &= \dirtmp{V}{W}(s),
    \end{align*}
    where for $k > p$ we set $\Phi_{V \to W}(k)$ and $\Phi_{V \to V}(k)$ to zero. This shows the claim for paths on the process graph $G$ which contain only one edge.
    
    To prove the general case, let $\pi = X \to \cdots \to W \to Y = \pi' \to Y$. Observe that $F^{-1}(\pi, s)$ can be considered as the following union of product sets
    \begin{align*}
         \bigcup_{0 \leq r \leq s} F^{-1}(\pi', r) \times F^{-1}(W\to Y, s-r).
    \end{align*} Accordingly, 
    \begin{align*}
        \sum_{\rho \in F^{-1}(\pi, s)} \Phi^{(\rho)} &= \sum_{0 \leq r \leq s} (\sum_{\rho'F^{-1}(\pi', r)} \Phi^{(\rho')}) (\sum_{\rho'' \in F^{-1}(W \to Y, s-r)} \Phi^{(\rho'')}) \\
        &= \sum_{0 \leq r \leq s} \Dirtmp^{(\pi')}(r) \dirtmp{W}{Y}(s-r)\\
        &= \Dirtmp^{(\pi)}.
    \end{align*}
\end{proof}
\begin{proof}[Proof of Proposition 1 in the main paper]
    Let us denote by $\mathcal{G}'$ the subgraph of $\mathcal{G}$ that contains only those edges that do not point to any $W(t)$, where $W \in \mathbf{Z} \cup \{X\}$, and by $\mathcal{P}_{\mathcal{G'}}(X(t-s), Y(t))$ the set of directed paths from $X(t-s)$ to $Y(t)$ on $\mathcal{G}'$, then
    \begin{align*}
        \cetemp{X}{Y}{\mathbf{Z}}(s) &= \sum_{\pi \in \mathrm{P}_{\mathbf{Z} \cup \{X\}}(X,Y)} \Dirtmp^{(\pi)}\\
        &= \sum_{\pi \in \mathrm{P}_{\mathbf{Z} \cup \{X\}}(X,Y)} \sum_{\rho \in F^{-1}(\pi, s)} \Phi^{(\rho)}\\
        &=\sum_{\rho \in \mathcal{P}_{\mathcal{G}'}(X(t-s), Y(t))}\Phi^{(\rho)}.
     \end{align*}
     The last equality follows, since for every $\rho \in \mathcal{P}_{\mathcal{G}'}(X(t-s), Y(t))$ its projection $F(\rho)$ is a path in $\mathrm{P}_{\mathbf{Z} \cup \{X\}}(X,Y)$. If on the other hand $\pi \in \mathrm{P}_{\mathbf{Z} \cup \{X\}}(X,Y)$, then $F^{-1}(\pi) \subset \mathcal{P}_{\mathcal{G}'}(X(t-s), Y(t))$. 
\end{proof}

\section{Parameterisation of the process graphs in Figure 7 of the main paper}\label{section: paramterisations} 
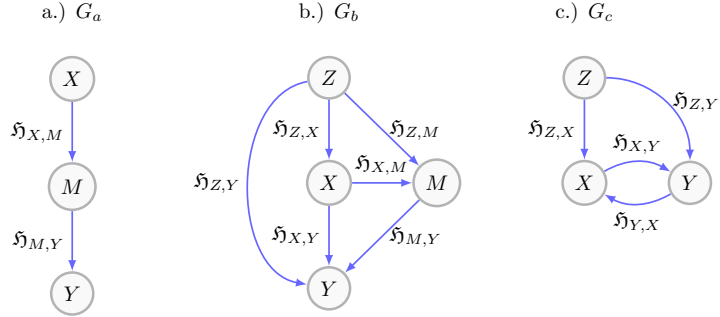
\begin{figure}
    \centering
    \resizebox{0.8\textwidth}{!}{
    \begin{tikzpicture}[
       ts_node/.style={circle, draw=gray, fill=gray!60, thick}, 
       header/.style={rectangle, minimum size=1cm, thick},
       snode/.style={circle, draw=gray!60, fill=gray!5, very thick}
    ]

    \node[header] (a) at (0,0) {a.) $G_a$};
    \node[snode] (X_a) [below=0.2cm of a] {$X$};
    \node[snode] (M_a) [below=of X_a] {$M$};
    \node[snode] (Y_a) [below=of M_a] {$Y$};  

    \node[header] (b) [right=3.cm of a] {b.) $G_b$};
    \node[snode] (Z_b) [below=0.2cm of b] {$Z$};
    \node[snode] (X_b) [below= of Z_b] {$X$};
    \node[snode] (Y_b) [below=of X_b] {$Y$};
    \node[snode] (M_b) [right=of X_b] {$M$};

    \node[header] (c) [right=3.cm of b] {c.) $G_c$};
    \node[snode] (Z_c) [below=0.2cm of c] {$Z$};
    \node[snode] (X_c) [below=of Z_c] {$X$};
    \node[snode] (Y_c) [right=of X_c] {$Y$};

    \draw[->, draw=blue!60, thick] (X_a) -- node[left] {$\dirfrq{X}{M}$} (M_a);
    \draw[->, draw=blue!60, thick] (M_a) -- node[left] {$\dirfrq{M}{Y}$} (Y_a);

    \draw[->, draw=blue!60, thick] (Z_b) -- node[left] {$\dirfrq{Z}{X}$} (X_b);
    \draw[->, draw=blue!60, thick] (Z_b) -- node[right] {$\dirfrq{Z}{M}$} (M_b);
    \draw[->, draw=blue!60, thick] (Z_b) to [out=190, in=170] node[left] {$\dirfrq{Z}{Y}$} (Y_b);
    \draw[->, draw=blue!60, thick] (X_b) -- node[left] {$\dirfrq{X}{Y}$} (Y_b);
    \draw[->, draw=blue!60, thick] (X_b) -- node[above] {$\dirfrq{X}{M}$} (M_b);
    \draw[->, draw=blue!60, thick] (M_b) -- node[right] {$\dirfrq{M}{Y}$} (Y_b);

    \draw[->, draw=blue!60, thick] (Z_c) -- node[left] {$\dirfrq{Z}{X}$} (X_c);
    \draw[->, draw=blue!60, thick] (X_c) to [out=30, in=150] node[above] {$\dirfrq{X}{Y}$} (Y_c);
    \draw[->, draw=blue!60, thick] (Y_c) to [out=210, in = 330] node[below] {$\dirfrq{Y}{X}$} (X_c);
    \draw[->, draw=blue!60, thick] (Z_c) to [out=0, in=90] node[right] {$\dirfrq{Z}{Y}$} (Y_c);
    \end{tikzpicture}}
    \caption{This figure shows three example process graphs.}
    \label{fig: supplement summary_graphs_demo}
\end{figure}
The process graphs in Figure \ref{fig: supplement summary_graphs_demo} have the following underlying SAVR processes.

\textbf{Process graph a.)} For this graph $\mathcal{G}_S^a$ we choose the following underlying VAR-model
\begin{align*}
    \X_X(t) &= 0.7X(t-1) + \eta_X(t)\\
    \X_M(t) &= -0.5M(t-2) + 0.3M(t-1) + 0.3X(t-1) + \eta_M(t) \\
    \X_Y(t) &= 0.7Y(t-1)+ 0.3M(t-1)+  \eta_Y(t).
\end{align*}
These equations specify the link-transfer-functions. They are as follows
\begin{align*}
    \dirfrq{X}{M}(\omega) &= \frac{0.3 \exp(-i\omega)}{1 - 0.3\exp(-i\omega) + 0.5 \exp(-i2\omega)} \\
    \dirfrq{M}{Y} (\omega) &= \frac{0.3 \exp(-i\omega)}{1- 0.7 \exp(-i\omega)}.
\end{align*}
The spectral density of the internal dynamics is specified by its diagonal entries, as this model does not include latent confounders 

\textbf{Process graph b.)} For the process graph $\mathcal{G}_S^b$ we choose the following underlying VAR-model
\begin{align*}
    \X_Z(t) &= -0.5Z(t-5) + 0.3Z(t-1) + \eta_{Z}(t) \\
    \X_X(t) &= -0.4X(t-4) - 0.2 X(t-3) + 0.3X(t-1) + 0.2 Z(t-1) \eta_{X}(t) \\
    \X_M(t) &= 0.5M(t-1) + 0.3 X(t-2) + 0.2 Z(t-1) + \eta_{M}(t) \\
    \X_Y(t) &= -0.5Y(t-6) + 0.3Y(t-1) + 0.3X(t-1) + 0.2M(t-2) + \eta_{Y}(t),
\end{align*}
from which we get the following link-transfer-functions
\begin{align*}
    \dirfrq{Z}{X}(\omega) &= \frac{0.2 \exp(-i \omega)}{1- 0.3\exp(-i\omega) + 0.2 \exp(-i 3\omega ) + 0.4 \exp(-i 4 \omega)} \\
    \dirfrq{Z}{M}(\omega) &= \frac{0.2 \exp(- i \omega)}{1- 0.5 \exp(-i \omega)} \\
    \dirfrq{X}{M}(\omega) &= \frac{0.2 \exp(-i 2 \omega)}{1- 0.5 \exp(-i \omega)} \\
    \dirfrq{X}{Y}(\omega) &= \frac{0.3 \exp(- i \omega)}{1- 0.3 \exp(-i \omega ) + 0.5 \exp(-i 6 \omega)}\\
    \dirfrq{M}{Y}(\omega) &= \frac{0.2 \exp(- i\omega)}{1-0.3 \exp(-i\omega) + 0.5 \exp(-i 6 \omega)}.
\end{align*}
\textbf{Process graph c.)} For the process graph $\mathcal{G}^c_S$ we choose the following underlying VAR-model
\begin{align*}
    \X_Z(t) &= 0.5 Z(t-1) + \eta_Z(t) \\
    \X_X(t) &= 0.3X(t-1) - 0.5X(t-2) + 0.3Z(t-1) + 0.3Y(t-2) + \eta_X(t) \\
    \X_Y(t) &= 0.3Y(t-1) -0.3Y(t-3) + 0.3X(t-2) + 0.2Z(t-3) \eta_Y(t).  
\end{align*}
The associated link-transfer-functions are as follows
\begin{align*}
    \dirfrq{Z}{X} (\omega) &= \frac{0.3 \exp(- i \omega )}{1 - 0.3\exp(-i\omega ) + 0.5 \exp(-i2\omega)} \\
    \dirfrq{Z}{Y} (\omega) &= \frac{0.2 \exp(-i 3 \omega)}{1 - 0.3 \exp(-i \omega) + 0.3 \exp(-i 3 \omega)} \\
    \dirfrq{X}{Y} (\omega) &= \frac{0.3 \exp(-i2\omega)}{1 - 0.3 \exp(-i \omega) + 0.3 \exp(-i 3 \omega)} \\
    \dirfrq{Y}{X} (\omega) &= \frac{0.3 \exp(-i \omega )}{1 - 0.3\exp(-i\omega ) + 0.5 \exp(-i2\omega)}.
\end{align*}

\section{Spectral density decomposition}\label{section: spectrum decomposition}

\subsection{Computation of the confounding factor in the spectral density decomposition in section 5.3 of the main paper}
Recall that in the main paper we derived a decomposition of the spectral density of a process $W$ of the form 
\begin{align*}
    \Sp{W} &= \spcausal{V}{W} + \spconf{V}{W} + \spresidual{V}{W},
\end{align*}
where the second factor indicates the possible confounding of the causal effect $\ucetemp{V}{W}$ and is defined as the sum over all trek-monomial-functions associated with the treks in $\T{V \leftrightarrow W} \coloneqq \T{V}(W,W) \setminus \T{V\to W}$. To illustrate how this measures the degree of confounding, we inspect the cross-spectrum $\Sp{W, V}$, which is the sum of trek-monomial-functions associated with the treks in $\T{}(W,V)$. This set can be partitioned into the set of treks $\pi$ for which $V \in \trekL(\pi) \cap \trekR(\pi)$, represented by $\T{L \cap R}(W,V)$, and those treks $\pi$ for which $V \notin \trekR(\pi)$, labelled $\T{R \setminus L}(W,V)$, so that we can decompose   
\begin{equation} \label{eq: cross-spectrum confounding}
    \begin{split}
        \Sp{W,V} &= \sum_{\pi_{L \cap R} \in \T{L \cap R}(W,V)} \Sp{}^{(\pi_{L\cap R})} +  \sum_{\pi_{L \setminus R} \in \T{L \setminus R}(W,V)} \Sp{}^{(\pi_{L\setminus R})} \\
    &= \ucefrq{V}{W} \Sp{V} + \sum_{\pi_{R \setminus L} \in \T{R \setminus L}(W,V)} \Sp{}^{(\pi_{R \setminus L})}, 
    \end{split}
\end{equation}
where the last equality holds as any trek $\pi_{L \cap R}\in \T{L\cap R}(W,V)$ can be uniquely decomposed as a concatenation of the form $\pi_{L \cap R} = \pi_1 \leftarrow \pi_2$, where $\pi_1 \in \mathrm{P}_V(V,W)$ and $\pi_2 \in \T{}(V,V)$. This decomposition of $\Sp{W,V}$ shows that $\ucefrq{V}{W} = \frac{\Sp{W,V}}{\Sp{V}}$ whenever the second summand in (\ref{eq: cross-spectrum confounding}) is zero, e.g. if $\T{R \setminus L}(W,V)=\emptyset$. In this case, the causal effect of $V$ on $W$ is said to be unconfounded so that it can be retrieved directly from the spectral density matrix. With the help of the trek-rule, we decompose the confounding contribution to $\Sp{W}$ as follows 
\begin{equation}\label{eq:supplement spectral confounding}
    \begin{split}
        \spconf{V}{W} &\coloneqq \sum_{\pi \in \T{V\leftrightarrow W}} \Sp{}^{(\pi)} \\
        &=(\sum_{\pi_{R \setminus L} \in \T{R \setminus L}(W,V)} \Sp{}^{(\pi_{R \setminus L})}) \ucefrq{V}{W}^\ast + \ucefrq{V}{W}(\sum_{\pi_{L \setminus R} \in \T{L \setminus R}(V,W)} \Sp{}^{(\pi_{R \setminus L})}) \\
        &= \Sp{W,V} \ucefrq{V}{W}^\ast + \ucefrq{V}{W}\Sp{V,W} - 2|\ucefrq{V}{W}|^2\Sp{V} \\
        &= 2\mathrm{Re}(\ucefrq{V}{W}\Sp{V,W} - |\ucefrq{V}{W}|^2\Sp{V})
    \end{split}
\end{equation}
where $\T{L \setminus R}(V,W)$ is the set consisting of those treks $\pi$ from $V$ to $W$ for which $V$ lies not in $ \trekR(\pi)$. The first equality follows as for every trek $\pi \in \T{V \leftrightarrow W}$ either $V \in \trekL(\pi)$ or $V \in \trekR(\pi)$ (but not both). Suppose $\pi \in \T{V \leftrightarrow W}$ and $V \in \trekR(\pi)$, then $\pi$ can be uniquely decomposed as a concatenation of the form $\pi_1\to\pi_2$ where $\pi_1 \in \T{L \setminus R}(W,V)$ and $\pi_2 \in \mathrm{P}_V(V,W)$. Hence, the sum over all trek-monomial-functions associated with treks $\pi\in \T{V \leftrightarrow W}$ such that $V \in \trekL(\pi)$ equals 
\begin{align*}
    (\sum_{\pi_1 \in \T{L \setminus R}(V,W)} \Dirfrq^{(\pi_1)}) (\sum_{\pi_2 \in \mathrm{P}_V(V,W)} \Dirfrq^{(\pi_2)}) &= \Sp{W,V}\ucefrq{V}{W}^\ast - |\ucefrq{V}{W} |^2\Sp{V}.
\end{align*}
This follows from the definition of $\ucefrq{V}{W}$ and equation (\ref{eq: cross-spectrum confounding}). Analgously one shows that the sum over all trek-monomial functions that are associated with treks $\pi \in \T{V \leftrightarrow W}$ such that $V \in \trekR(\pi)$ is equal to 
\begin{align*}
    \ucefrq{V}{W}\Sp{V,W} - |\ucefrq{V}{W}|^2 \Sp{V}.
\end{align*}
This implies the equations (\ref{eq:supplement spectral confounding}). 
\subsection{Details on the computation of the spectral density decomposition in Example 6 of the main paper}

We start by determining the causal contribution $\spcausal{X}{Y}$. Recall that we have to find all the treks $\pi$ such that $X$ appears both in its left-hand side, i.e. $X \in \trekL(\pi)$, and its right-hand side, i.e. $X \in \trekR(\pi)$. As the process graph contains a cycle involving both $X$ and $Y$ there are infinitely many treks along which $X$ causally contributes to the variability of $Y$. We partition the set of treks into five sets $\mathcal{T}_{X\to Y}^i$, where $i=1, \dots, 5$, and compute for each set its associated trek-monomial-function-series. We begin with
\begin{align*}
    \mathcal{T}_{X \to Y}^1 &= \{ X \leftrightarrows^{\times k} Y \leftarrow X \to Y \rightleftarrows^{\times l} X \}_{l,k \geq 0}.
\end{align*}
The elements in this set shall be understood as follows: For each pair of non-negative integers $(k,l)$ we get a trek. On its left-hand side we first go from $X$ to $Y$ and then we traverse $k$-times through the cycle $Y \to X \to Y$. The right-hand side of the trek is defined analogously. In what follows we abbreviate the product of transfer-functions, $\dirfrq{Y}{X}\dirfrq{X}{Y}$, by $\circfrq{Y}{X}$. Summing the trek-monomial-functions associated to the treks in $\T{X \to Y}^1$ yields  
\begin{align*}
    \sum_{\pi \in \T{X \to Y}^1} \Sp{}^{(\pi)} &= \dirfrq{X}{Y} \sum_{k \geq 0} \circfrq{Y}{X}^k \Sp{X}^\internal (\dirfrq{X}{Y} \sum_{l \geq 0} \circfrq{Y}{X}^l)^\ast \\
    &= \frac{\dirfrq{X}{Y}}{1 - \circfrq{X}{Y}} \Sp{X}^\internal \frac{\dirfrq{X}{Y}^\ast}{1 - \circfrq{X}{Y}^\ast}\\
    &= \frac{|\dirfrq{X}{Y}|^2}{|1 - \circfrq{X}{Y}|^2} \Sp{X}^\internal,
\end{align*}
see the blue line in Figure 10.a) of the main paper for an illustration. This expression is well defined as long as $|\circfrq{Y}{X}(\omega)| < 1$, for all $\omega \in [0, 2\pi)$. 

The second subset of $\T{X\to Y}$ consists of the treks
\begin{align*}
    \T{X \to Y}^2 &= \{ X \leftrightarrows^{\times k} Y \leftarrow X \leftarrow Z \to X \to Y \rightleftarrows^{\times l} X \}_{k,l \geq 0},
\end{align*}
so that its associated trek-monomial-function-series takes the form  
\begin{align*}
    \sum_{\pi \in \T{X\to Y}^2} \Sp{}^{(\pi)}
    &= \frac{|\dirfrq{Z}{X}\dirfrq{X}{Y}|^2}{|1- \circfrq{Y}{X}|^2} \Sp{Z}^\internal.
\end{align*}

The third subset contains the following treks
\begin{align*}
    \mathcal{T}^3_{X \to Y} &= \{ X \leftrightarrows^{\times k} Y \leftarrow X \leftarrow Y \leftarrow Z \to Y \to X \to Y \rightleftarrows^{\times l} X \}_{k,l \geq 0},
\end{align*}
so that its associated trek-monomial-function-series admits the expression
\begin{align*}
    \sum_{\pi \in \T{X \to Y}^3} \Sp{}^{(\pi)}&= \frac{|\dirfrq{Z}{Y} \circfrq{Y}{X}|^2}{|1 - \circfrq{Y}{X}|^2} \Sp{Z}^\internal.
\end{align*}

The fourth set is composed of the following treks
\begin{align*}
    \T{X \to Y}^4 &= \{ 
    X \leftrightarrows^{\times k} Y \leftarrow X \leftarrow Z \to Y \to X \to Y \rightleftarrows^{\times l} X \}_{k,l \geq 0} \\
    & \cup \{ X \leftrightarrows^{\times m} Y \leftarrow X \leftarrow Y \leftarrow Z \to X \to Y \rightleftarrows^{\times n} X \}_{m,n \geq 0},
\end{align*}
so that its associated trek-monomial-function-series is
\begin{align*}
    \sum_{\pi \in \T{X \to Y}^4} \Sp{}^{(\pi)}  &= 2\mathrm{Re} (\frac{\dirfrq{Z}{X} \dirfrq{X}{Y} \dirfrq{Z}{Y}^\ast \circfrq{Y}{X}^\ast}{|1 - \circfrq{Y}{X}|^2}) \Sp{Z}^\internal.
\end{align*}
The sum of the trek-monomial-function series associated to $\T{2}, \T{3}, \T{4}$ expresses the causal contribution of $\X_Z^\internal$ to $\Sp{Y}$ which got mediated by $X$. This contribution is the yellow line in Figure 10.a) in the main paper. 

Finally, the fifth set is
\begin{align*}
    \T{X\to Y}^5 &= \{ X \leftrightarrows^{\times l} Y \leftarrow X \leftarrow Y \to X \to Y \rightleftarrows^{\times l} X \}_{k,l \geq 0},
\end{align*}
so that its associated trek power-series is
\begin{align*}
    \sum_{\pi \in \T{X \to Y}} \Sp{}^{(\pi)} &= \frac{|\circfrq{Y}{X}|^2}{|1- \circfrq{Y}{X}|^2} \Sp{Y}^\internal, 
\end{align*}
this function is the salmon-colored line Figure 10.a) in the main paper. 
We can now conclude with an analytical description of the causal contribution of $X$ to $Y$, i.e. $\spcausal{X}{Y}$ is the sum over all the five trek-monomial-function-series we have computed above. 

\textbf{The confounding $\spconf{X}{Y}$.} Note that in the process graph $S(\mathcal{G}_c)$ both $Z$ and $Y$ confound the causal relation between $X$ and $Y$. Hence, the confounding contribution is given in terms of the internal dynamics $\X_Z^\internal$ and $\X_Y^\internal$. The contribution of $Z$ to the confounding is captured by the following set of treks
\begin{align*}
    \T{X\leftrightarrow Y}^1 &=
    \{ X \leftrightarrows^{\times l } Y \leftarrow X  \leftarrow Z \to Y \}_{l \geq 0} \\
    &\cup \{  Y \leftarrow Z \to X \to Y \rightleftarrows^{\times k} X \}_{k \geq 0}, \\
    \T{X \leftrightarrow Y}^2 &= 
    \{ X \leftrightarrows^{\times l } Y \leftarrow X \leftarrow Y \leftarrow Z \to Y \}_{l \geq 0} \\
    & \cup \{ Y \leftarrow Z \to Y \to X \to Y \rightleftarrows^{\times k} Y \}_{l \geq 0}. 
\end{align*}
The sum of over all trek-monomial-functions associated to the treks in $\T{X\leftrightarrow Y}^1 \cup \T{X\leftrightarrow Y}^2$ is 
\begin{align*}
    \sum_{\pi \in \T{X \leftrightarrow Y}^1} \Sp{}^{(\pi)} + \sum_{\pi \in \mathcal{T}_{X \leftrightarrow Y}^2} \Sp{}^{(\pi)}
    &= 2 \mathrm{Re}(\frac{\dirfrq{Z}{X} \dirfrq{X}{Y} \dirfrq{Z}{Y}^\ast}{1 - \circfrq{Y}{X}} + \frac{\dirfrq{Z}{Y}\circfrq{Y}{X} \dirfrq{Z}{Y}^\ast}{1-\circfrq{Y}{X}}) \Sp{Z}^\internal,
\end{align*}
this expression is the yellow line in Figure 10.b.) in the main paper.

The contribution of $Y$ to $\spconf{X}{Y}$ is due to the following set of treks 
\begin{align*}
    \T{X \leftrightarrow Y}^3 & = \{ Y \rightleftarrows^{\times k} X \}_{k \geq 0} \\
    & \cup \{ X \leftrightarrows^{\times l} Y \}_{l \geq 0}.
\end{align*}
Its associated trek function series is 
\begin{align*}
    \sum_{\pi \in \T{X\leftrightarrow Y}^3} \Sp{}^{(\pi)} &= 2\mathrm{Re}(\frac{1}{1 - \circfrq{Y}{X}}) \Sp{Y}^\internal, 
\end{align*}
in Figure 10.b) of the main paper this expression is the salmon-colored line. 

\textbf{The residual $\spresidual{X}{Y}$.} Finally, the variability which is independent of the presence of $X$ is
\begin{align*}
    \spresidual{X}{Y} &= |\dirfrq{Z}{Y}|^2 \Sp{Z}^\internal + \Sp{Y}^\internal,
\end{align*}
see Figure 10c.) in the main paper. 

\end{document}